\documentclass[smallcondensed]{svjour3}

\usepackage{mathtools}
\usepackage{algorithm}
\usepackage{algorithmicx}
\usepackage{algpseudocode}
\usepackage[pdftex]{graphicx}                 %%% graphics for pdfLaTeX
\DeclareGraphicsExtensions{.pdf,.eps}              %%% standard extension for included graphics
\usepackage[pdftex]{thumbpdf}                 %%% thumbnails for pdflatex
\graphicspath {
    {./}
    {./figures/}
}
\usepackage{amsmath}
\usepackage{amssymb}
\usepackage{amsfonts}
\usepackage{mathrsfs}
\usepackage{color}
\newcommand{\mR}{{\mathbb{R}}}

\newcommand{\ignore}[1]{}

\newcommand{\Ad}{\mathbf{Ad}}

\newcommand{\cB}{{\mathcal B}}
\newcommand{\cL}{{\mathcal L}}

\newcommand{\cH}{\mathcal H}

\newcommand{\bR}{{\boldsymbol R}}
\newcommand{\bA}{{\boldsymbol A}}
\newcommand{\bL}{{\boldsymbol L}}
\newcommand{\bT}{{\boldsymbol T}}
\newcommand{\bK}{{\boldsymbol K}}
\newcommand{\Diff}{\mathrm{Diff}}
\newcommand{\id}{\mathrm{id}}

\newcommand{\scp}[2]{{\left\langle {#1}\, , \, {#2}\right\rangle}}
\newcommand{\lform}[2]{{\left( {#1}\, \left\vert \vphantom{#1}\,  {#2}\right.\right)}}
\graphicspath {
    {./}
    {./figures/}
}

\begin{document}
 
\title{Metamorphosis of Images in Reproducing Kernel Hilbert Spaces}
\author{Casey L. Richardson\and Laurent Younes}
 \institute{C.L. Richardson, Johns Hopkins University,
Applied Physics Laboratory,
11100 Johns Hopkins Road,
Laurel, Maryland, 20723-6099 USA\\
\email{Casey.Richardson@jhuapl.edu}\\
 \and
L. Younes, Center for Imaging Science and Department of Applied Mathematics and Statistics, Johns Hopkins University, 3400 North Charles Street, Baltimore, MD 21218-2686, USA\\
\email{Laurent.Younes@jhu.edu}
}

\thanks{This work was completed while C.L. Richardson was a member of the Center for Imaging Science at JHU, and was partially supported by the National Science Foundation under grant number DMS-1016038.}

\maketitle

\begin{abstract}
\noindent
Metamorphosis is a method for diffeomorphic matching of shapes, with many potential applications for anatomical shape comparison in medical imagery, a problem which is 
central to the field of computational anatomy.  An important tool for the practical application of metamorphosis is a numerical method based on shooting from the initial momentum,
as this would enable the use of statistical methods based on this momentum, as well as the estimation of templates from hyper-templates using morphing.  In this paper we  introduce a shooting method, in the particular case of morphing images that lie in a reproducing kernel Hilbert space (RKHS).  We derive the relevant shooting equations 
from a Lagrangian frame of reference, present the details of the numerical approach, and illustrate the method through morphing of some simple images.
\end{abstract}

\section{Introduction}

Metamorphosis is a  pattern matching framework that combines  diffeomorphic  mapping  with variations in shape or image space;
it has potential for interesting applications in shape analysis and computational anatomy \cite{hipp_volume_2003,tangent_space_rep_2004,shape_dat_2007}.  
One of its advantages is to  allow for transgression of the diffeomorphic constraint, inducing changes in topology between the template and the target image, enabling an exact matching
between template and target, through the minimization of a geodesic cost associated to a Riemannian metric on the product space of shapes and deformations. For images, this is accomplished by allowing both deformations of the template as 
well as smooth changes in the template's intensity values.  Through this combination of changes, the template is morphed into the target (see \cite{ty_2005,hty_2008,younes_book_2010}
for a precise description, and Section \ref{section.setup} for more details).

In this paper, we  generalize previously known results for image metamorphosis, and introduce a new shooting method for computing minimizers of the image metamorphosis matching functional, 
in the case where the images have some degree of smoothness (they are elements of a certain reproducing kernel Hilbert space).
Our work builds upon \cite{hty_2008}, which introduced a general  formulation of metamorphosis
using the Euler-Poincar\'e framework, and then derived the continuous-time evolution equations for metamorphosis (EPMorph) in 
several concrete situations, such as image matching, density matching, and measure matching.  This paper also suggested 
extensions of its analysis and numerics for further work, e.g. the numerics for morphing of discrete measures which
was analyzed by the authors of this paper in \cite{ry_2013}.  In Section 11.2 of \cite{hty_2008}, Holm et al. 
apply metamorphosis to the case of images that are members of a reproducing kernel Hilbert space (RKHS),
and then they propose the development of numerical methods for the EPMorph equations in this context.
In this paper, we develop this idea into a shooting method for morphing RKHS images, by deriving the appropriate 
forward and adjoint equations, and then we present some numerical experiments that illustrate the use of such a method for 
simple examples of shape matching. We also complete the theoretical analysis of these methods, in a framework that covers a large range of applications. 

The first part of the paper provides a formal presentation of the approach, leaving the detailed discussion of the hypotheses and rigorous proofs to the second part, constituted by section \ref{sec:proofs}. The basic notation and assumptions are presented in section \ref{section.setup} together with the metamorphosis variational problem and associated optimality equations. Section \ref{sec:singular} describes a family of singular solutions that satisfy the optimality equations, providing a key component of the proposed numerical procedure. These singular solutions are then reinterpreted in section \ref{sec:relaxed} as the solutions that arise from a relaxation of the original problem replacing the infinite-dimensional boundary conditions in image space with a finite number of constraints. The numerical solution of the relaxed problem is then described in section \ref{sec:solution}, with complements given in the appendix. Section \ref{sec:experiments} then provides experimental results.

\section{Mathematical Setup}
\label{section.setup}
Reproducing kernel Hilbert spaces will be key elements in our construction.
If $X$ is a Banach or Hilbert space, we will denote by $\lform{\mu}{h}$ the pairing between a linear form $\mu\in X^*$ and a vector $h\in X$; the inner product in a Hilbert space $X$ will be denoted
by $\scp{h}{k}_X$, $h,k\in X$. In the Hilbert case, we will denote by $\bK_X$ the isometry map between $X^*$ and  $X$, such that $\lform{\mu}{h} = \scp{\bK_X\mu}{h}_X$, and by $\bA_X$ its inverse, $\bA_X = \bK_X^{-1}$. If $X$ and $Y$ are Banach and $A: X\to Y$ a bounded operator, we let $A^*: Y^* \to X^*$ be the adjoint, defined by $\lform{A^*\mu}{h} = \lform{\mu}{Ah}$. If $X=Y$ are Hilbert, we let $A^T$ be the transpose, defined by $\scp{A^T h}{\tilde h}_H = \scp{h}{A\tilde h}$, or $A^T = \bK_XA^* \bA_X$.  We will also denote by $A^T$ the transpose matrix of a finite-dimensional operator. Finally, if $X,Y$ are two Banach spaces $\cL(X,Y)$ denotes the set of bounded linear operators from $X$ to $Y$, and the operator norm is denoted $\|\cdot\|_{\cL(X,Y)}$. If $Y=X$, we will use $\cL(X)$ instead of $\cL(X,Y)$.

A Hilbert space $X$ continuously embedded in $L^2(\mR^d, \mR^k)$ is a reproducing kernel Hilbert space (RKHS) if, for all $x\in \mR^d$, the Dirac measure $\delta_x: X \to \mR^k$, defined by $\delta_x(h) = h(x)$ is  a bounded linear map. If $X$ is an RKHS, and given $a\in \mR^k$, we will denote by  $a\cdot \delta_x$ the continuous linear form $\lform{a\cdot \delta_x}{h} = a\cdot h(x)$, where the latter denotes the usual dot product in $\mR^k$. The kernel of $X$ is then the matrix-valued function $(x,y) \mapsto K_X(x,y)$ defined by
\[
K_X(x,y)a = \bK_X(a\cdot \delta_y)(x).
\]
($K_X(x,y)$ is a $k$ by $k$ matrix, and $k$ will be either $d$ or 1 in the following discussion.)\\

Metamorphosis is a diffeomorphic registration framework: it is 
formulated using a certain subgroup of diffeomorphisms of $\mathbb{R}^d$ acting, as a left group action,
on images (see \cite{miller_younes_1,camionyounes_2001,ty_2005,hty_2008,ry_2013} for more general classes of metamorphoses). 
This group, denoted $\Diff_V$, is the set of all diffeomorphisms of $\mR^d$ that can be attained as flows of time-dependent vector fields $v \in L^2([0,1]; V)$,
where $V$ is a reproducing kernel Hilbert space continuously embedded in $\cB^p := C_0^p(\mathbb{R}^d; \mathbb{R}^d)$ for 
some $p \geq 1$ (the space 
of $C^p$ vector fields that decay to zero at infinity). More precisely, $\psi\in G$ if and only if $\psi = \varphi(1)$, where $\varphi$ is the solution of 
\[
\begin{array}{l}
 \dot{\varphi}(t) = v(t) \circ \varphi(t), \\ 
  \varphi(0) = \mbox{id}
\end{array}
\]
for some $v$ satisfying $\int_0^1 \|v(t)\|_V^2 dt < \infty$. The group $\Diff_V$ is then embedded in the space $\Diff^p$ of diffeomorphisms $\psi$ such that $\psi-\id$ and $\psi^{-1}-\id$ both belong to $\cB^p$, which forms an open subset of the  affine space $\mbox{id} + \cB^p$.\\

In most this paper, the image space is a scalar RKHS, denoted $H$ (we will weaken this assumption in some of the results of section \ref{sec:proofs}). To simplify the discussion, we will assume that $H$ is equivalent to a Sobolev space $\cH^r(\mR^d)$ (the space of functions with square integrable partial derivatives up to order $r$) for some $r > d/2+1$, so that elements of $H$ are differentiable. Assuming that $p\geq r$, we will consider the the action of $C^p$ diffeomorphisms  $H$ given by $\varphi\cdot q = q\circ \varphi^{-1}$ .\\

In order to connect two images $q^{(0)}$ and $q^{(1)}$ in $H$ with a continuous path $q(t)$, image metamorphosis solves the optimal control problem
\begin{multline}
\label{eq:meta.cost}
\frac12 \int_0^1 \|v(t)\|_V^2 dt + \frac1{2\sigma^2} \int_0^1 \|\zeta(t)\|^2_H dt \longrightarrow \min\\
\text{subject to } \dot q(t) = \nabla q(t)\cdot v(t) + \zeta(t),\ q(0) = q^{(0)} \text{ and } q(1) = q^{(1)}. 
\end{multline}• 

We will prove in section \ref{sec:proofs} that, under some additional conditions, solutions of this problem exist and satisfy a Pontryagin maximum principle (PMP) that we derive formally  here.  Introduce the control-dependent Hamiltonian
\[
H(p,q,v, \zeta) = \lform{p}{\nabla q\cdot v + \zeta} - \frac12 \|v\|_V^2 - \frac1{2\sigma^2} \|\zeta\|_H^2.
\]
The PMP \cite{vincent1999nonlinear,agrachev2004control} states that optimal solutions of \eqref{eq:meta.cost} satisfy
\begin{equation*}
\begin{cases}
\dot q(t) = \partial_p H\\
\dot p(t) = -\partial_q H\\
(v, \zeta) = \mathrm{argmax}\, H(p, q, \cdot, \cdot)
\end{cases}
\end{equation*}
yielding
\begin{equation}
\label{eq:pmp}
\begin{cases}
\dot q(t) = \nabla q(t)\cdot v(t) + \zeta(t)\\
\dot p(t) + \nabla\cdot(p(t) v(t)) = 0\\
\zeta(t) = \sigma^2	\bK_H p(t)\\
v(t) = -\bK_V(\nabla q(t)\cdot p(t))
\end{cases}
\end{equation}
\\

We will use the following reformulation of 
problem \eqref{eq:meta.cost}. The evolution equation for $q$ is an advection and is equivalent to 
\[
\dot m(t, \cdot) = \zeta(t, \varphi(t, \cdot))
\]
with $m(t, \cdot) = q(t, \varphi(t, \cdot))\in H$. Considering $(\varphi, m)$ as a new state, we can  define the problem
\begin{multline}
\label{eq:meta.cost.2}
\frac12 \int_0^1 \|v(t)\|_V^2 dt + \frac1{2\sigma^2} \int_0^1 \|\zeta(t)\|^2_H dt \longrightarrow \min\\
\text{subject to } \dot \varphi(t) = v(t)\circ\varphi(t),\ \dot m(t) = \zeta(t)\circ \varphi(t),\ m(0) = q^{(0)} \text{ and } m(1) = q^{(1)}\circ \varphi(1). 
\end{multline}• 
One of the interests of introducing \eqref{eq:meta.cost.2} is that the formulation does not require $m$ to be differentiable (in space) anymore (one can however use a generalized form of the evolution equation in \eqref{eq:meta.cost} to make this problem equivalent to \eqref{eq:meta.cost.2} --- see \cite{ty_2005}). Moreover, applying (still formally) the PMP to \eqref{eq:meta.cost.2} yields another set of optimality conditions that will be convenient later. Introduce a co-state $\rho = (\rho_\varphi, \rho_m) \in (\cB^p)^* \times H^*$ and the Hamiltonian
\[
H(\rho_\varphi, \rho_m,\varphi, m,v, \zeta) = \lform{\rho_\varphi}{v\circ \varphi} + \lform{\rho_m}{\zeta\circ \varphi} - \frac12 \|v\|_V^2 - \frac1{2\sigma^2} \|\zeta\|_H^2.
\]
For $\varphi\in \Diff_V$, introduce the operators $\bT_\varphi: v \to v\circ \varphi$ and $\tilde \bT_\varphi: \zeta \to \zeta \circ \varphi$, respectively from $V$ to $\cB^p$ and from $H$ to itself.
The PMP then gives the equations
\begin{equation}
\label{eq:pmp.2}
\begin{dcases}
\dot \varphi = v \circ \varphi\\
\dot m = \zeta\circ \varphi\\
\dot \rho_\varphi = - \partial_\varphi\lform{\rho_\varphi}{v\circ \varphi} - \partial_\varphi \lform{\rho_m}{\zeta\circ \varphi}\\
\dot \rho_m = 0 \\
v = \bK_V \bT_\varphi^* \rho_\varphi\\
\zeta = \sigma^2 \bK_H \tilde \bT_\varphi^* \rho_m
\end{dcases}
\end{equation}
These conditions imply, in particular, that $\rho_m$ is constant.
The boundary condition $m (1) \circ \varphi(1) = q^{(1)}$ implies a boundary condition for $\rho$, namely that $\lform{\rho_\phi(1)}{w} + \lform{\rho_m}{z} = 0$ whenever 
\[
z  =  \nabla m(1)\cdot D\varphi(1)^{-1} w,
\]
since $\nabla q^{(1)}\circ \varphi(1)  = D\varphi(1)^{-T}\nabla q^{(1)}$. This yields
\[
\lform{\rho_\varphi(1)}{w} + \lform{\rho_m}{\nabla m(1) \cdot D\varphi(1)^{-1} w} = 0.
\]
for all $w\in \cB^p$, or, replacing $w$ by $D \varphi(1)w$,
\begin{equation}
\label{eq:bndry}
\lform{\rho_\varphi(1)}{D\varphi(1) w} + \lform{\rho_m}{\nabla m(1)\cdot w} = 0
\end{equation}
holding for all $w\in \cB^{p}$.\\

  Note that system \eqref{eq:pmp.2} implies that 
\begin{multline*}
\partial_t \left( \lform{\rho_\varphi(t)}{D \varphi(t) w} + \lform{\rho_m}{\nabla m(t)\cdot w} \right)=\\
 - \lform{\rho_\varphi(t)}{Dv(t) \circ \varphi(t) D\varphi(t) w} - \lform{\rho_m}{\nabla \zeta(t)\circ \varphi(t)\cdot D\varphi(t) w}\\
+ \lform{\rho_\varphi(t)}{Dv(t)\circ \varphi(t) D \varphi(t) w} + \lform{\rho_m}{\nabla \zeta(t) \circ \varphi(t)\cdot D \varphi(t) w}
=0,
\end{multline*}
for which we have used $\partial_t D \varphi(t) = Dv(t) \circ \varphi(t) D \varphi(t)$ and $\partial_t \nabla m(t) = D \varphi(t)^T \nabla \zeta(t)\circ \varphi(t)$. This implies that the linear form
\[
\mu(t) : w \mapsto \lform{\rho_\varphi(t)}{D \varphi(t) w} + \lform{\rho_m(t)}{\nabla m(t)\cdot w}
\]
is invariant along \eqref{eq:pmp.2}, and the boundary condition \eqref{eq:bndry} propagates over all times, i.e., $\mu(t)=0$ over $[0,1]$.

Finally, we let the reader check that one can pass from solutions of \eqref{eq:pmp} to solutions of \eqref{eq:pmp.2} with the change of variables $q(t) \circ \phi(t) = m(t)$ and 
\[
\lform{p(t)}{z} = \lform{\rho_m}{z\circ \phi(t)}.
\]
Note also that the boundary condition can be rewritten in terms of $q = m\circ \varphi^{-1}$ as
\begin{equation}
\label{eq:bdry.q}
\lform{\rho_\varphi(t)}{w} = \lform{\rho_m}{\nabla q(t)\cdot w}.
\end{equation}

 \section{Singular Solutions}
 \label{sec:singular}

It was recognized in \cite{hty_2008} that system \eqref{eq:pmp} admits a family of singular solutions. These solutions are obtained directly from \eqref{eq:pmp.2} by taking $\rho_\varphi$ and $\rho_m$ in the form
\begin{eqnarray}
\label{eq:sing.1}
\rho_\varphi(t) &=& \sum_{k=1}^N z_k(t)\cdot \delta_{x_k^{(0)}}\\
\label{eq:sing.2}
\rho_m &=& \sum_{k=1}^N \alpha_k \delta_{x_k^{(0)}}
\end{eqnarray}•
Here,  $x^{(0)} = \{x^{(0)}_k\}_{k=1}^N$ is a collection of points, or particles, in $\mR^d$, $z(t) = \{z_k(t)\}_{k=1}^N$ is a collection of time-dependent vectors in $\mR^d$, $\alpha = \{\alpha_k\}_{k=1}^N$ is a time-independent collection of scalars.\\

Introduce the trajectories $x_k(t) := \varphi(t, x^{(0)}_k)$. Using this notation, we have
\[
\lform{\bT_{\varphi(t)}^*\rho_\varphi(t)}{w} = \lform{\rho_\varphi(t)}{w\circ \phi(t)} = \sum_{k=1}^N z_k(t)\cdot w(x_k(t))
\]
so that 
\[
\bT_{\varphi(t)}^*\rho_\varphi(t) = \sum_{k=1}^N z_k(t)\cdot \delta_{x_k(t)}
\]
and \eqref{eq:pmp.2} implies that (using the reproducing kernel of $V$)
\[
v(t, \cdot) = \sum_{\ell=1}^N K_V(\cdot, x_\ell(t)) z_\ell(t).
\]
Similarly, one gets
\[
\zeta(t, \cdot) =\sigma^2 \sum_{\ell=1}^N K_H(\cdot, x_\ell(t)) \alpha_\ell.
\]
The third equation in \eqref{eq:pmp.2} gives, for $w \in \cB^{p}$,
\[
\sum_{k=1}^N \dot z_k(t)\cdot w(x_k^{(0)})  = 
-  \sum_{k=1}^N z_k(t)\cdot Dv(x_k(t)) w(x_k^{(0)}) - \sum_{k=1}^N \alpha_k \nabla \zeta(x_k(t))\cdot w(x_k^{(0)})
\]
from which we get
\[
\dot z_k(t) = -Dv(x_k(t))^T z_k(t) - \alpha_k \nabla \zeta(x_k(t)).
\]
Using the expansions of $v$ and $\zeta$ and the fact that $\dot x_k = v(t, x_k)$, we obtain the fact that \eqref{eq:sing.1} and \eqref{eq:sing.2} provide solutions of \eqref{eq:pmp.2} as soon as $x$, $m$  and $z$ satisfy the  coupled dynamical system
\begin{equation}
\label{eq:disc.syst}
 %\left\{
\begin{cases}
  \displaystyle\dot{x}_k(t) = \sum_{\ell=1}^{N} K_{V}(x_k(t), x_\ell(t)) z_\ell(t) \\
  \displaystyle \dot m_k(t) = \sum_{\ell=1}^N K_{H}(x_k(t), x_\ell(t)) \alpha_\ell\\
  \displaystyle \dot{z}_k(t) = - \sum_{\ell=1}^{N} \nabla_1 K_V(x_k(t), x_\ell(t)) z_\ell(t)\cdot z_k(t) 
 - \frac1{\sigma^2}\sum_{\ell=1}^N \nabla_1 K_H(x_k(t), x_\ell(t)) \alpha_k \alpha_\ell
\end{cases}
%\right.
\end{equation}
(with the notation $m_k(t) = m(t,x_k^{(0)})$).
The boundary condition applied to $\rho_\varphi$ and $\rho_m$ is 
\[
\sum_{k=1}^N z_k(t)\cdot w = -\sum_{k=1}^N \alpha_k \nabla m(t, x_k^{(0)})\cdot D\varphi(t, x_k^{(0)})^{-1} w
\]
yielding
\[
z_k(t) = - \alpha_k D\varphi(t, x_k^{(0)})^{-T} \nabla m(x_k^{(0)}) = -\alpha_k \nabla q(t, x_k(t))
\]

Note that, given the initial positions $\{ x^{(0)}_k \}$, and initial image $q^{(0)}$, the above system is uniquely specified by the choice of the scalar field $\alpha$, since $z_k(0) = -\alpha_k \nabla q^{(0)}(x_k^{(0)})$. The solutions $\{x_k, z_k\}$ then determine the controls $v$ and $\zeta$ for all $t$ and $x\in \mR^d$, which define in turn the evolving image $q$. This will allow us to design a shooting method for computing metamorphoses that will look for initial conditions that bring trajectories to a desired endpoint.

\section{Discrete Relaxed Problem}
\label{sec:relaxed}

Equations \eqref{eq:disc.syst} are optimality equations for the following relaxation of \eqref{eq:meta.cost.2}:
\begin{multline}
\label{eq:meta.cost.3}
\frac12 \int_0^1 \|v(t)\|_V^2 dt + \frac1{2\sigma^2} \int_0^1 \|\zeta(t)\|^2_H dt \longrightarrow \min\\
\text{subject to }\dot x_k(t) = v(t, x_k(t)),\ \dot m_k(t) =  \zeta(t, x_k(t)),\ m_k(0) = q^{(0)}(x_k^{(0)}),\\
 \text{ and }m_k(1) = q^{(1)}(x_k(1)).
\end{multline}
This is just \eqref{eq:meta.cost.2} with boundary conditions only enforced at the initial and final points of the trajectories $x_k(t), k=1, \ldots, N$. Because the constraints only depend on the evaluation of $v$ and $\zeta$ along the discrete trajectories, the optimal ones should minimize their respective norms subject to the values taken at these points. Well-known results on RKHS's \cite{aronszajn1950theory,wahba1990spline} imply that these optimal solutions must assume the form 
\begin{eqnarray*}
v(t, \cdot) &=& \sum_{k=1}^N K_V(\cdot, x_k(t)) z_k(t) \\
\zeta(t, \cdot) &=& \sum_{k=1}^N K_H(\cdot, x_k(t)) \alpha_k(t)
\end{eqnarray*}•
for some coefficients $z$ and $\alpha$, and that their norms are given by
\begin{eqnarray*}
\|v\|_V^2 &=& \sum_{k, \ell=1}^N z_k(t)\cdot K_V(x_k(t), x_\ell(t)) z_\ell(t) \\
\|\zeta(t)\|_H^2 &=& \sum_{k,\ell=1}^N K_H(x_k(t), x_\ell(t)) \alpha_k(t) \alpha_\ell(t).
\end{eqnarray*}

Solutions of \eqref{eq:meta.cost.3} are therefore solutions of the reduced problem
\begin{align}
\label{eq:meta.cost.4}
\frac12 \sum_{k, \ell=1}^N  \int_0^1 z_k(t)\cdot K_V(x_k(t), x_\ell(t)) z_\ell(t) dt &+ \frac1{2\sigma^2} \sum_{k,\ell=1}^N \int_0^1 K_H(x_k(t), x_\ell(t)) \alpha_k(t) \alpha_\ell(t) dt \longrightarrow \min\\
\nonumber
&\text{subject to }\\
\nonumber
&\dot x_k(t) = \sum_{\ell=1}^N K_V(x_k(t), x_\ell(t)) z_\ell(t),\\
\nonumber
&\dot m_k(t) =  \sum_{\ell=1}^N K_H(x_k(t), x_\ell(t))  \alpha_\ell(t),\\
\nonumber
 &m_k(0) = q^{(0)}(x_k^{(0)}) \text{ and }m_k(1) = q^{(1)}(x_k(1)).
\end{align}\\

The PMP associated to this problem derives, as before, from a control-dependent Hamiltonian
\begin{eqnarray}
\label{eq:disc.ham}
H_{\alpha, z}(p_x, p_m, x,m) &=& \sum_{k, \ell=1}^N p_{x,k}(t)\cdot K_V(x_k(t), x_\ell(t)) z_\ell(t) \\
\nonumber &&+ \sum_{k,\ell=1}^N K_H(x_k(t), x_\ell(t)) p_{m,k}(t) \alpha_\ell(t)\\
\nonumber && -\frac12 \sum_{k, \ell=1}^N  z_k(t)\cdot K_V(x_k(t), x_\ell(t)) z_\ell(t) dt \\
\nonumber &&- \frac1{2\sigma^2} \sum_{k,\ell=1}^N  K_H(x_k(t), x_\ell(t)) \alpha_k(t) \alpha_\ell(t)
\end{eqnarray}
It is then easy to check that the optimality conditions $\partial_z H = 0$ and $\partial_\alpha H = 0$ imply that $p_x=z$ and $p_m = \alpha$; from $\partial_m H=0$, one finds that $\alpha$ is constant; finally, the equation $\dot z = - \partial_x H$ yields an equation identical to the evolution of $z$ in  \eqref{eq:disc.syst}. \\

The boundary condition for \eqref{eq:meta.cost.4} is 
\[
z_k(1) = - \alpha_k \nabla q^{(1)}(x_k(1)).
\]
This identity propagates over time as follows: define $\tilde m(t)\in H$ by $\partial_t \tilde m =\zeta(t) \circ \varphi(t)$ with $\tilde m(1) = q^{(1)} \circ \varphi(1)$. Define $\tilde q(t)$ such that $\tilde m(t) = \tilde q(t) \circ \varphi(t)$. Then
\[
z_k(t) = - \alpha_k \nabla \tilde q(t, x_k(t))
\]
at all times. To prove this statement write
\[
\partial_t \nabla \tilde m(t) = D\varphi(t)^T \nabla \zeta(t)\circ \varphi(t)
\]
on the first hand, and, on the other hand, 
\begin{eqnarray*}
\partial_t \nabla \tilde m(t) &=& \partial_t (D\varphi(t)^T \nabla \tilde q(t) \circ \varphi(t))\\
&=& D\varphi(t)^T Dv(t)\circ \varphi(t)^T \nabla \tilde q(t) \circ \varphi(t) + D\varphi(t)^T \partial_t (\nabla \tilde q(t)\circ \varphi(t)).
\end{eqnarray*}
Identifying the expressions, we find
\[
\partial_t (\nabla \tilde q(t, x_k(t))) = - Dv(t, x_k(t))^T \nabla \tilde q(t, x_k(t)) + \nabla \zeta(t, x_k(t)).
\]
This implies
\[
\partial_t (z_k(t) + \alpha_k \nabla \tilde q(t,x_k(t))) = - Dv(t, x_k(t))^T (z_k(t) + \alpha_k \nabla \tilde q(t,x_k(t)))
\]
proving that $D \varphi(t, x_k(0))^T (z_k(t) + \alpha_k \nabla \tilde q(t,x_k(t)))$ is conserved along the motion. This quantity therefore vanishes at all times as soon as it vanishes at time $t=1$.

Note that this boundary condition differs from the one we had in the unrelaxed problem, because $\tilde m$ and $\tilde q$ are not necessarily identical to $m$ and $q$. We have, actually, $q(t, x_k(t)) = \tilde q(t, x_k(t))$ for all $k$ and $t$, since they have the same derivative and coincide at $t=1$,  but this identity does not hold for the the full functions $q(t, \cdot)$ and $\tilde q(t, \cdot)$, since the constraints at $t=1$ only involve the particles.
Note also that, if one initializes system \eqref{eq:disc.syst} with $z_k(0) = -\alpha_k \nabla q^{(0)}(x_k^{(0)})$, one also gets $z_k(t) = - \alpha_k \nabla  q(t, x_k(t))$ at all times. This can be an interesting constraint to enforce, since it is consistent with the continuous problem, even though this does not provide a solution of the relaxed problem.

\section{Solution of the Discrete Problem} 
\label{sec:solution}

We now describe a shooting method for the solution of \eqref{eq:meta.cost.4}, in which we solve for  $(\alpha_1, \ldots, \alpha_N)$ and $(z_1^{(0)}, \ldots, z_N^{(0)})$ such that the solution of \eqref{eq:disc.syst} initialized at $x_k(0) = x_k^{(0)}$, $m_k(0) = q^{(0)}(x_k^{(0)})$ and $z_k(0) = z_k^{(0)}$ satisfies $m_k(1) = q^{(1)}(x_k(1))$ for $k=1,\ldots, N$. Considering $x_k(\cdot)$ and $m_k(\cdot)$ as functions of $\alpha$ and $z^{(0)}$, we  minimize
\begin{equation}
\label{eq:ener}
 E(\alpha, z^{(0)}) = \sum_{k=1}^N (m_k(1) - q^{(1)}(x_k(1))^2.
\end{equation}
Here, we assume that $q^{(1)}$ is defined and known everywhere (by interpolation, for example). Computing the differential of $E$ gives
\begin{equation}
\label{eq:de}
dE = 2\sum_{k=1}^N (m_k(1) - q^{(1)}(x_k(1)) (dm_k(1) - \nabla q^{(1)}(x_k(1))\cdot dx_k(1)) 
\end{equation}
where $d m_k$ and $d x_k$ are differentials dual to infinitesimal changes in the discrete variables $m_k$ and $x_k$.

To compute $dE$, we apply the well-known adjoint method to compute derivatives of functions of solutions of dynamical systems. Writing $\theta(t) = (x,m,z)$, and defining $F$ so that \eqref{eq:disc.syst} is $\dot \theta = F(\theta, \alpha)$, we let $\boldsymbol\theta(t, \theta^{(0)}, \alpha)$ denote the solution of this equation with initial condition $\theta(0) = \theta^{(0)}$ and parameter $\alpha$. Given variations $\delta \alpha$ and $\delta \theta^{(0)}$, then
\[
\delta \theta(t) := \partial_{\theta^{(0)}} \boldsymbol\theta . \delta\theta^{(0)} + \partial_\alpha \boldsymbol\theta . \delta\alpha
\] 
satisfies the ODE
\[
\partial_t \delta \theta = \partial_\theta F(\theta, \alpha) . \delta \theta + \partial_\alpha F(\theta, \alpha). \delta\alpha
\]
with initial condition $\delta \theta(0) = \delta \theta^{(0)}$. Introduce the solution $\xi = (\xi_x, \xi_m, \xi_z)$ of the adjoint ODE
\[
\partial_t \xi = -\partial_\theta F(\theta, \alpha).  \xi
\]
so that
\begin{equation}
\label{eq:dxi}
\partial_t (\xi\cdot\delta\theta) = \xi\cdot \partial_\alpha F(\theta, \alpha). \delta\alpha.
\end{equation}
If one takes 
\begin{equation}
\label{q:xi1}
\xi(1) = \{-2(m_k(1) - q^{(1)}(x_k(1))\nabla q^{(1)}(x_k(1)), 0, 2(m_k(1) - q^{(1)}(x_k(1)),\}_{k=1}^N
\end{equation}•
then, from \eqref{eq:de} and \eqref{eq:dxi},
\[
dE .\delta\theta = \xi(1)\cdot \delta \theta(1) = \xi(0)\cdot \delta \theta^{(0)} + \left(\int_0^1 \partial_\alpha F(\theta, \alpha).\xi dt\right)^T \delta \alpha.
\]
In other terms,  defining $\xi(t)$ and $\eta(t)$ as solutions of the system
\begin{equation}
\label{eq:adj.syst}
\begin{cases}
\partial_t \xi = -\partial_\theta F(\theta, \alpha)^T  \xi\\
\partial_t \eta = -\partial_\alpha F(\theta, \alpha)^T  \xi
\end{cases}
\end{equation}•
with $\xi(1)$ as above and $\eta(1) = 0$,  one finds
\[
\partial_{\theta^{(0)}} E = \xi(0) \text{ and } \partial_\alpha E = \eta(0).
\]
Detailed expressions for system \eqref{eq:adj.syst} expressed in terms of $x$, $\alpha$ and $z$ are provided in the appendix.\\

This system is used for the adjoint method to transport the discrete covector $dE$ backwards in time, in order to find
a descent direction for the optimization. In our implementation, the initial conditions $m^{(0)}$ and $x^{(0)}$ are fixed, and the optimization only operates on $z^{(0)}$ and $\alpha$, yielding Algorithm \ref{alg.shooting}.
\begin{algorithm}
\caption{Shooting Algorithm}
\label{alg.shooting}
\begin{algorithmic}
\Require { template $q^{(0)}$, target  $q^{(1)}$; specify kernels $K_{V}, K_H$}; matching parameter $\sigma$ \\
$\alpha \gets 0, z^{(0)} \gets 0$
\While{ (not stop CG) } \\
\hspace{\algorithmicindent}1. Compute $\partial_{z^{(0)}} E, \partial_\alpha E$: \\
\hspace{\algorithmicindent}\hspace{\algorithmicindent}1.1 Compute $dE = \partial_{x_k} E ~dx_k + \partial_{m_k}E ~dm_k $ given by \eqref{eq:de} \\
\hspace{\algorithmicindent}\hspace{\algorithmicindent}1.2 Compute $\xi_z(0), \eta(0)$: solve the adjoint system backwards in time starting from $dE$ at $t=1$. \\
%\hspace{\algorithmicindent}\hspace{\algorithmicindent}Find $\nabla_\alpha E$ as the solution of the discrete system $u \mapsto K_H u = \eta_\alpha(0)$.\\
\hspace{\algorithmicindent}2. Update conjugate direction and perform line search\\
\hspace{\algorithmicindent}3. Update $z^{(0)}, \alpha$
\EndWhile
\end{algorithmic}
\end{algorithm}

If, as discussed at the end of Section \ref{sec:relaxed}, the minimization is run with the constraint $z_k^{(0)} = - \alpha_k \nabla q^{(0)}(x_k^{(0)})$, the gradients obtained at step 1.2 of  Algorithm \ref{alg.shooting} only have to be combined into $\tilde \eta_{k}(0) = \eta_{k}(0) -\nabla q^{(0)}(x_k^{(0)})\cdot \xi_{z,k}(0)$ in order to update $\alpha$. Note also that the obtained derivatives, $\xi_z(0)$ and $\eta(0)$ (or $\tilde \eta(0)$) can be conditioned according to their natural inner product before performing step 2, using the linear transformation $\eta(0) \mapsto \mathcal K_H(x^{(0})^{-1} \eta(0)$ and $\xi_z(0) \mapsto
\mathcal K_V(x^{(0})^{-1} \xi_z(0)$, where $\mathcal K_H(x^{(0})$ is the matrix with entries $K_H(x_k^{(0)}, x_l^{(0)})$ and $\mathcal K_V(x^{(0})$ is formed similarly
with $d$ by $d$ blocks $K_V(x_k^{(0)}, x_l^{(0)})$.

\section{Numerical Experiments}
\label{sec:experiments}

We now illustrate our method with some simple numerical experiments.  We used Python for our 
implementation, making extensive use of the open source packages Numpy, Scipy, and the f2py tool to integrate Fortran and Python \cite{scipy}.  
The results in the examples below are visualized using Paraview \cite{paraview}.

For all numerical results, we use
\[
K_V(x,y) = (1 + u + 3u^2/7 + 2u^3/21 + u^4/105)\, e^{-u}. \mathrm{Id}_{\mR^d}
\]
and
\[
K_H(x,y) = (1 + \tilde u + \tilde u^2/3)\, e^{-\tilde u}
\]
with $u = |x-y|/\tau_V$ and $\tilde u = |x-y|/\tau_H$, where $\tau_V$ and $\tau_H$ are width parameters associated to the reproducing kernels. These kernels provide RKHS's equivalent to Sobolev spaces $H^k(\mR^d, \mR^d)$ and $H^r(\mR^d)$ with $k=(9+d)/2$ and $r=(5+d)/2$, yielding respective inclusions in $\cB^4$ and $C^2_0(\mR^d)$. All experiments are discretized on
a 2D grid with isotropic resolution $\Delta x_1 = \Delta x_2 = 1$.\\

The first examples match images from the training set in the MNIST character recognition database: the letter ``D'' and the digit ``8''. 
We use a discrete square with $72^2$ points and a time discretization $\Delta t = 0.1$ (10 timesteps). Images from the character database are 
upsampled  at the sampling rate for this grid. We used $\tau_V = 1.5$ and $\tau_h = 0.5$.  Figure \ref{fig.d72} illustrates the matching of two versions of the 
letter $D$ (bottom row at left, to bottom row at right).  The top row shows the optimal evolution of the template $m(t)$, while the bottom row shows the 
evolution of the deformed template $q(t) = m(t) \circ \varphi(t)^{-1}$.  Figure \ref{fig.eight40} shows matching of versions of the digit eight (top left to bottom right), 
along with the deformed gridlines to visualize the minimizing deformation.

In Figure \ref{fig.leaf100}, we show the metamorphosis of two leaves from the LeafSnap database \cite{leafsnap}, after downsampling the images to a grid of $100^2$
and converting to grayscale images.  Here, $\tau_V = 3.0$ and $\tau_h = 0.5$.

Figure \ref{fig.b_momenta} shows the minimizing momenta $\alpha$ when matching the image on the top row to each of the seven images of the final row (which shows the
final morphed image); the second row is an intensity map of the momenta.  On the linear space of momenta, we can take linear combinations, as depicted in Figure \ref{fig.avg_b_momenta}; this allows us to generate random images based on the ones obtained in Figure \ref{fig.b_momenta}, by solving \eqref{eq:disc.syst} with initial momentum
\begin{equation}
\label{eq:rand.mom}
\alpha(0) = \bar\alpha_0 + \frac{c}{\sqrt n}\sum_{k=1}^7 \xi_k(\alpha_{0,k} - \bar\alpha_0)
\end{equation}
where $\xi_1, \ldots, \xi_7$ are independent standard Gaussian random variables, $\alpha_{0,k}$ is the initial momentum obtained for the $k$th image in Figure \ref{fig.b_momenta} and $\bar\alpha_0$ is their average. The covariance structure of the resulting random momentum $\alpha(0)$ coincides with the empirical covariance estimated from the seven examples. 

\begin{figure}[h]
%\begin{center}
%\includegraphics[scale=.07]{figures/d72_template.png}
%\includegraphics[scale=.07]{figures/d72_target.png}\\
%\end{center}
\begin{center}
\includegraphics[width=.15\textwidth]{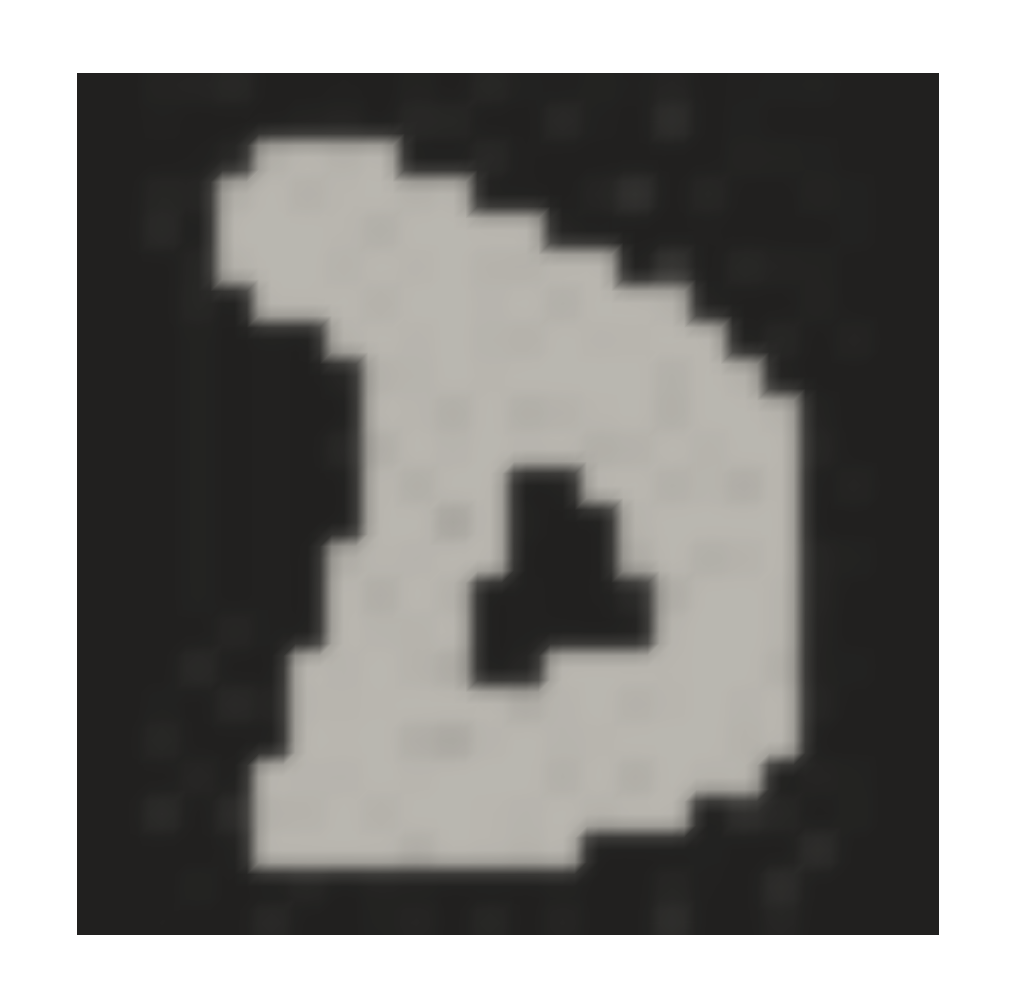}
\includegraphics[width=.15\textwidth]{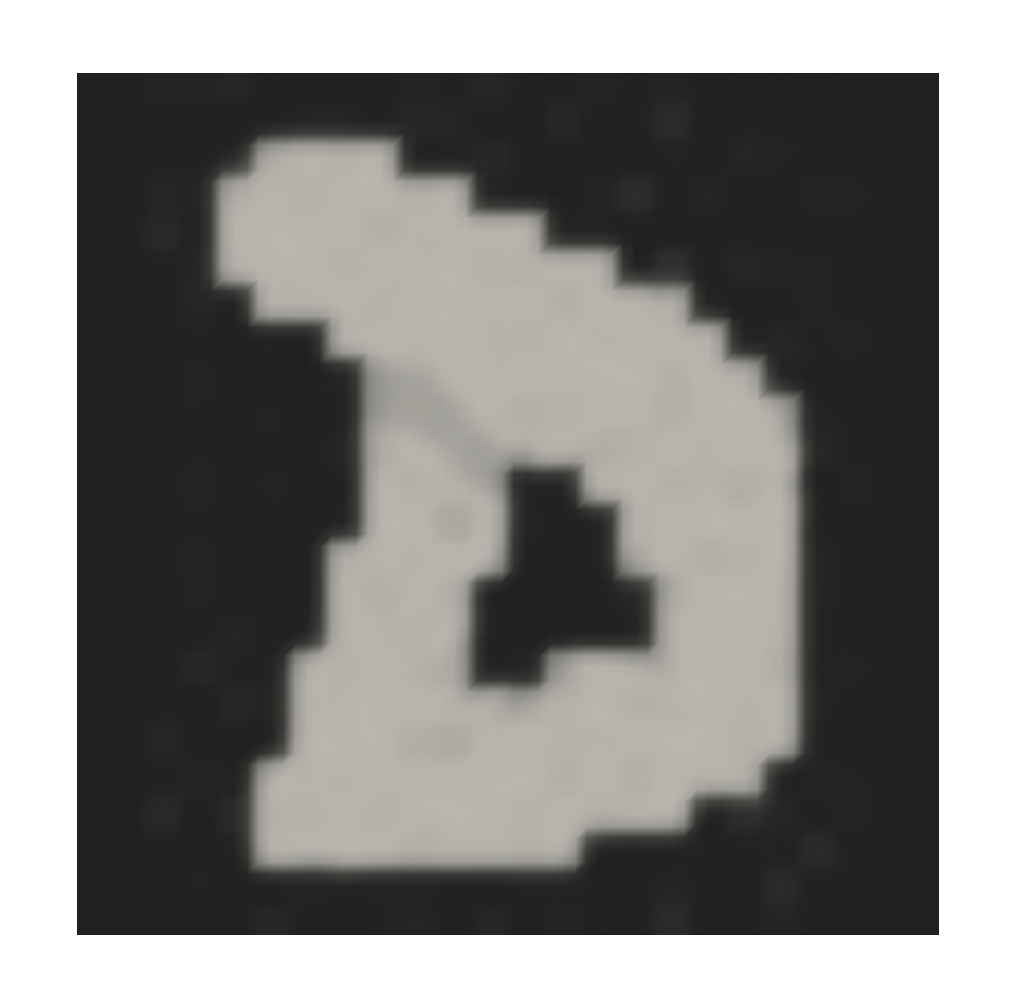}
\includegraphics[width=.15\textwidth]{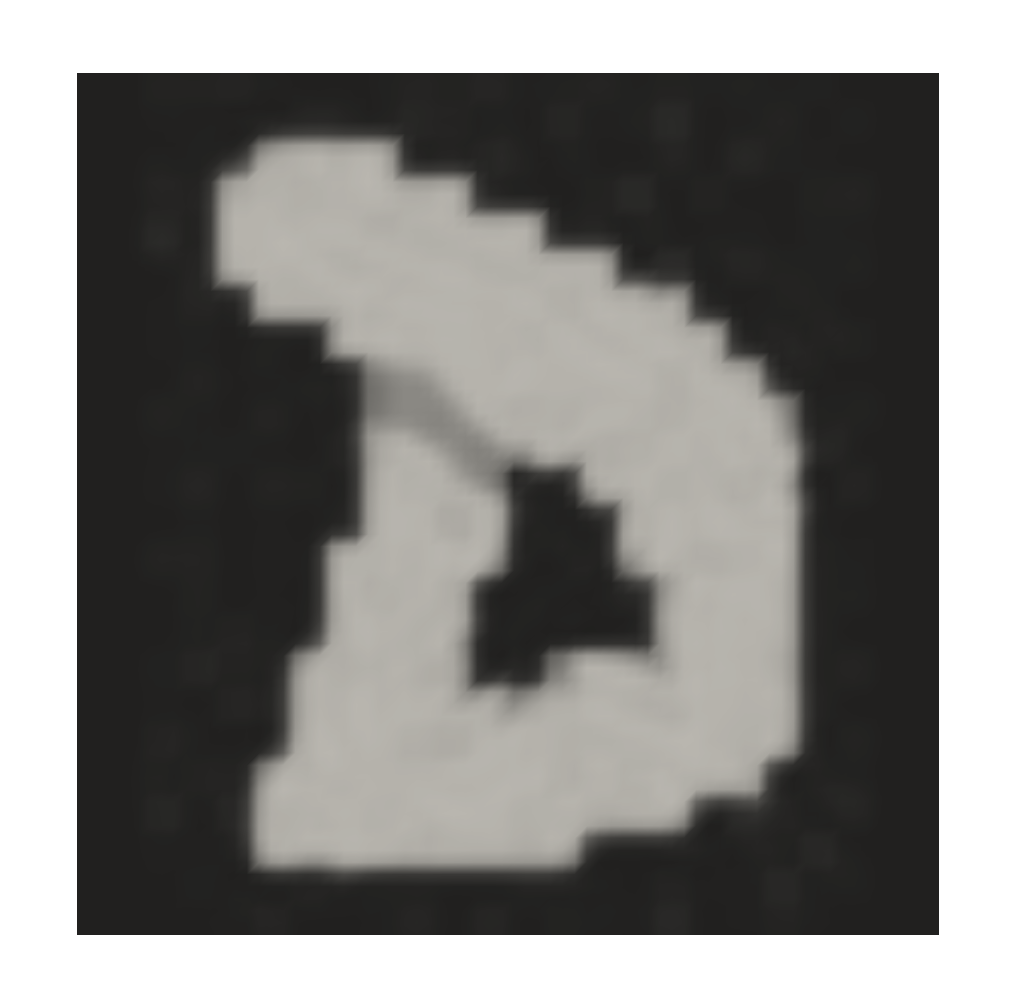}
\includegraphics[width=.15\textwidth]{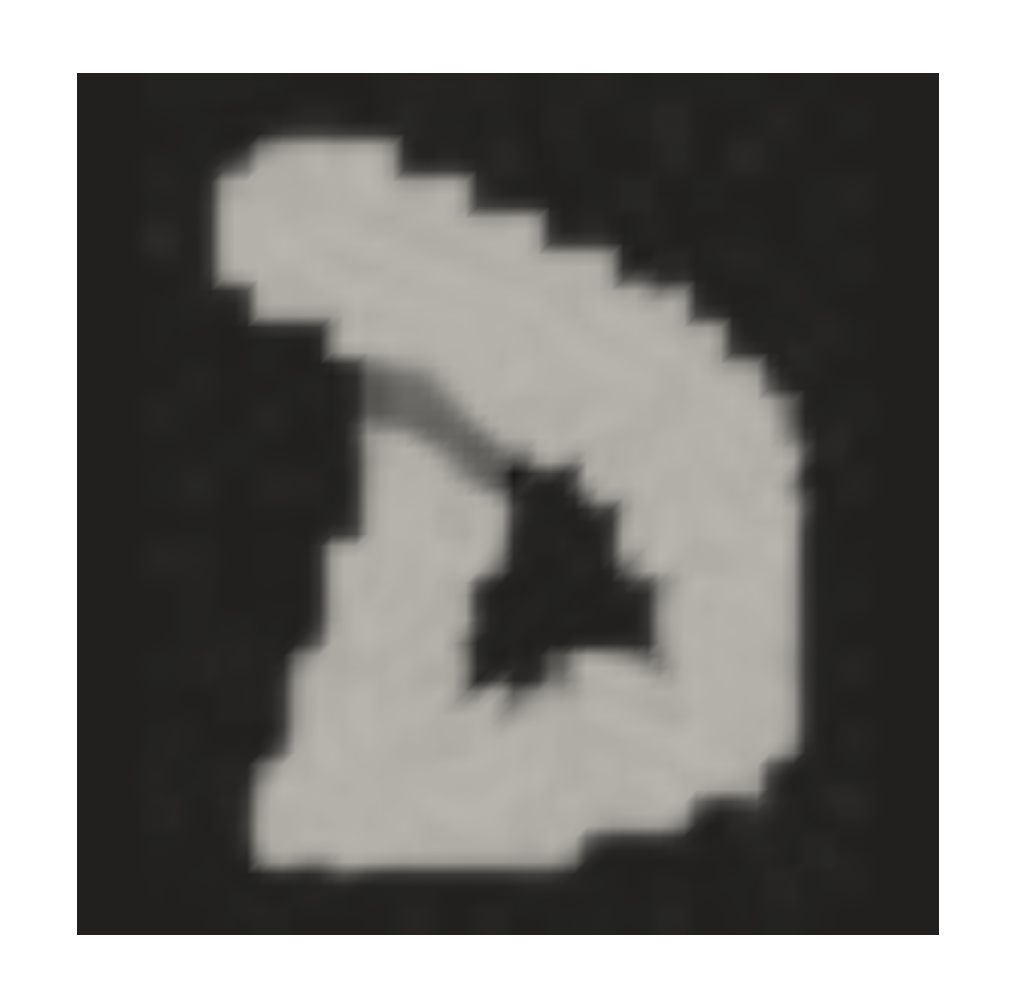}
\includegraphics[width=.15\textwidth]{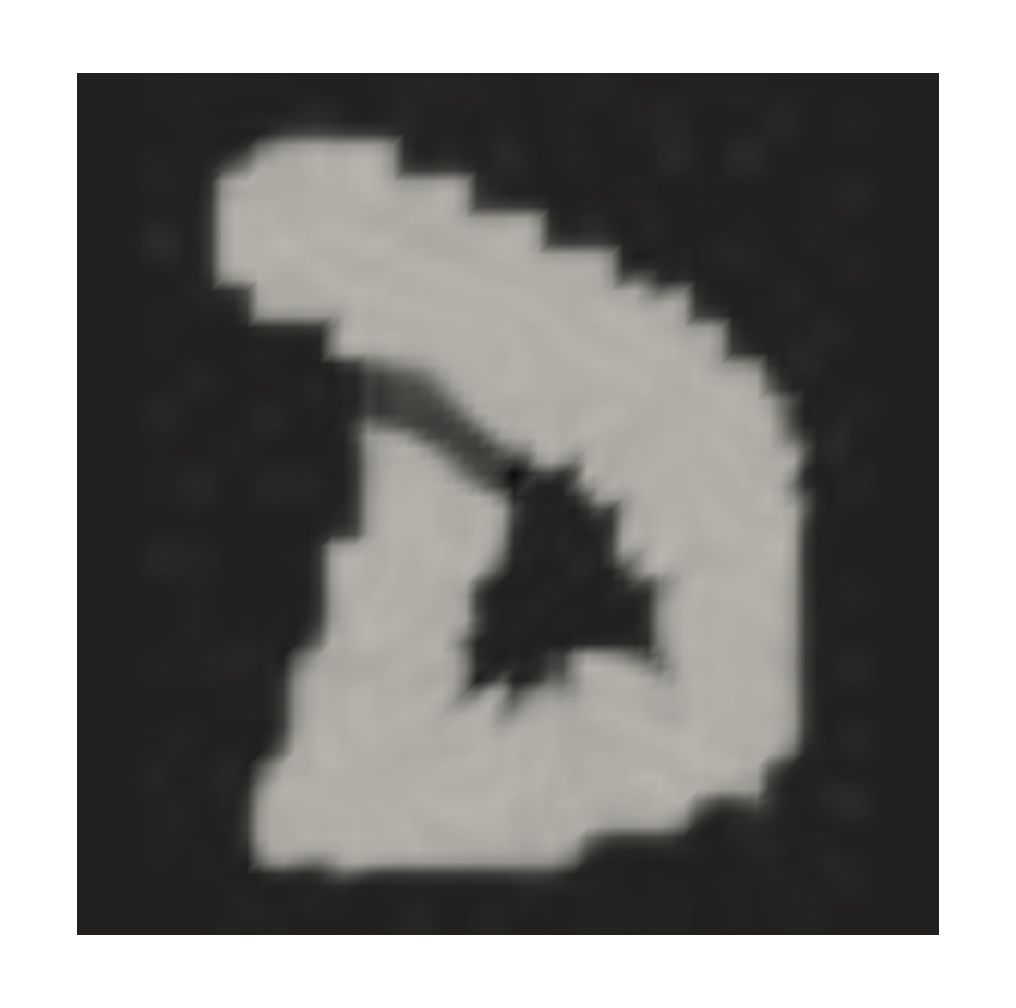}
\includegraphics[width=.15\textwidth]{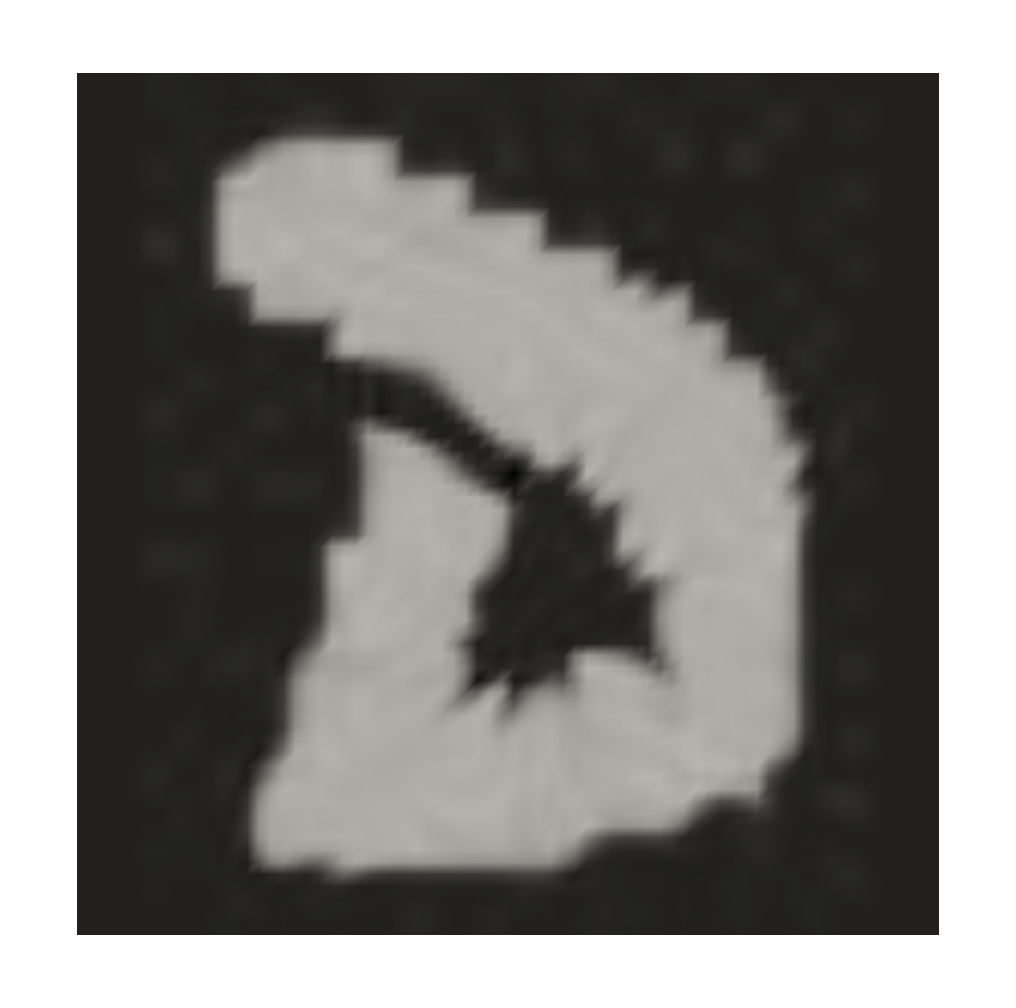}
\end{center}
\begin{center}
\includegraphics[width=.15\textwidth]{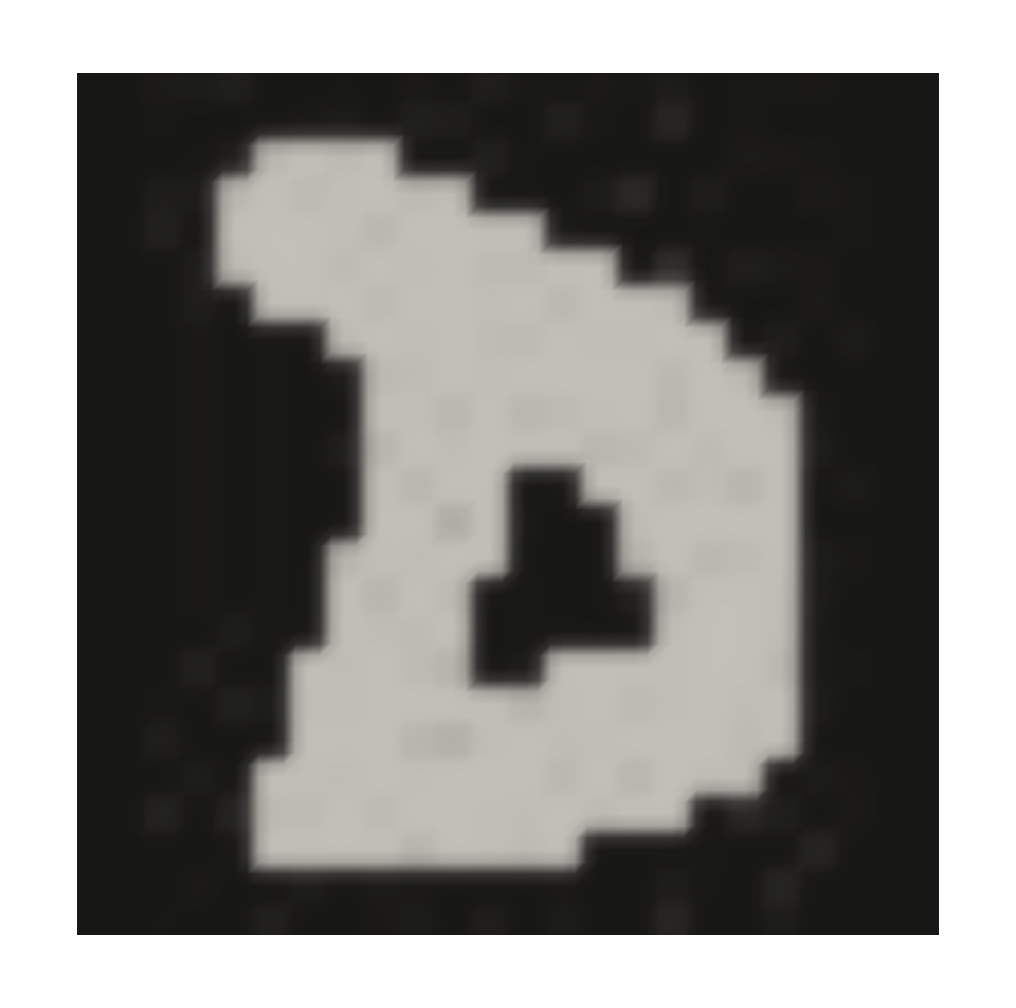}
\includegraphics[width=.15\textwidth]{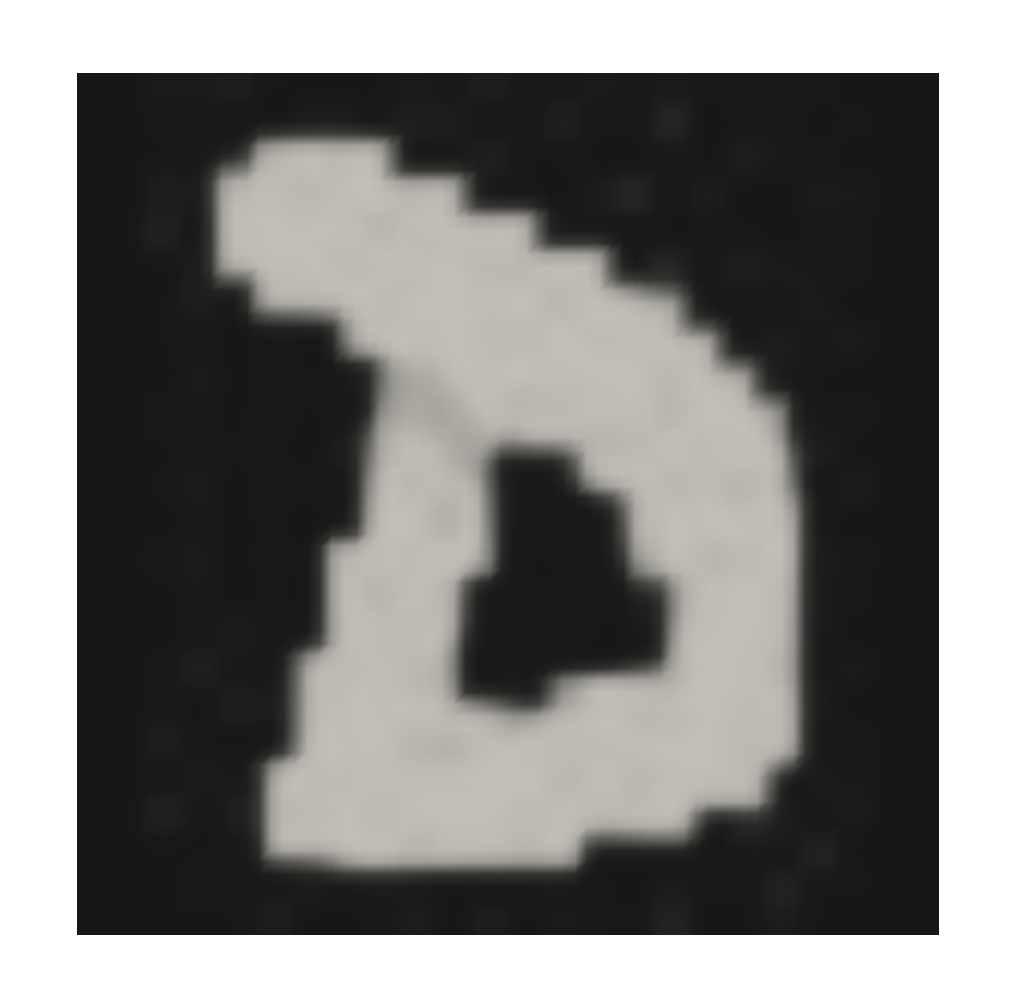}
\includegraphics[width=.15\textwidth]{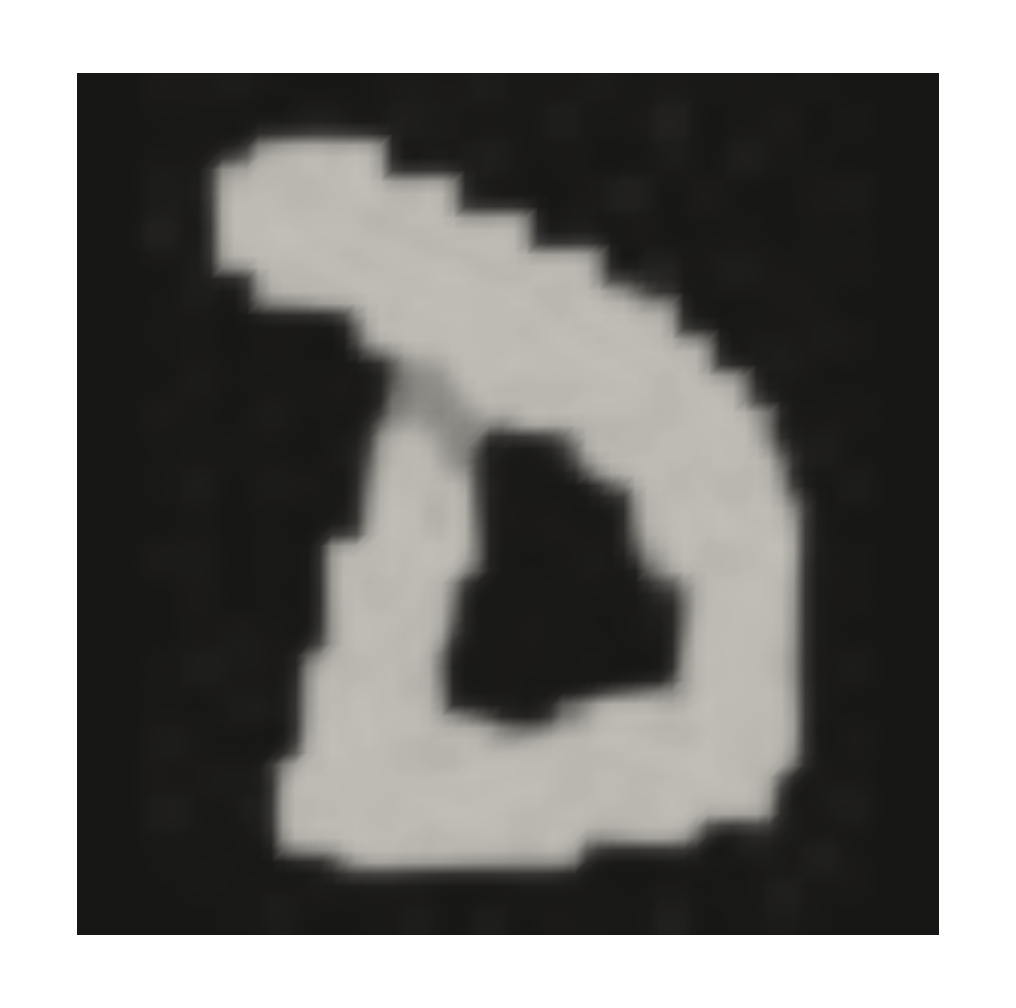}
\includegraphics[width=.15\textwidth]{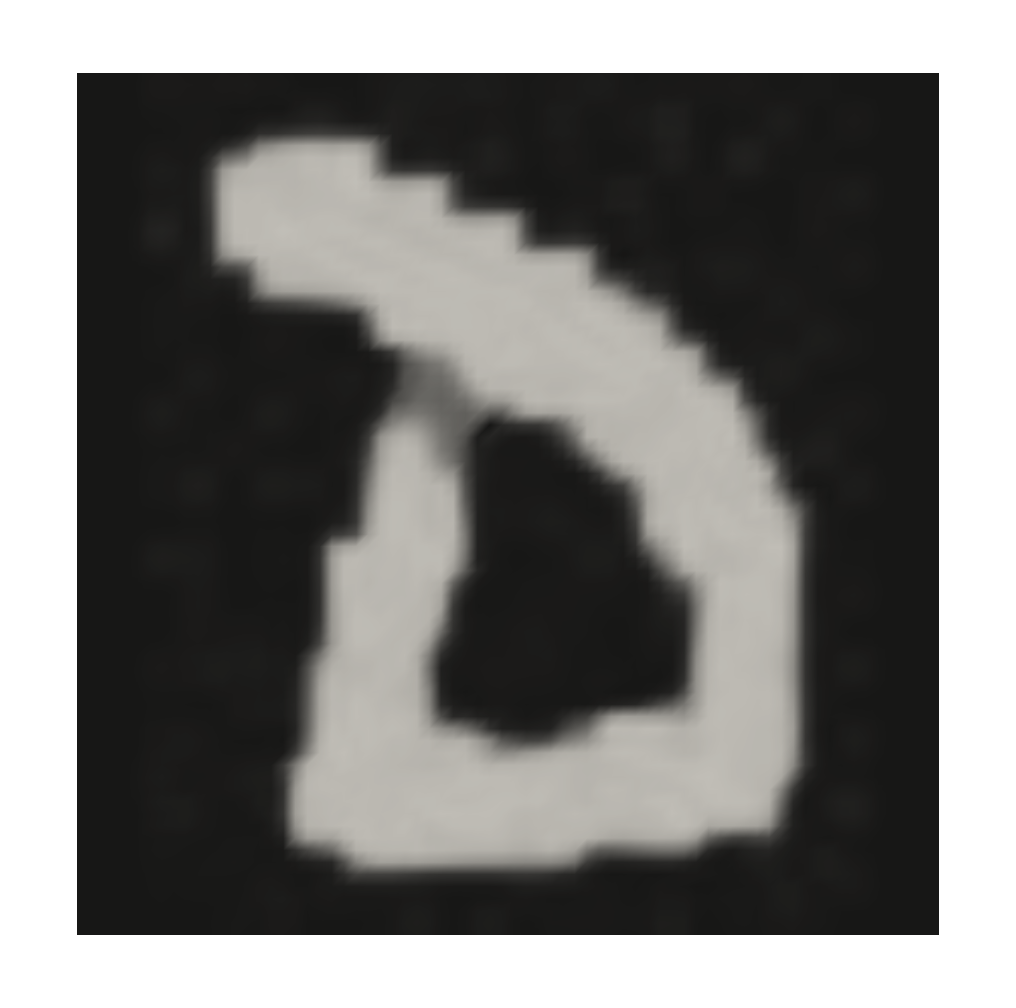}
\includegraphics[width=.15\textwidth]{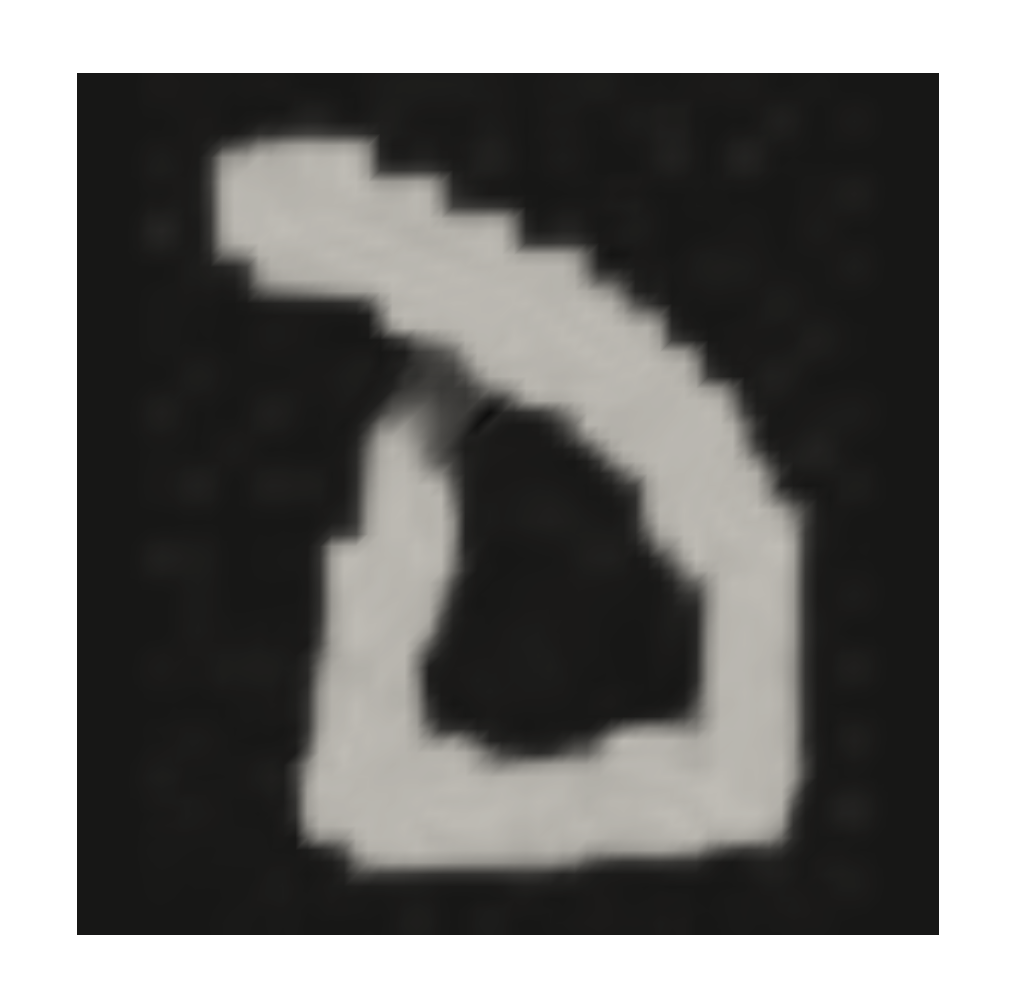}
\includegraphics[width=.15\textwidth]{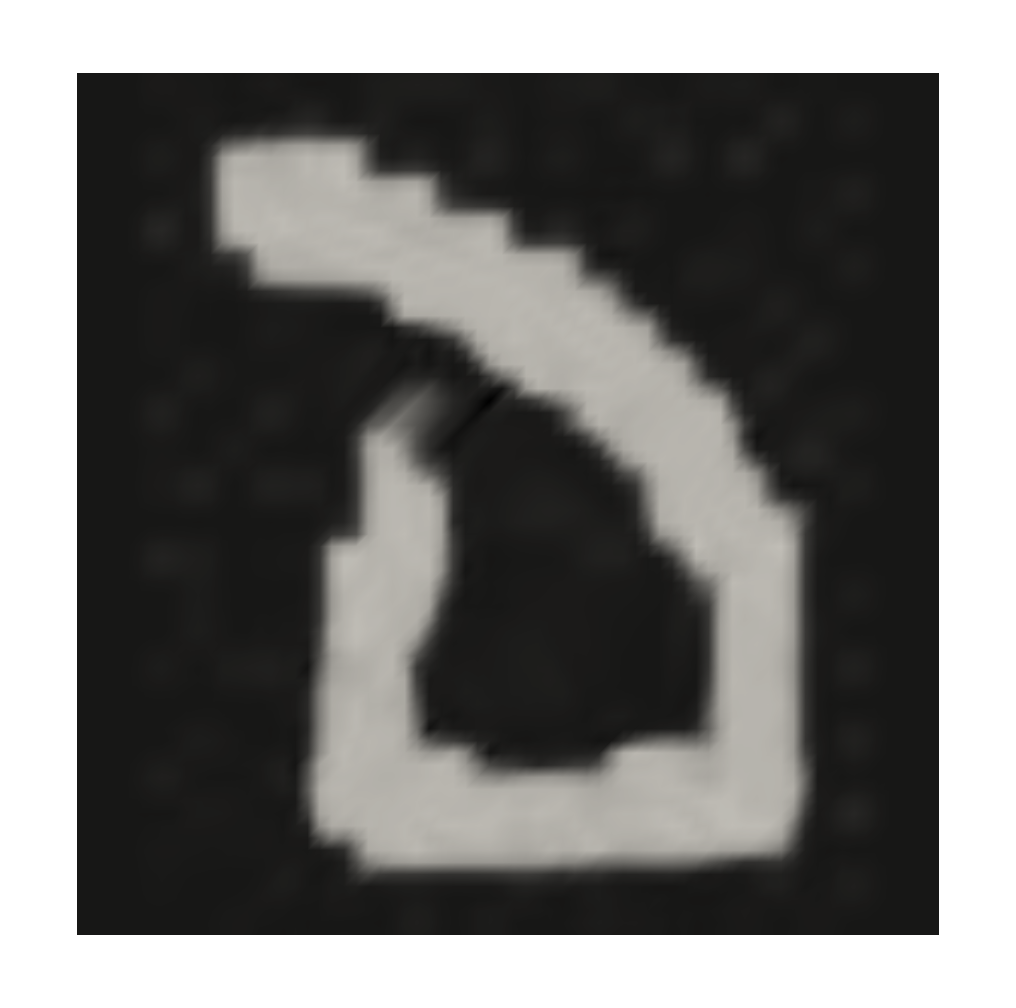}
\end{center}
\caption{Morphing of letter D from MNIST training set: top row shows evolution of the template;
bottom row shows evolution of deformed template.}
\label{fig.d72}
\end{figure}

\begin{figure}[h]
 \begin{center}
\includegraphics[width=0.3\textwidth]{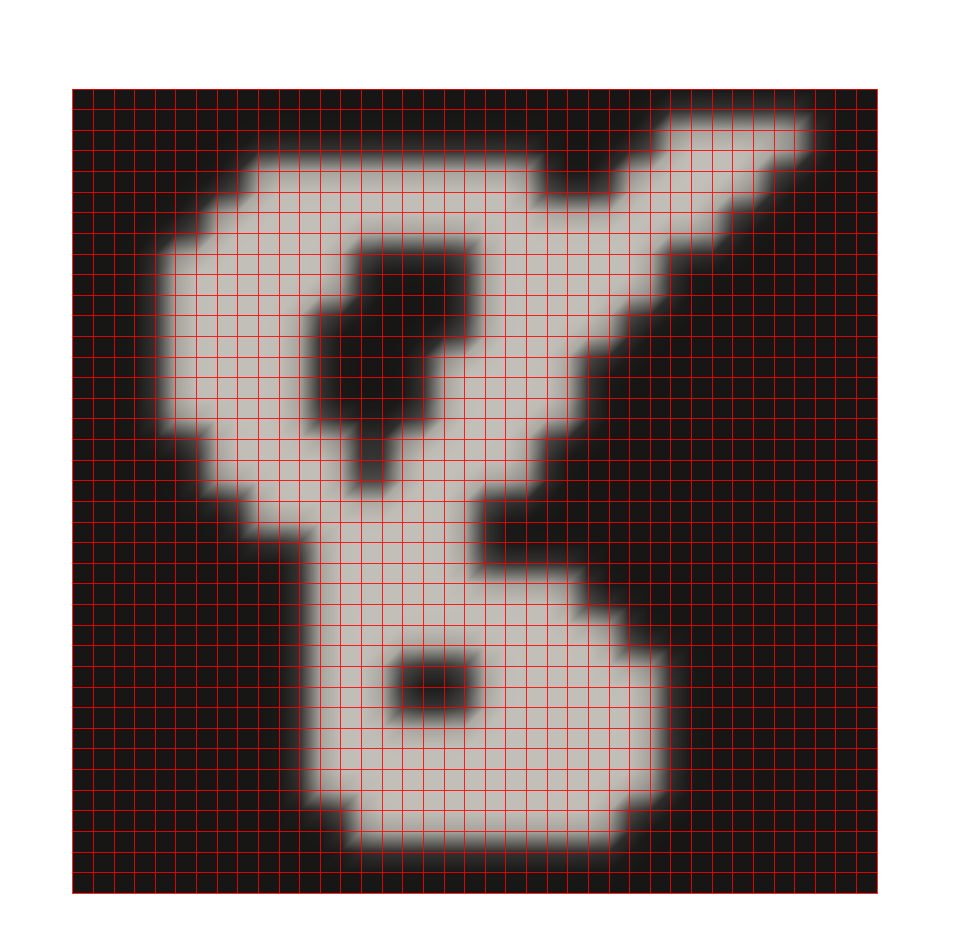}
\includegraphics[width=0.3\textwidth]{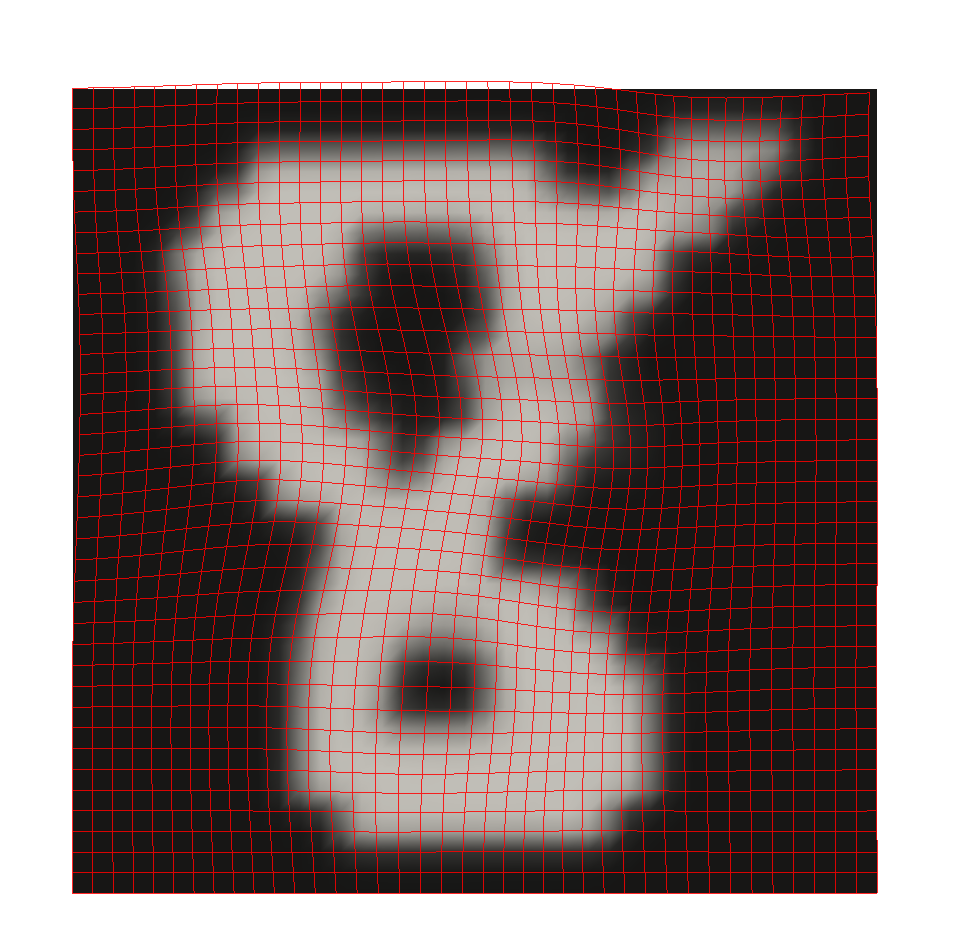}
\includegraphics[width=0.3\textwidth]{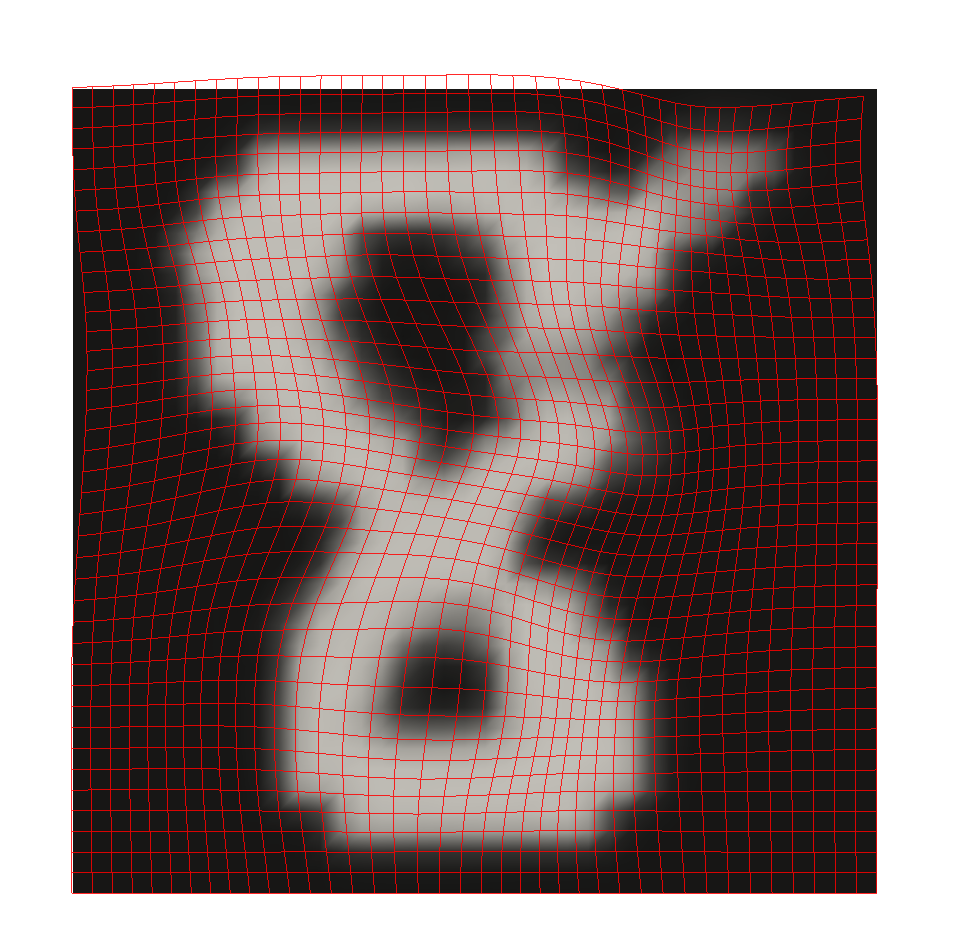}\\
\includegraphics[width=0.3\textwidth]{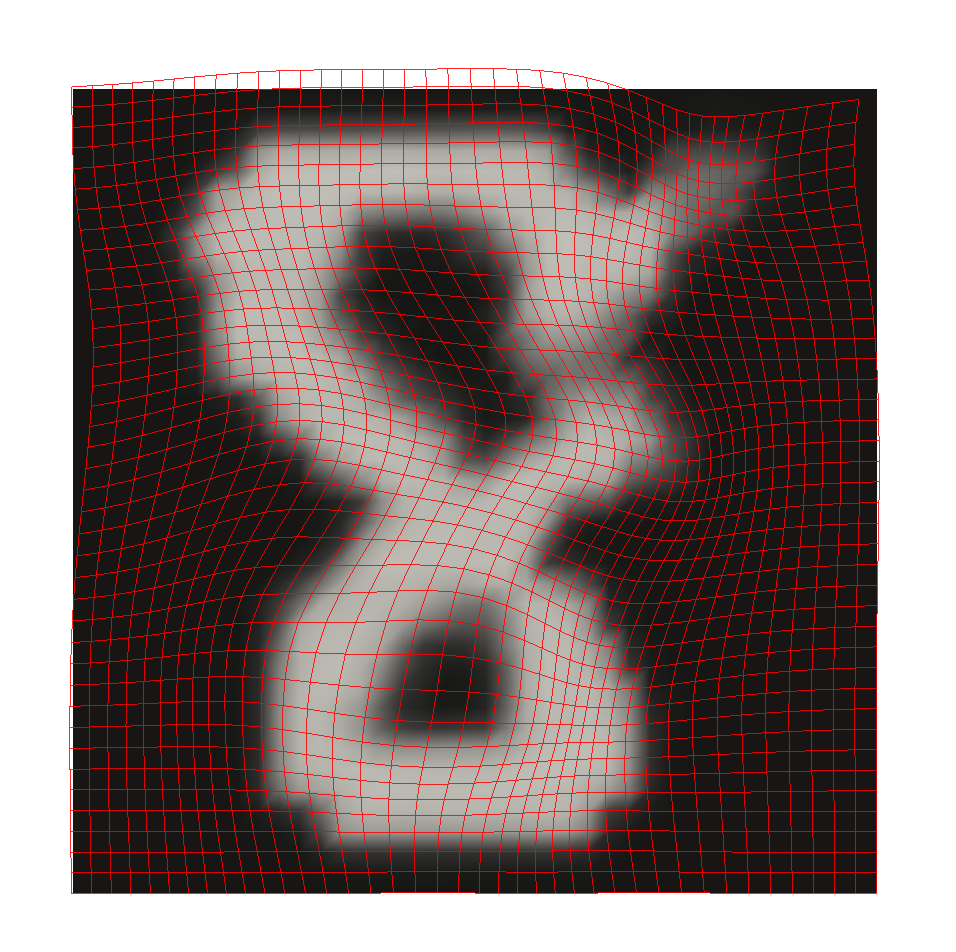}
\includegraphics[width=0.3\textwidth]{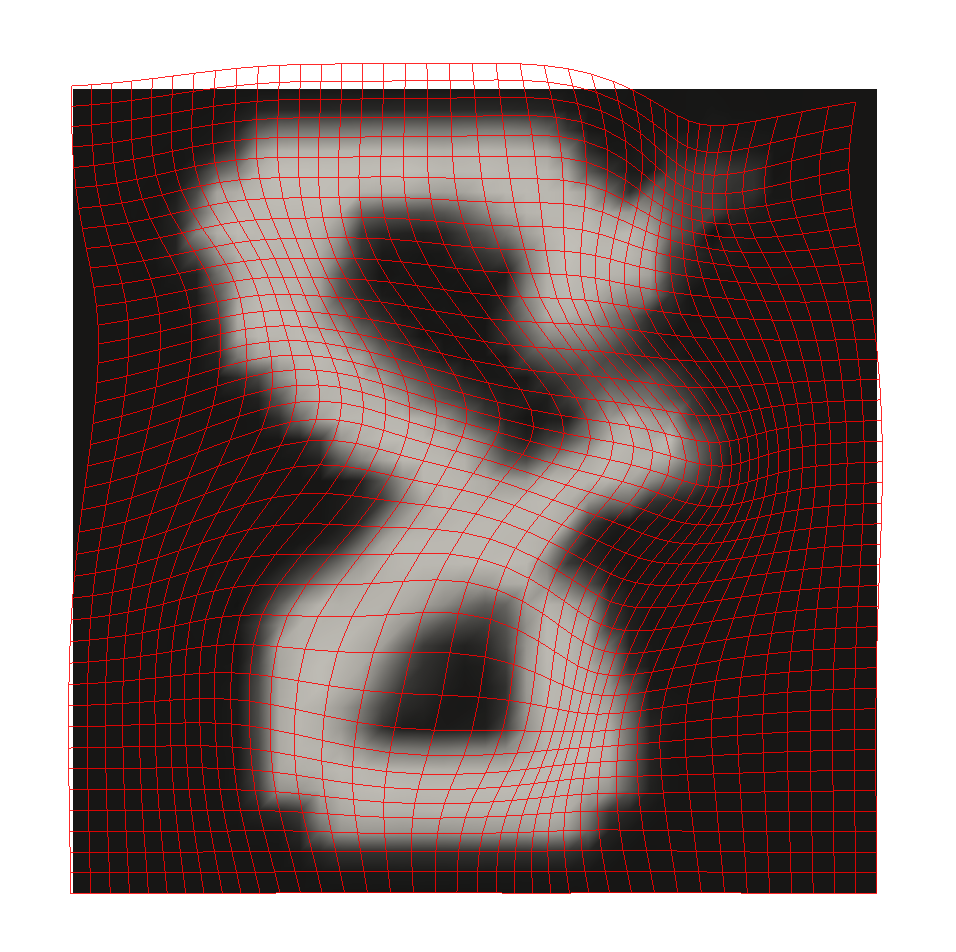}
\includegraphics[width=0.3\textwidth]{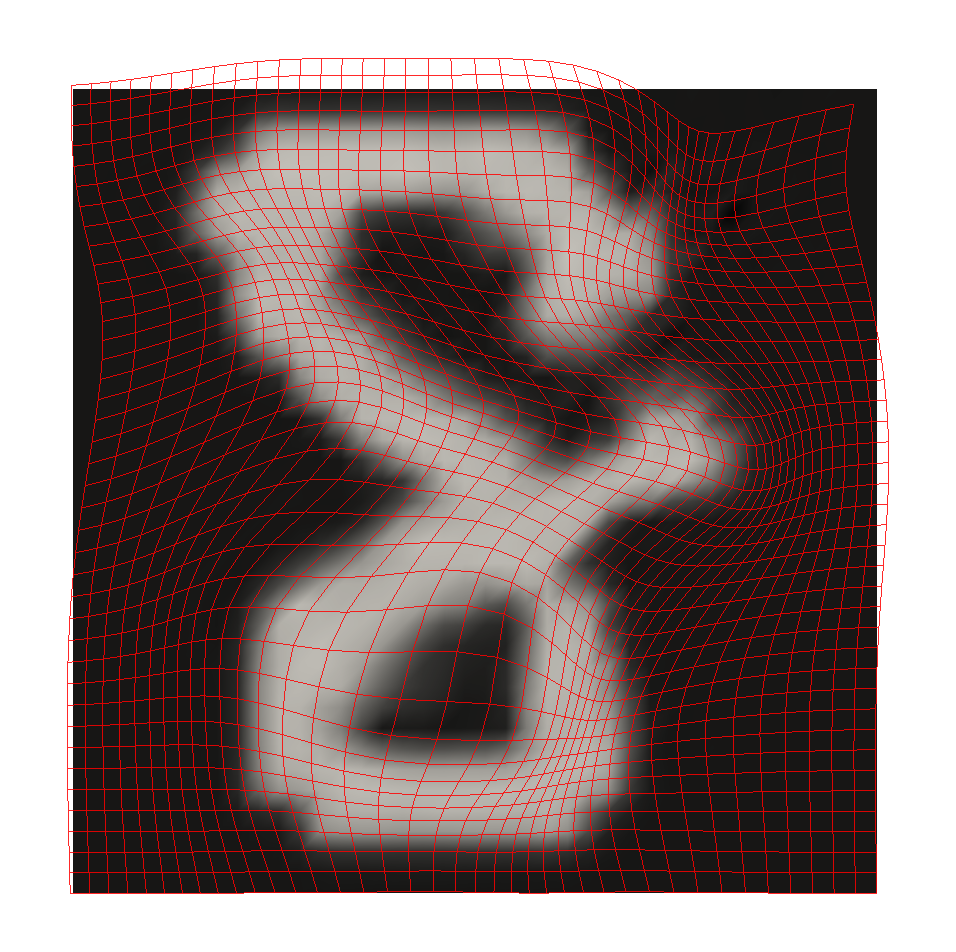}
\end{center}
\caption{Morphing of smoothed version of digit 8 from MNIST training set, 
where the coordinate grid is warped by the diffeomorphism and illustrated with grid lines.}
\label{fig.eight40}
\end{figure}

\begin{figure}[h]
 \begin{center}
 \includegraphics[width=0.15\textwidth]{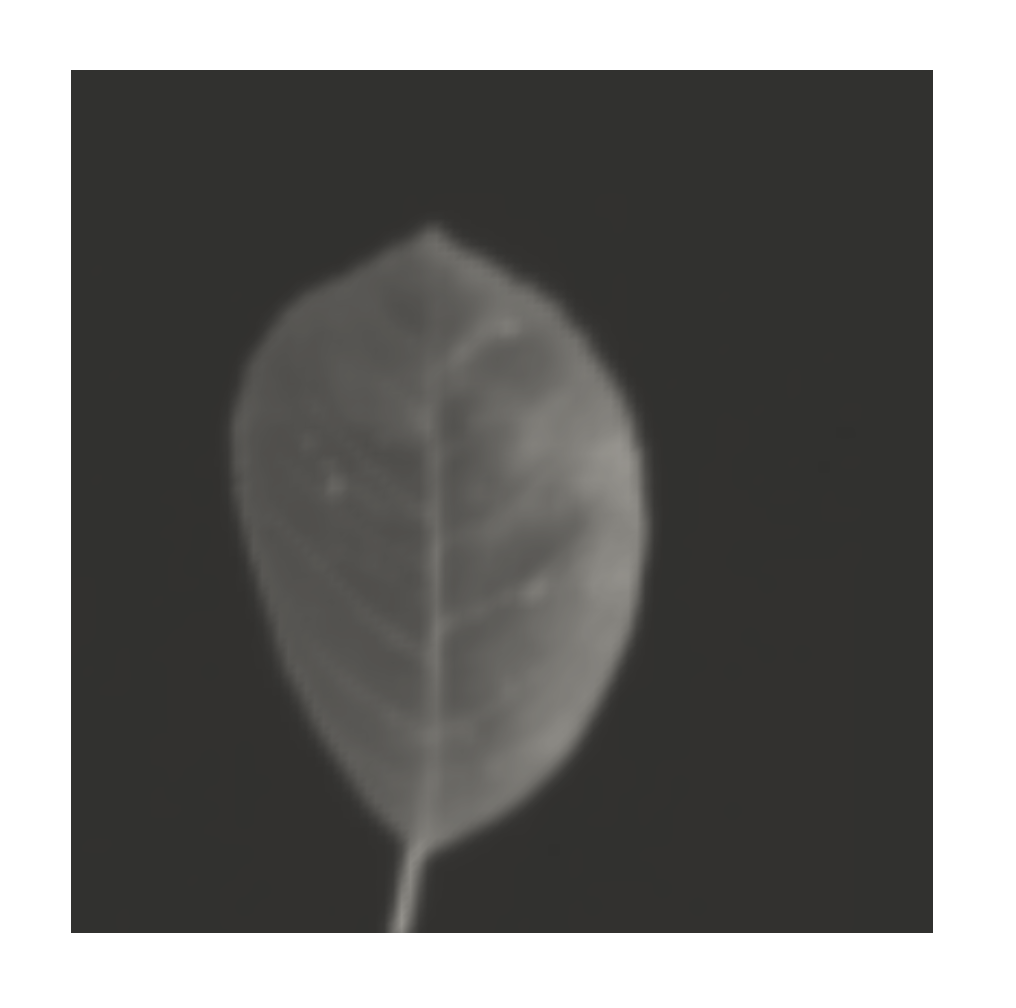}
 \includegraphics[width=0.15\textwidth]{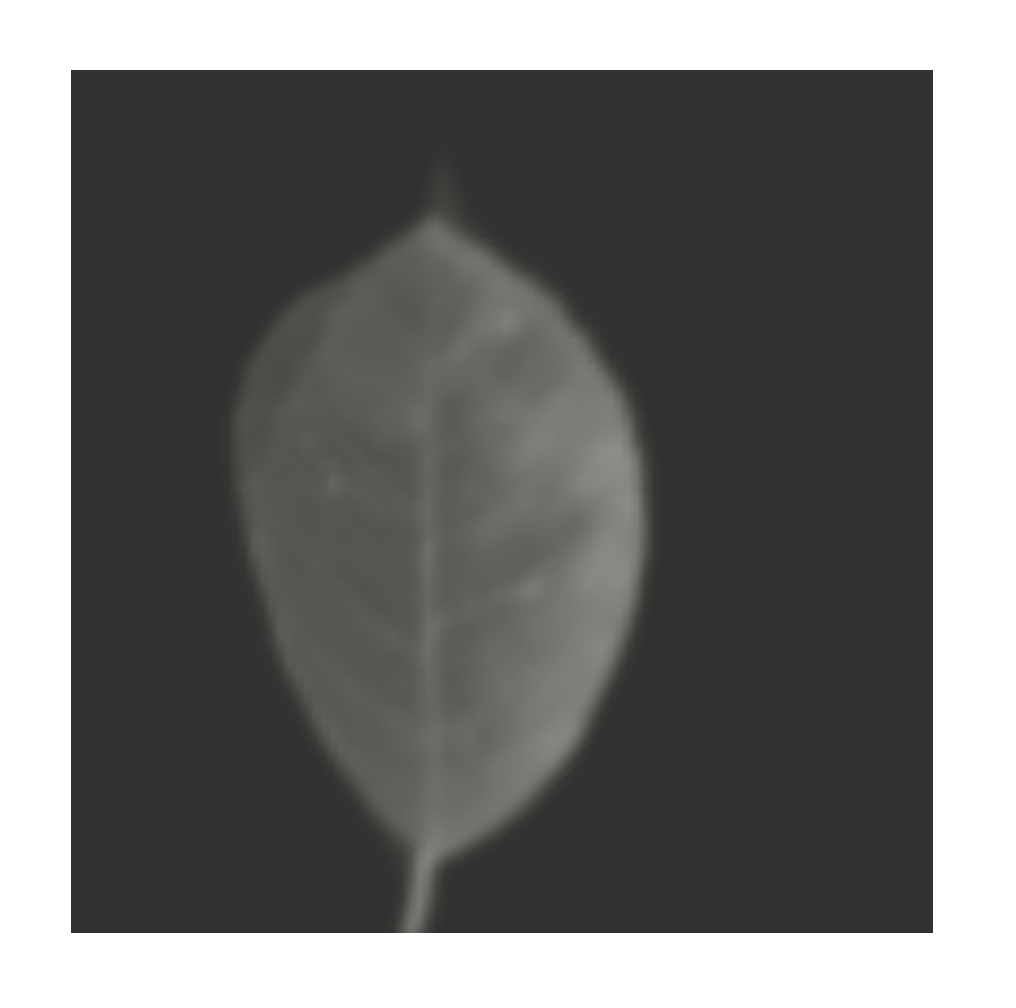}
 \includegraphics[width=0.15\textwidth]{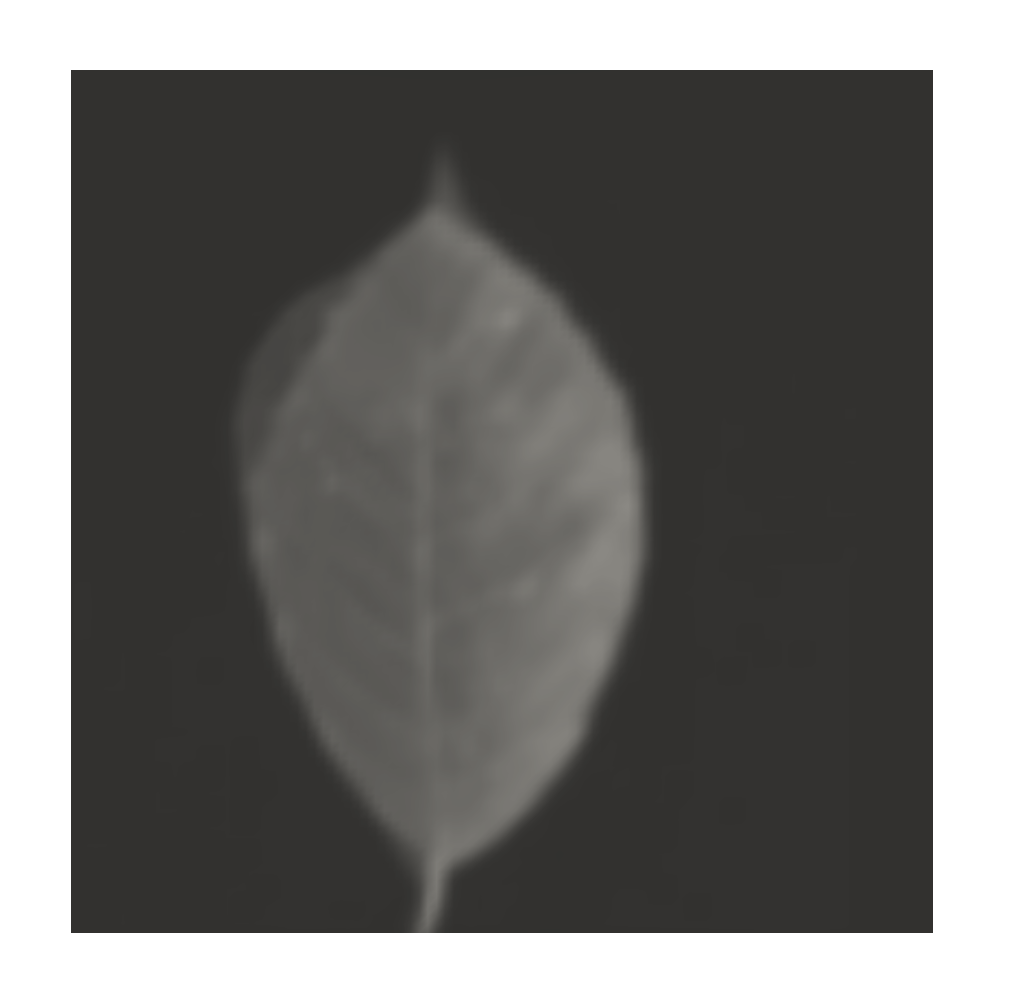}
 \includegraphics[width=0.15\textwidth]{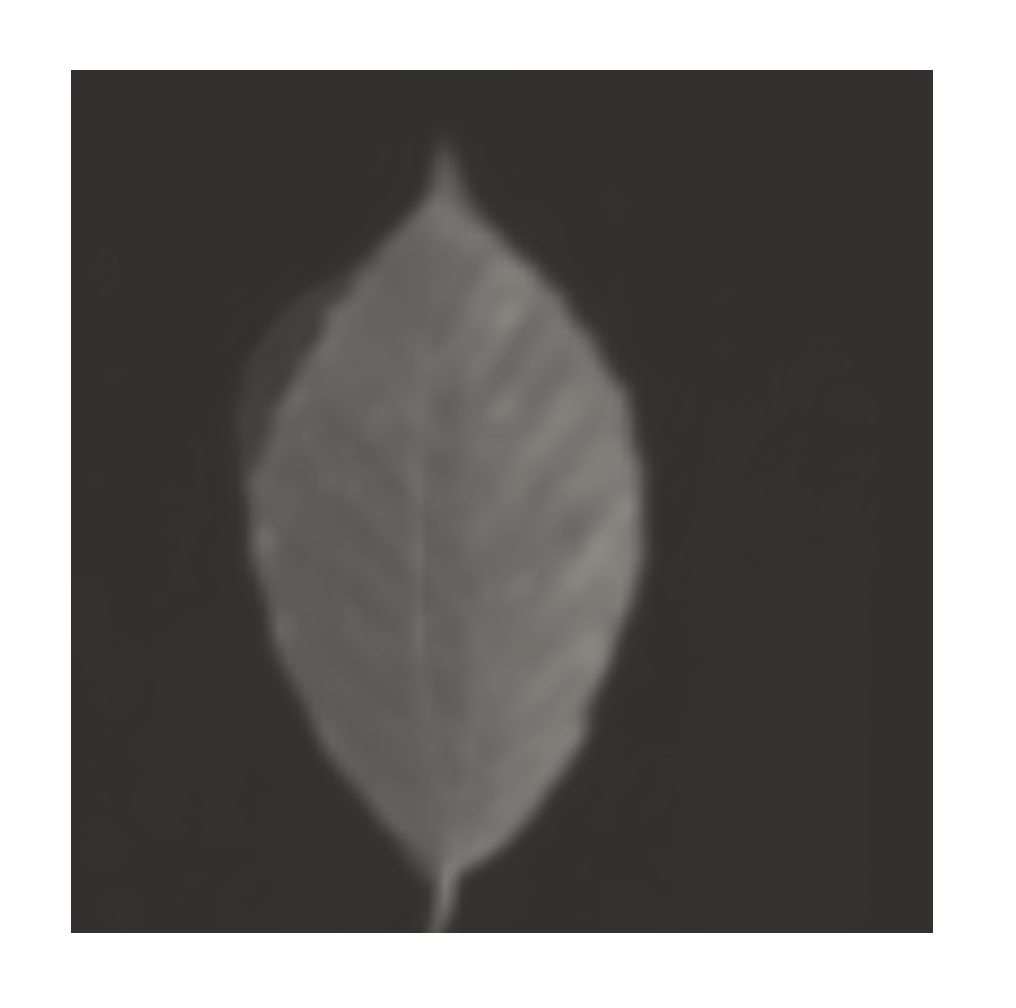}
 \includegraphics[width=0.15\textwidth]{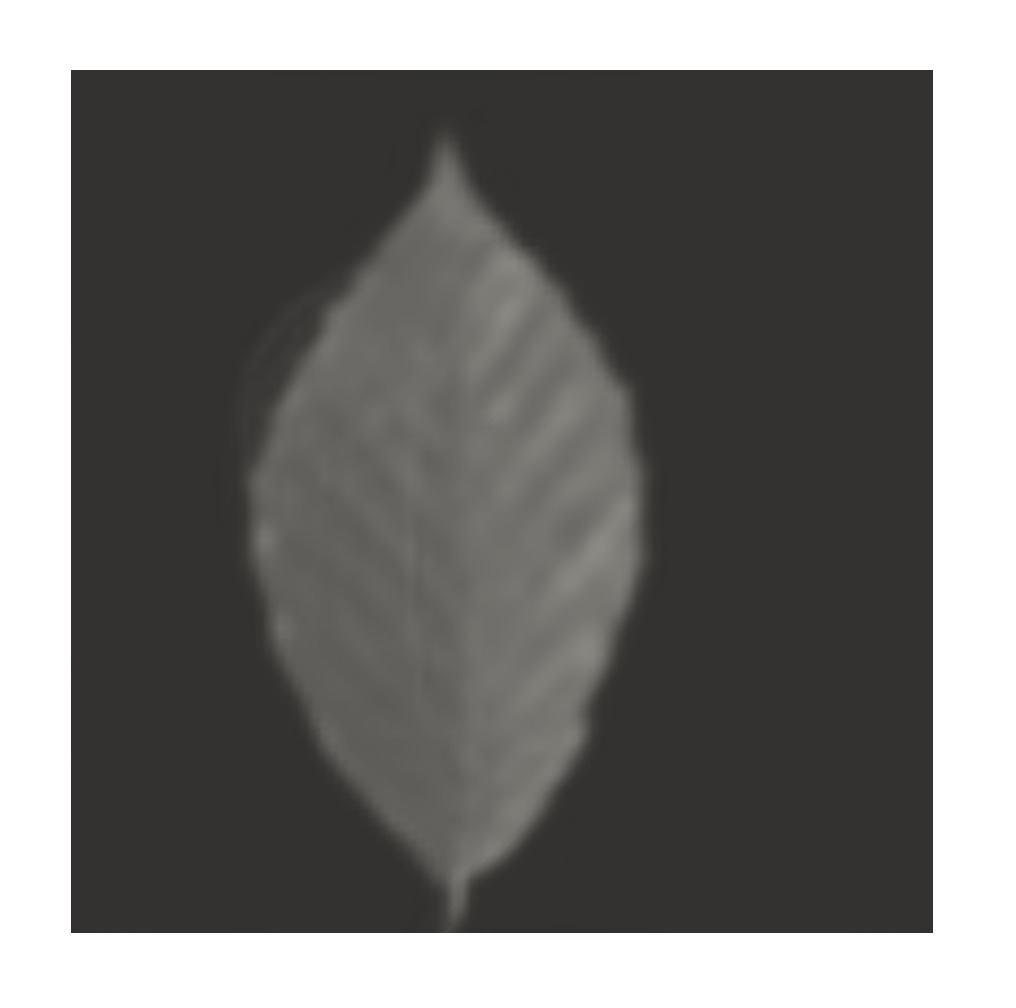}
 \includegraphics[width=0.15\textwidth]{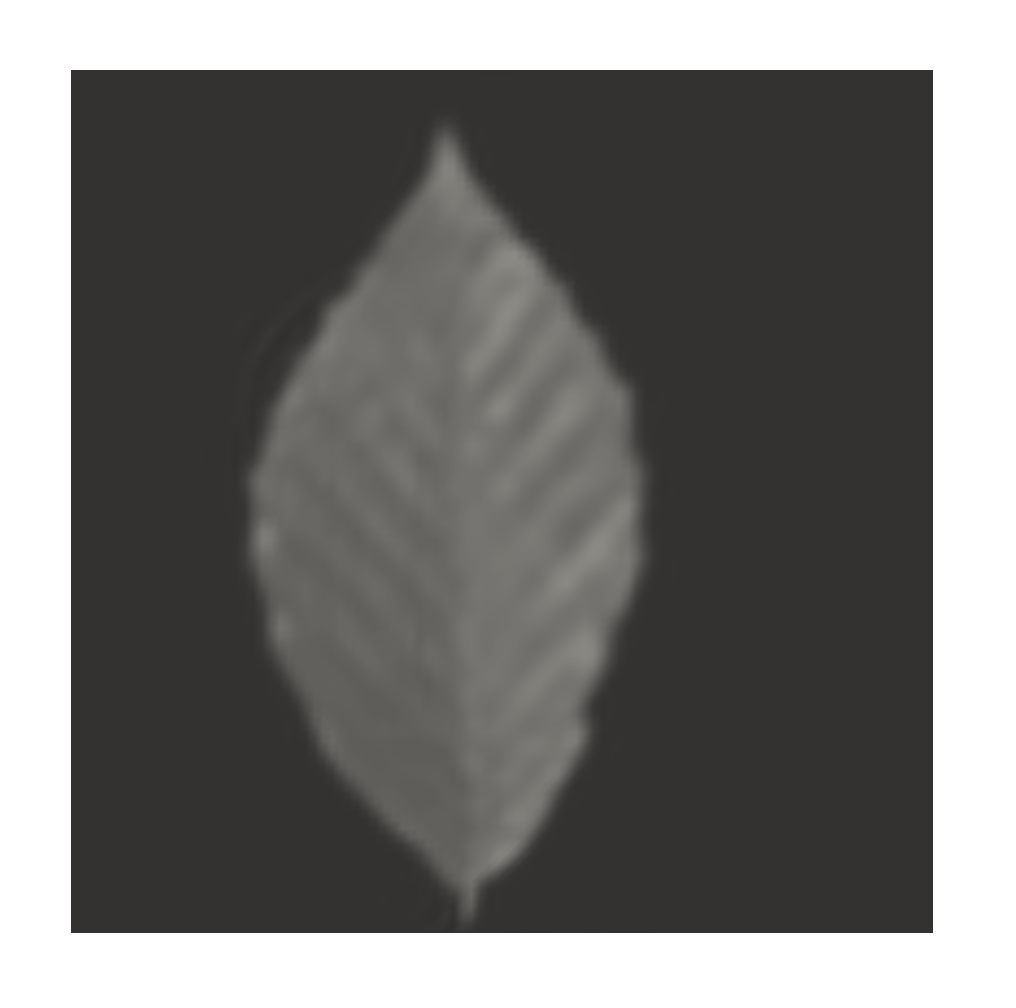}

\end{center}
\caption{Matching of two leaves from LeafSnap database using metamorphosis.}
\label{fig.leaf100}
\end{figure}

\begin{figure}[h]
 \begin{center}
\includegraphics[width=.13\textwidth]{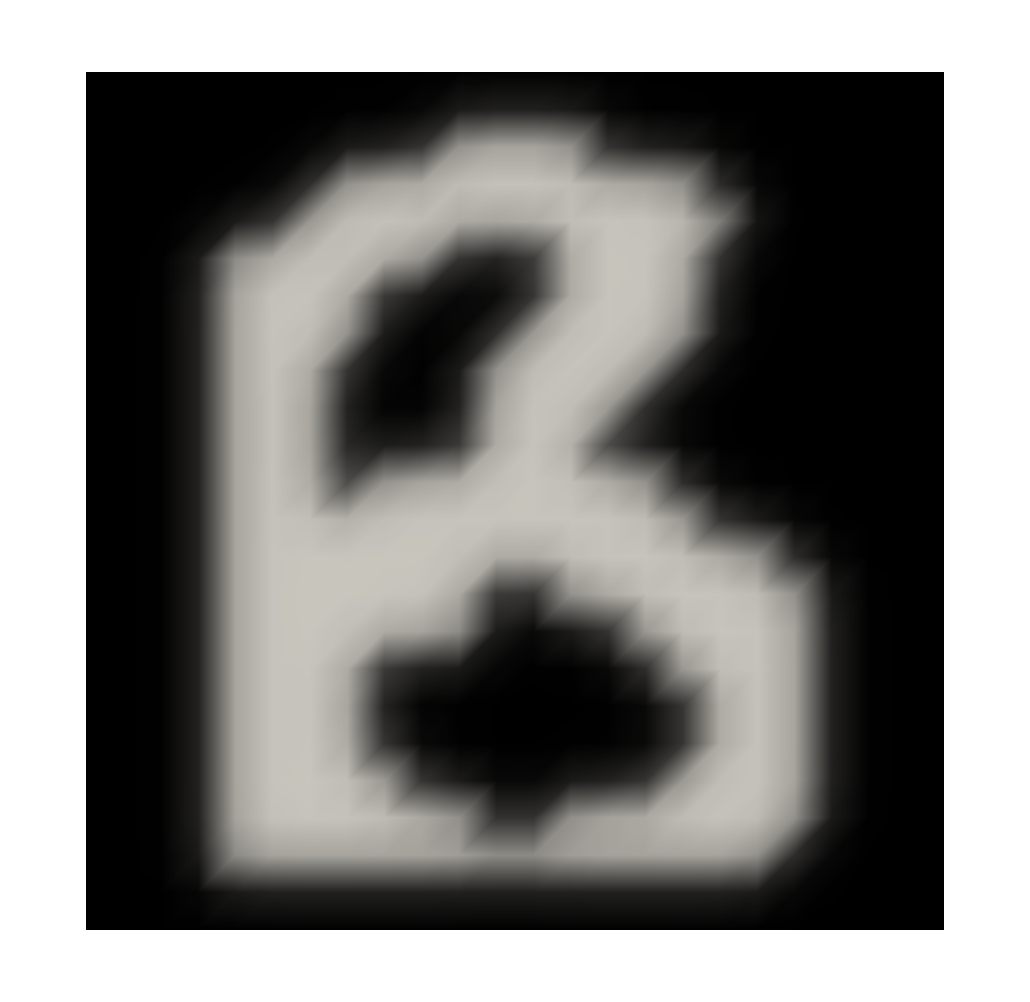}\\[.25\baselineskip]
\includegraphics[width=.13\textwidth]{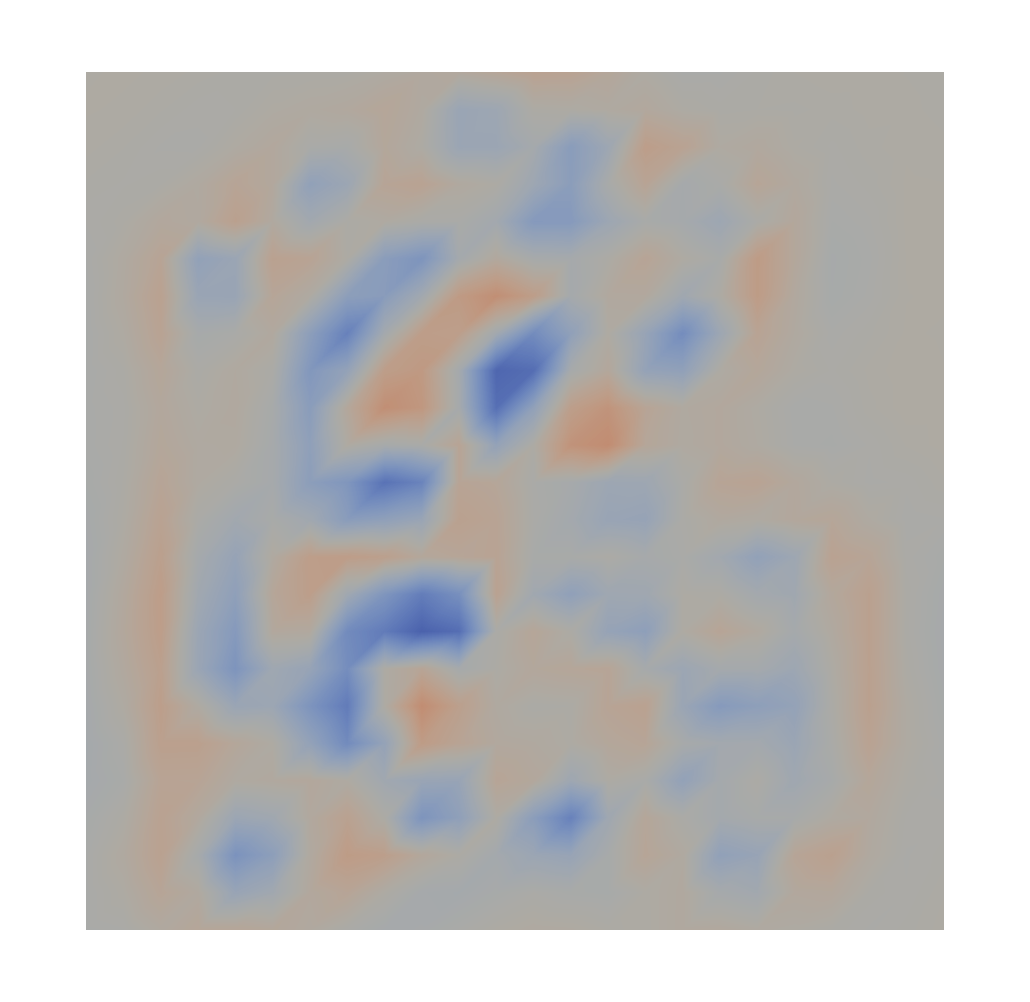}
\includegraphics[width=.13\textwidth]{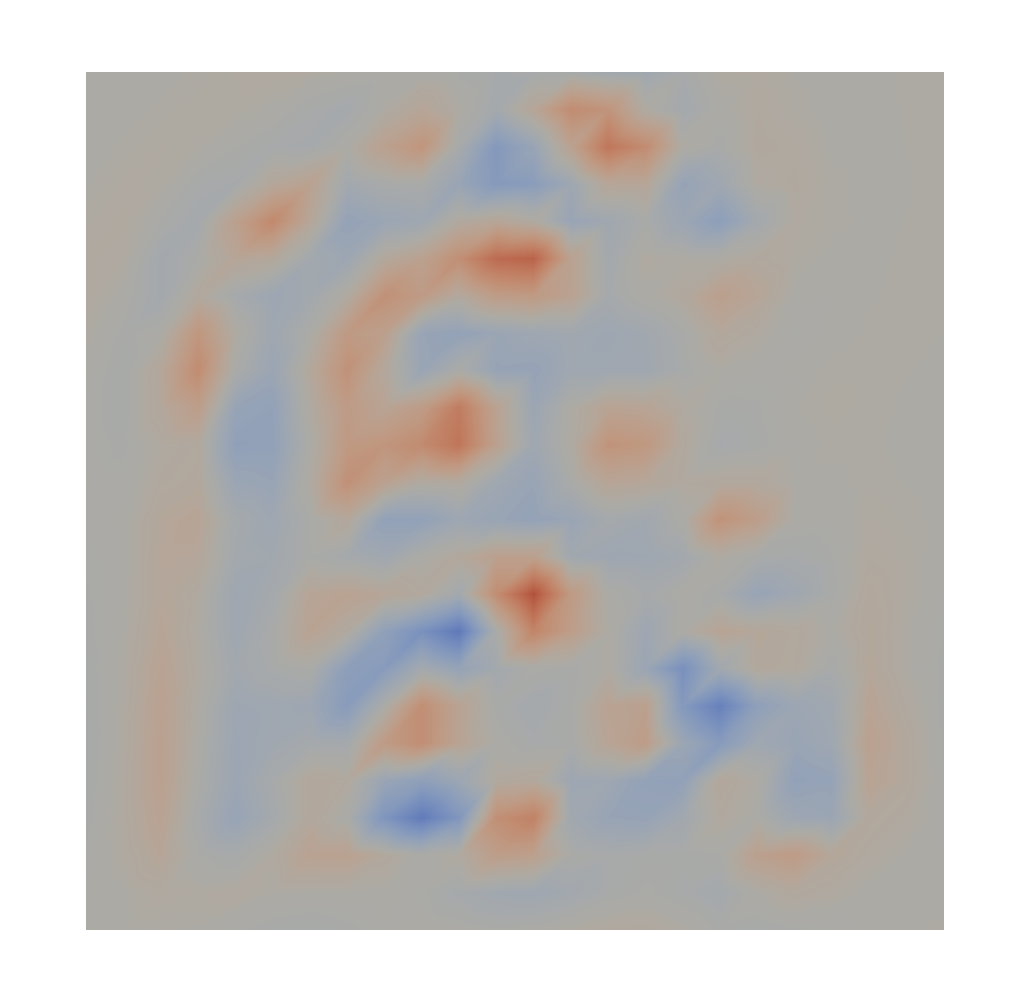}
\includegraphics[width=.13\textwidth]{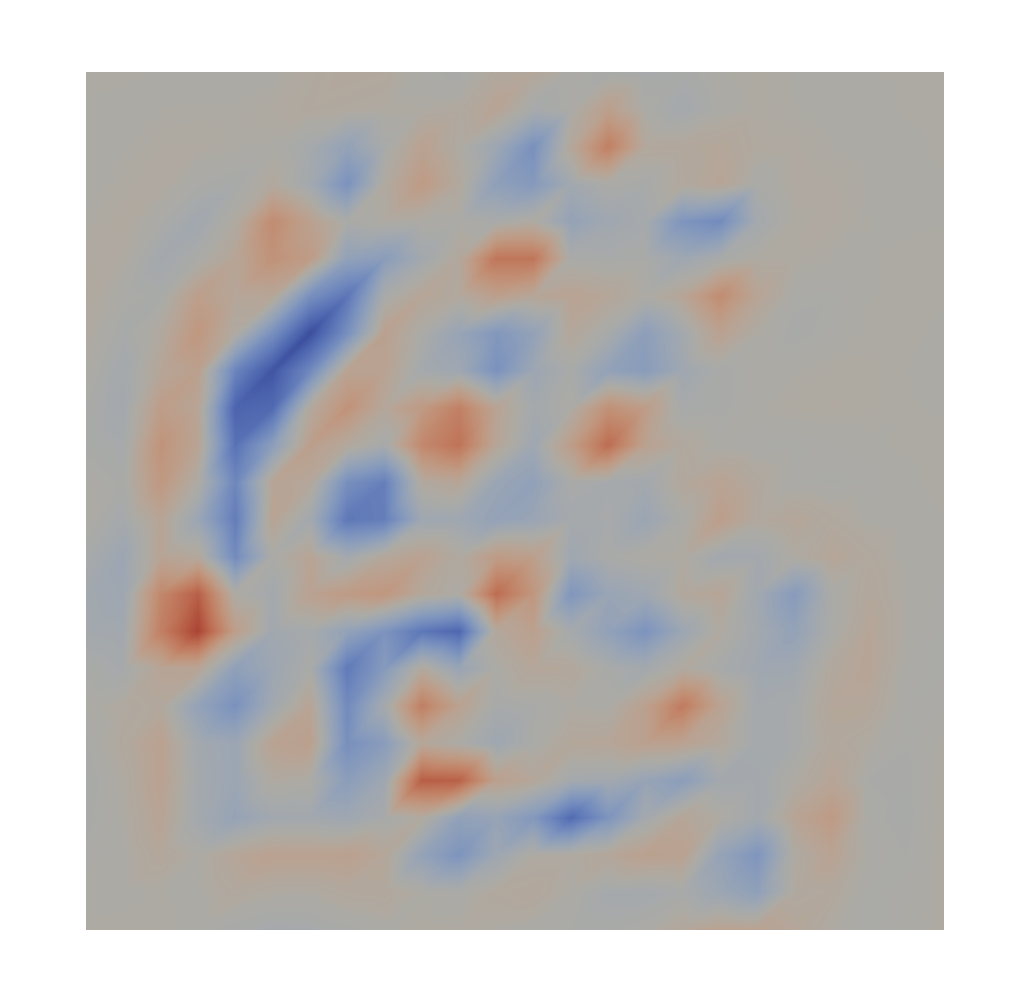}
\includegraphics[width=.13\textwidth]{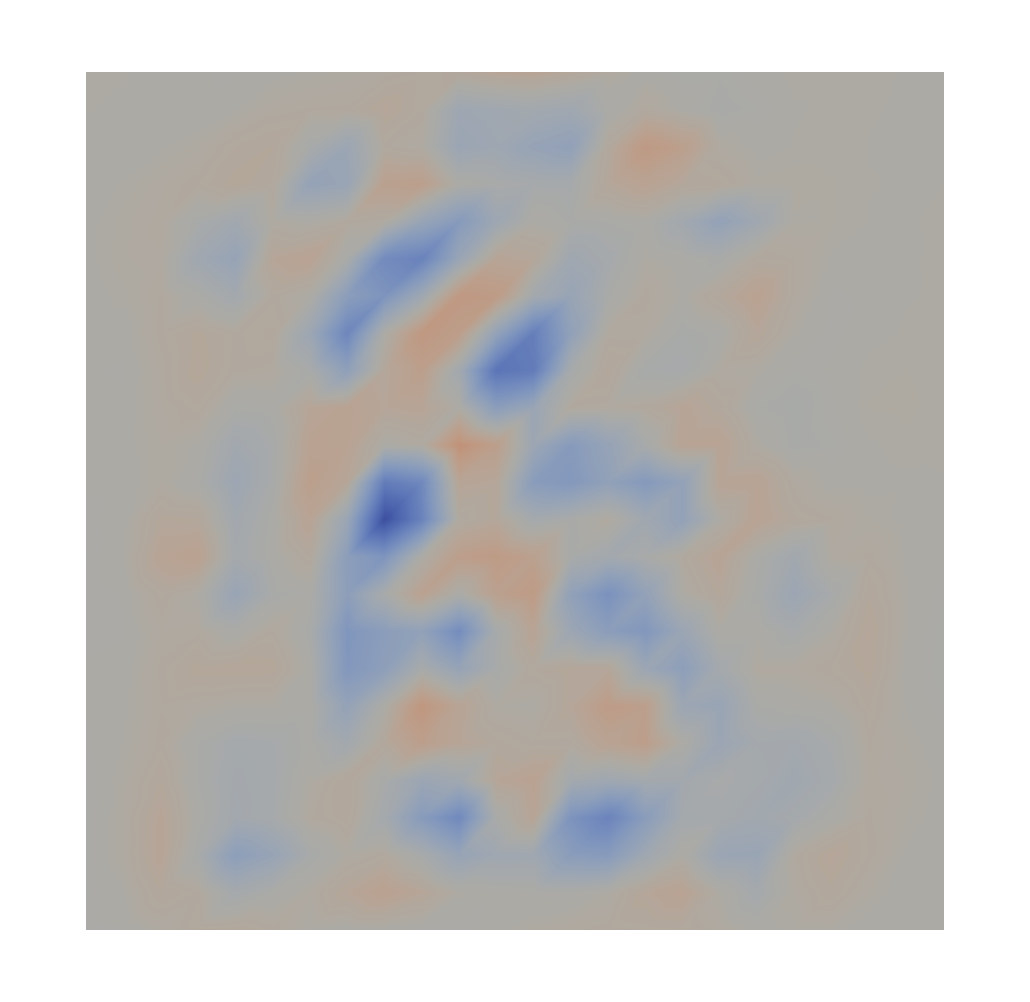}
\includegraphics[width=.13\textwidth]{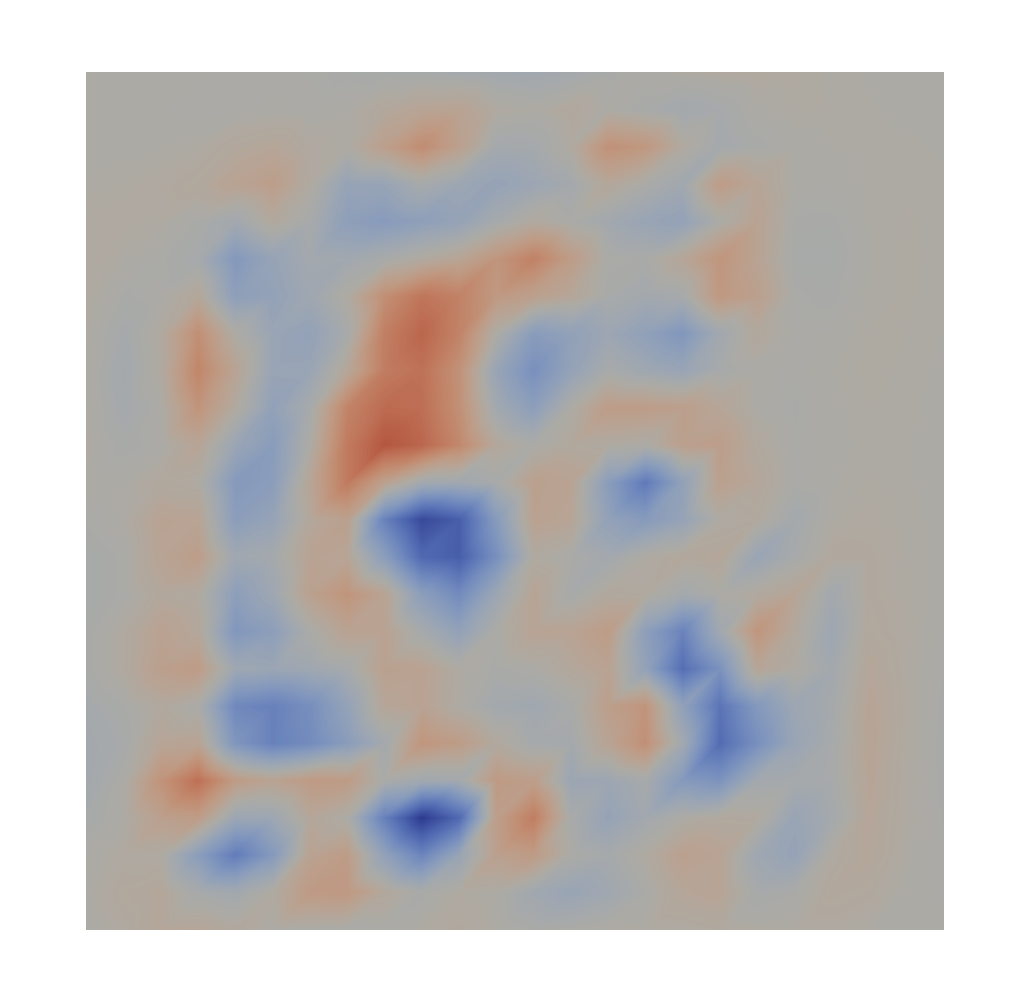}
\includegraphics[width=.13\textwidth]{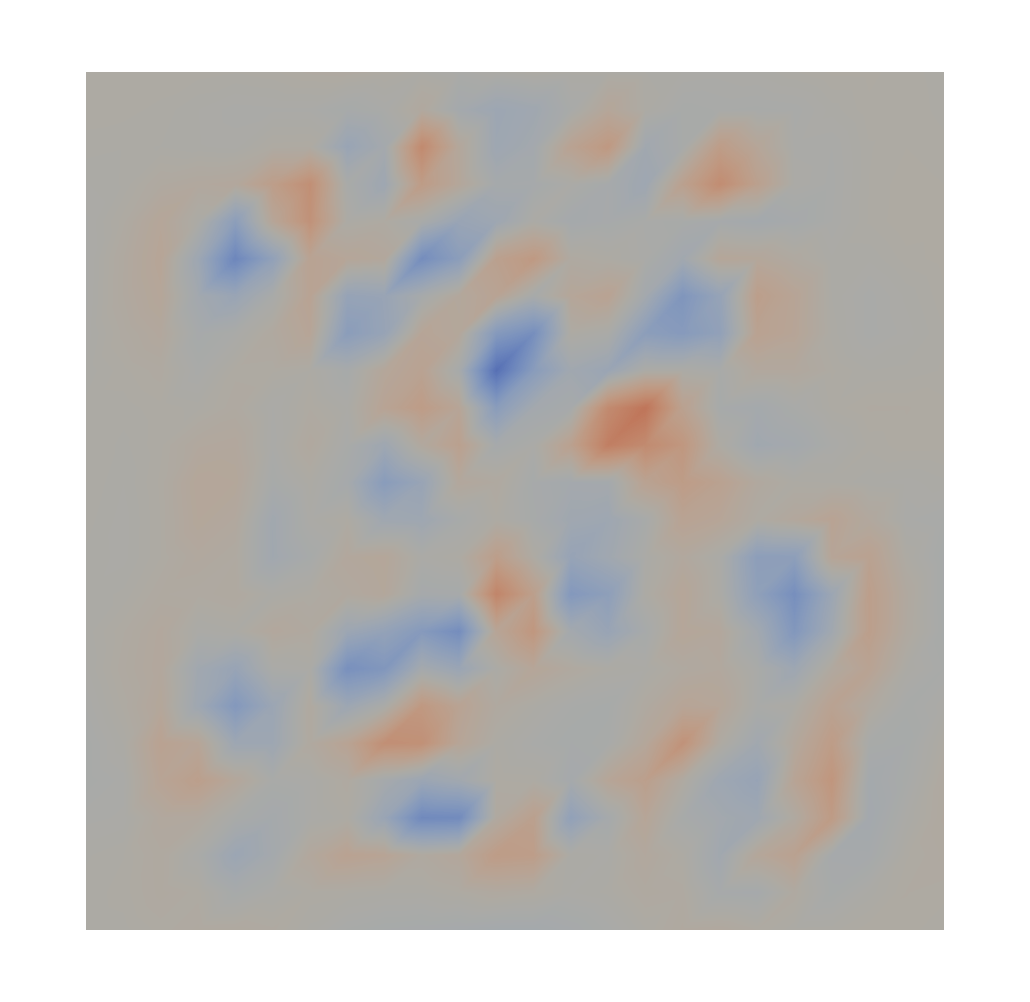}
\includegraphics[width=.13\textwidth]{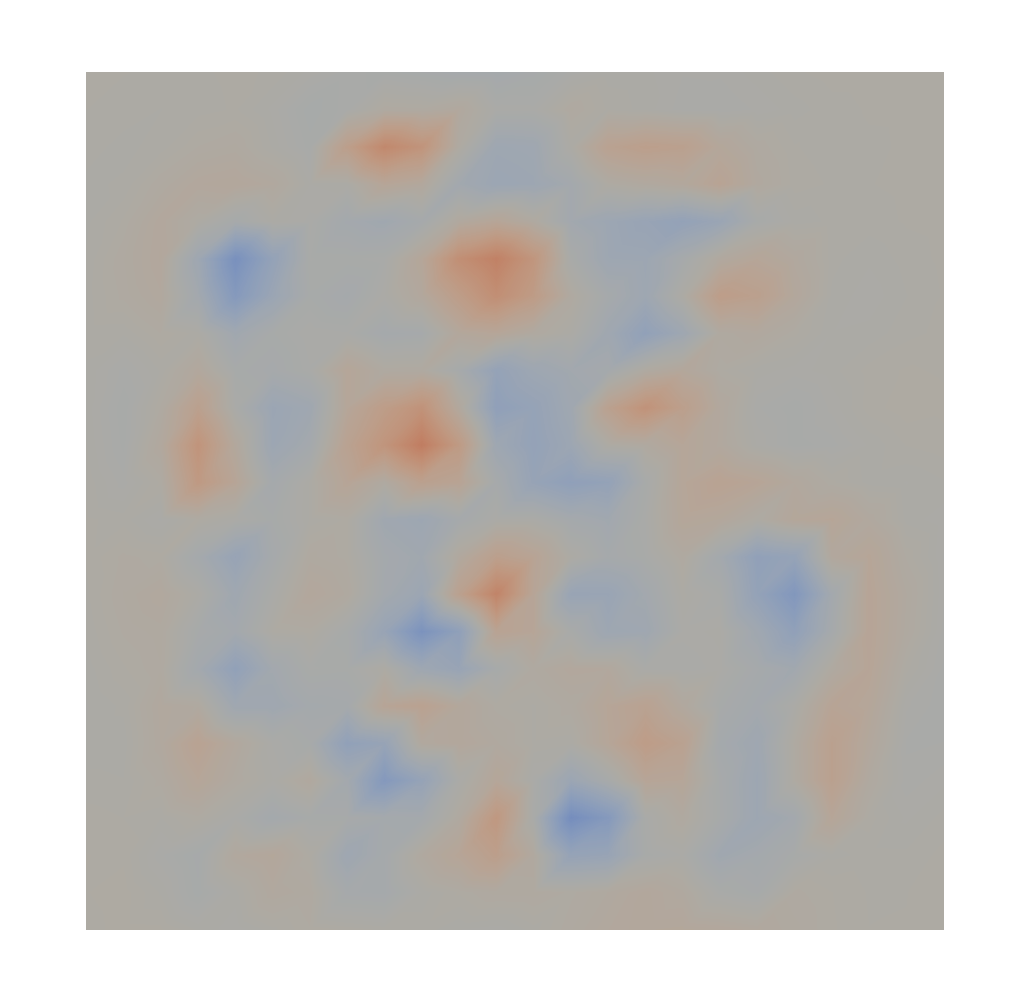}
\\[.25\baselineskip]
\includegraphics[width=.13\textwidth]{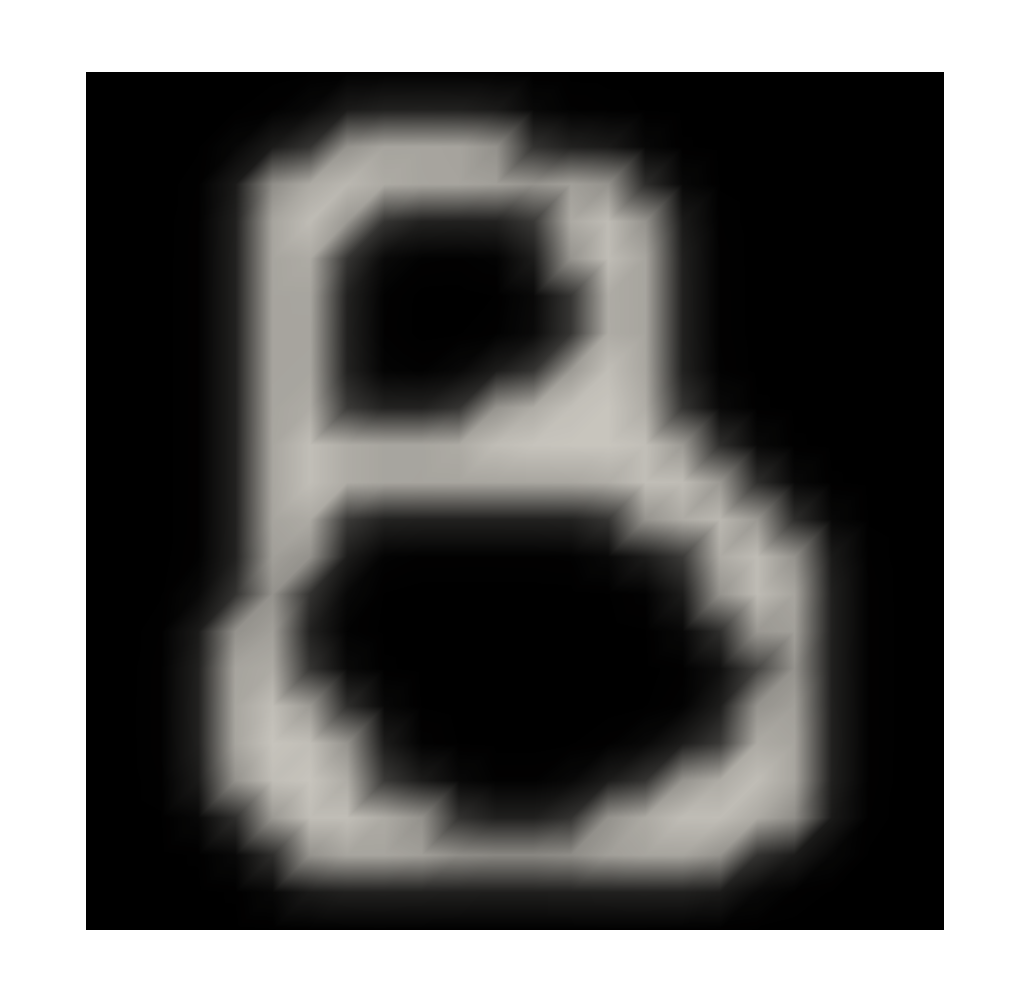}
\includegraphics[width=.13\textwidth]{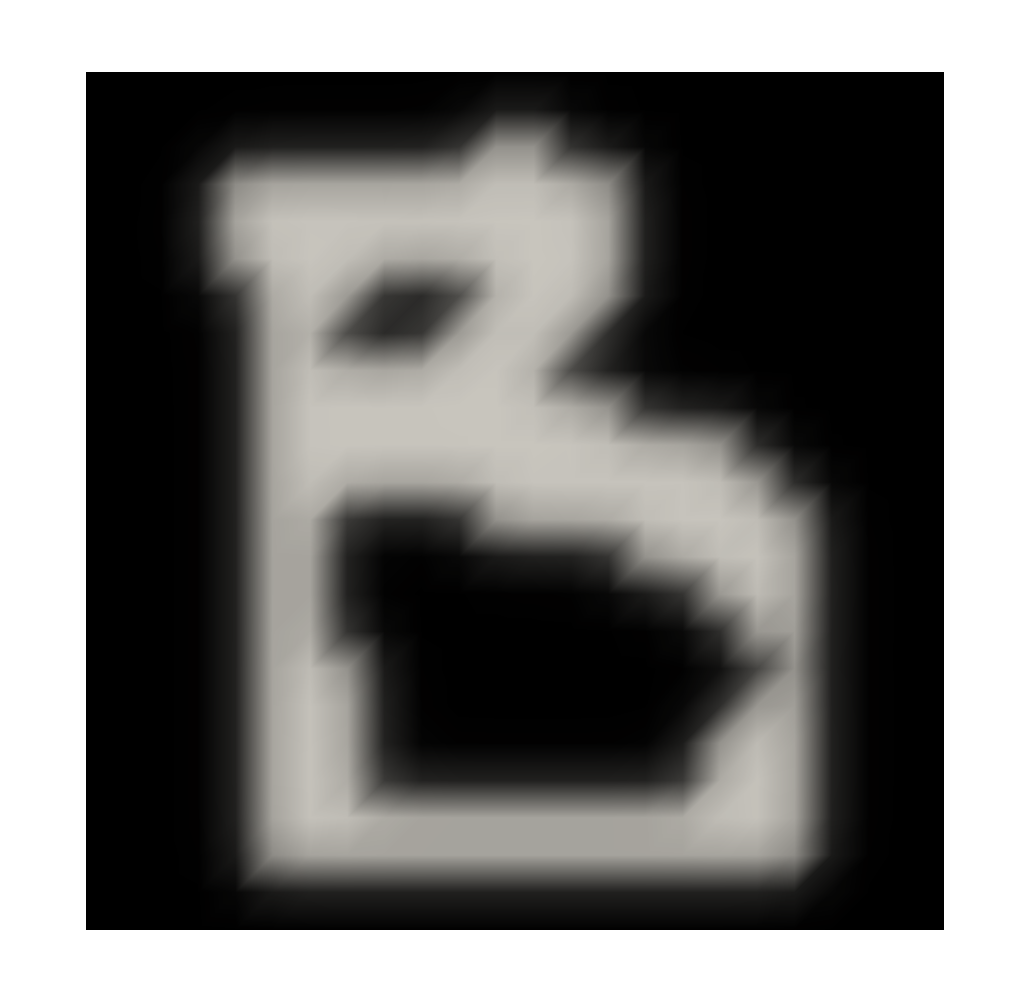}
\includegraphics[width=.13\textwidth]{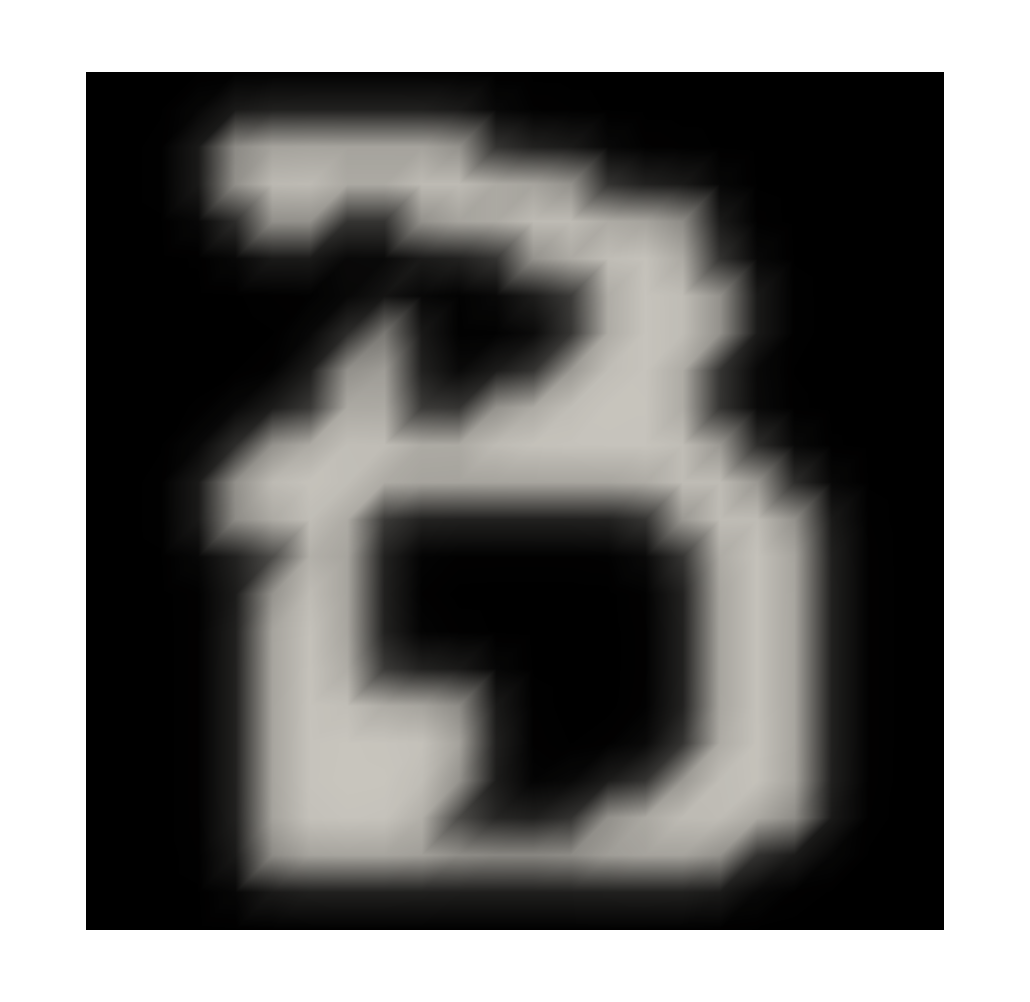}
\includegraphics[width=.13\textwidth]{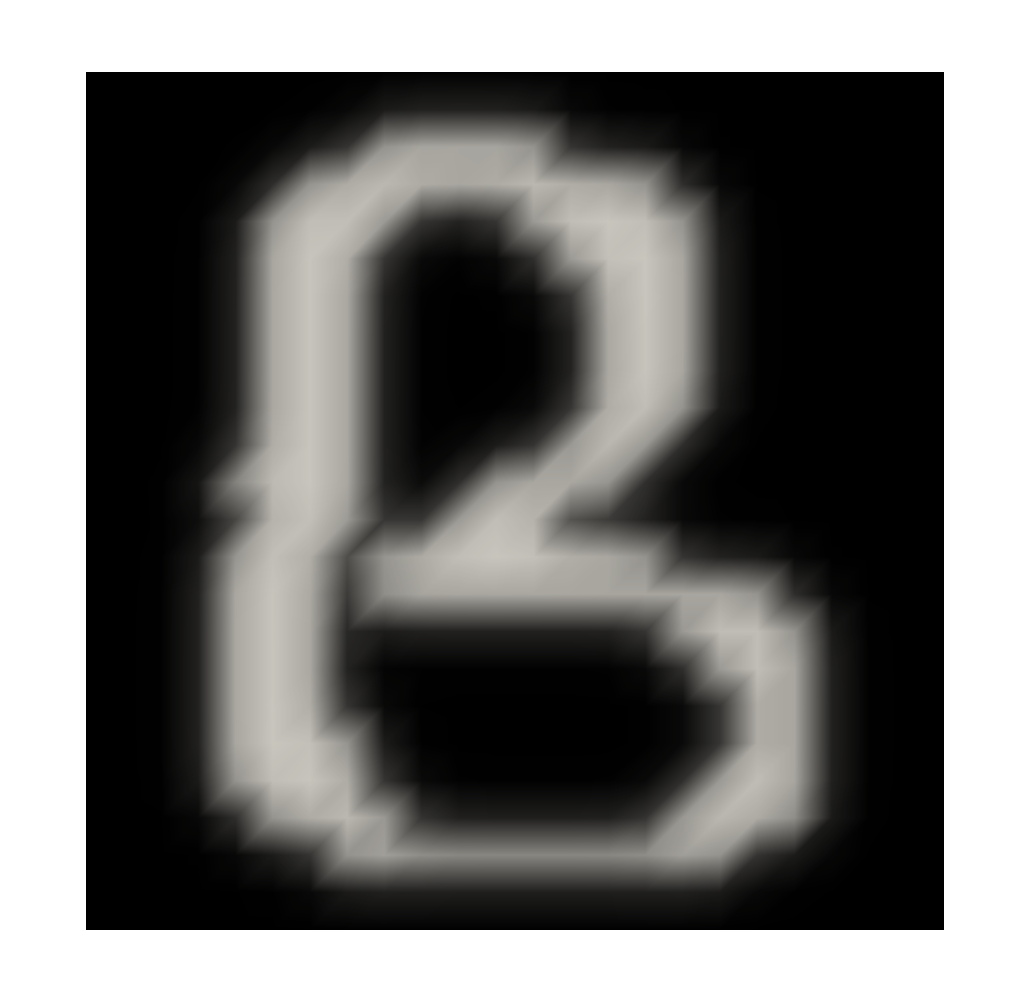}
\includegraphics[width=.13\textwidth]{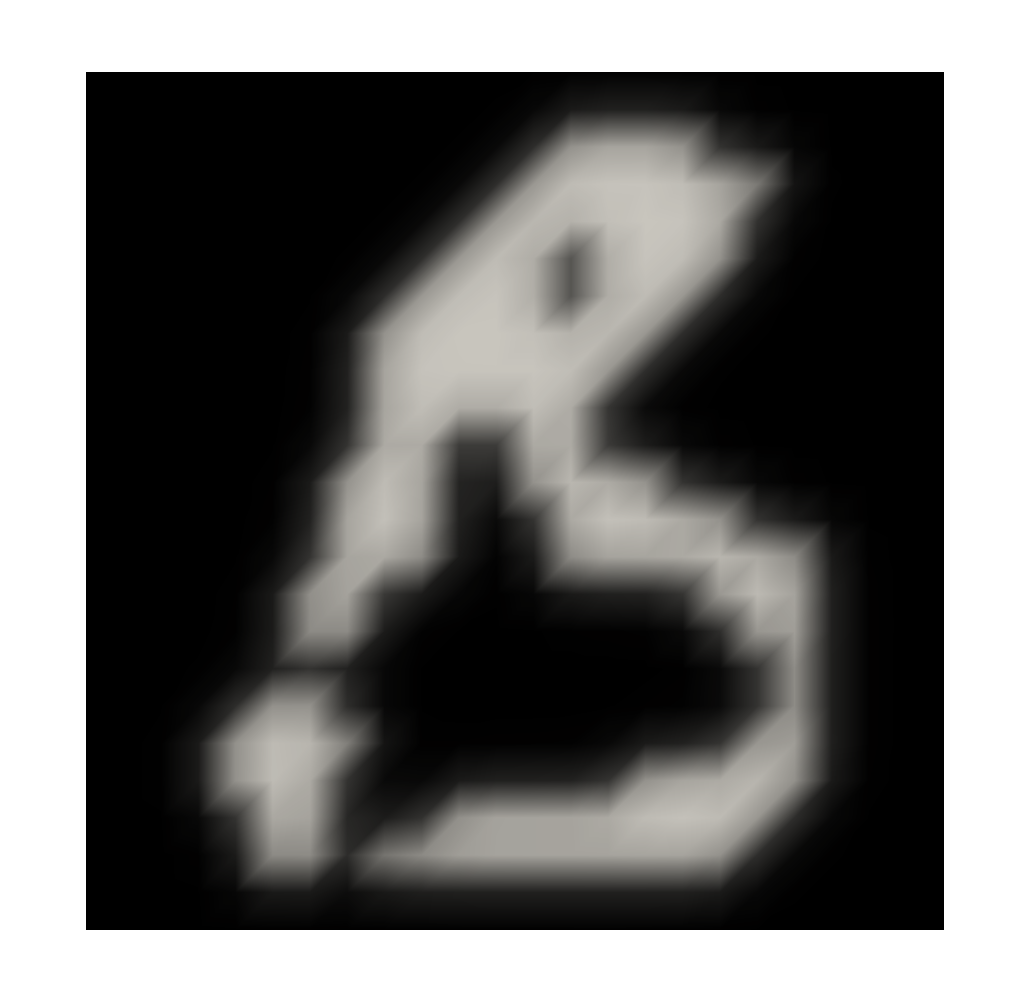}
\includegraphics[width=.13\textwidth]{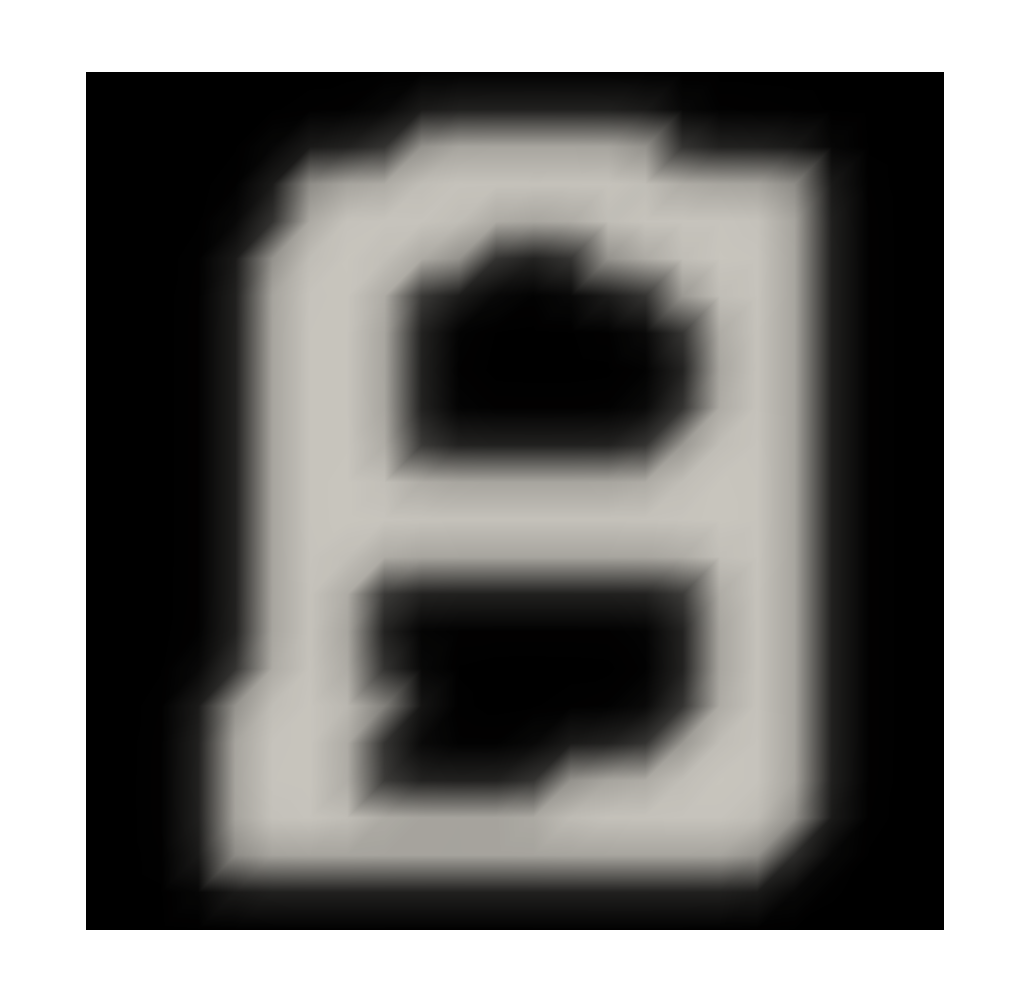}
\includegraphics[width=.13\textwidth]{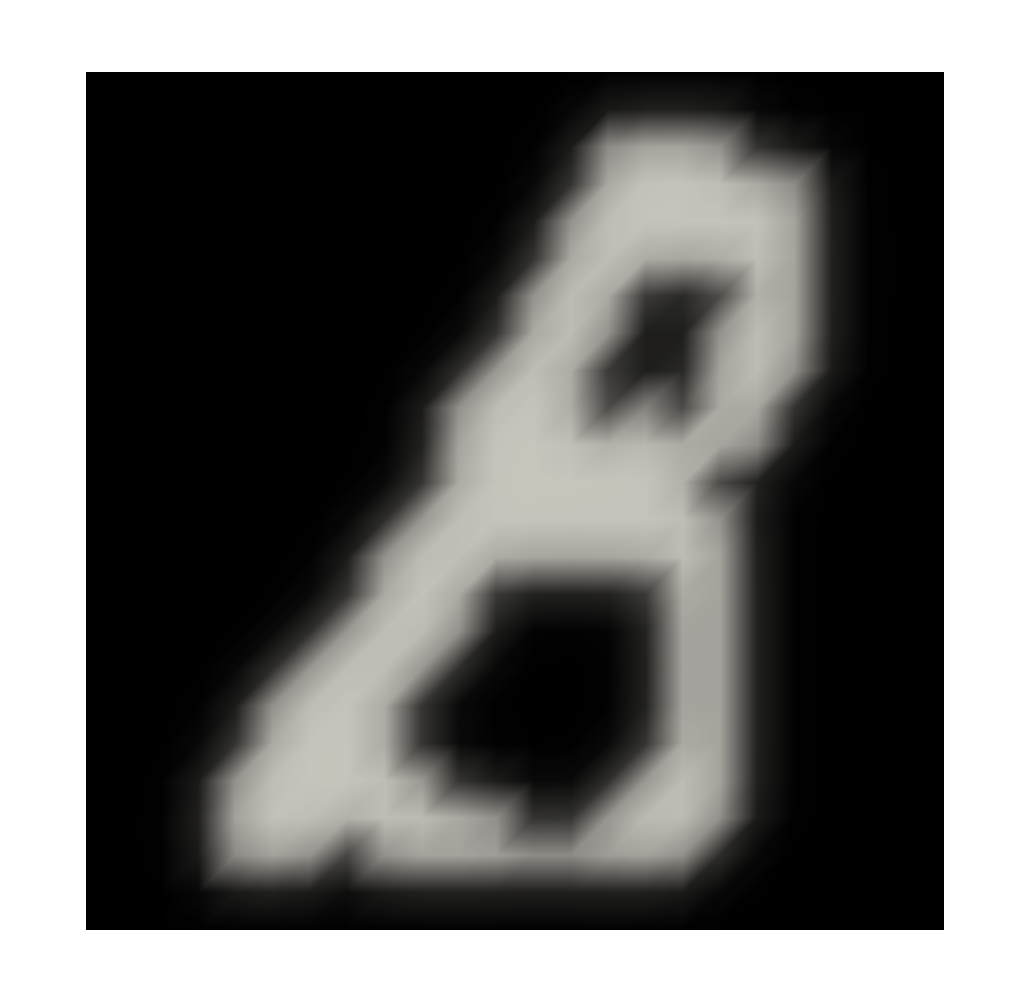}
\end{center}
\caption{Momentum field ($\alpha$) for matching template to several targets for letter B in MNIST training set;
figure in top row is the image chosen to be the template; second row shows $\alpha$ with color intensity indicating magnitude, 
negative values of $\alpha$ are colored blue, positive are red; third row shows final morphed image shooting from the $\alpha$ above.}
\label{fig.b_momenta}
\end{figure}

\begin{figure}[h]
 \begin{center}
 \includegraphics[width=0.19\textwidth]{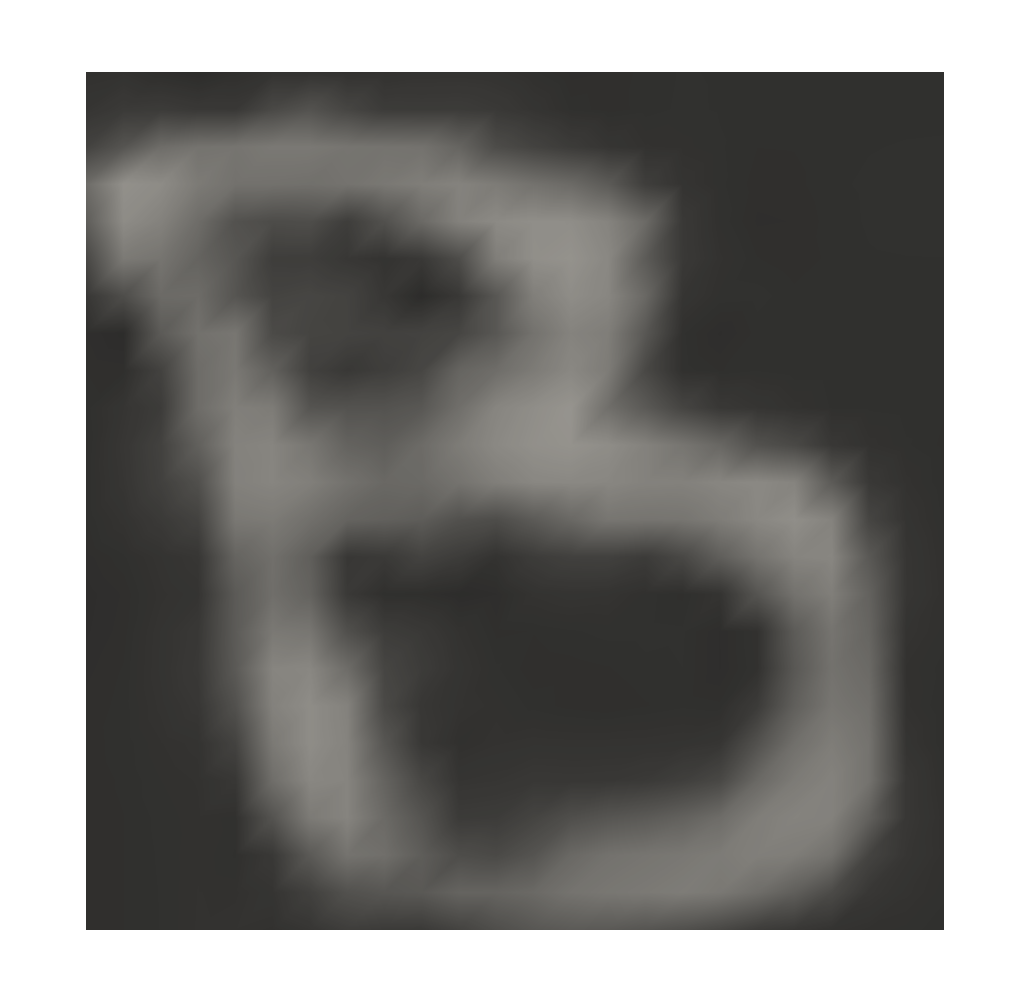}
 \includegraphics[width=0.19\textwidth]{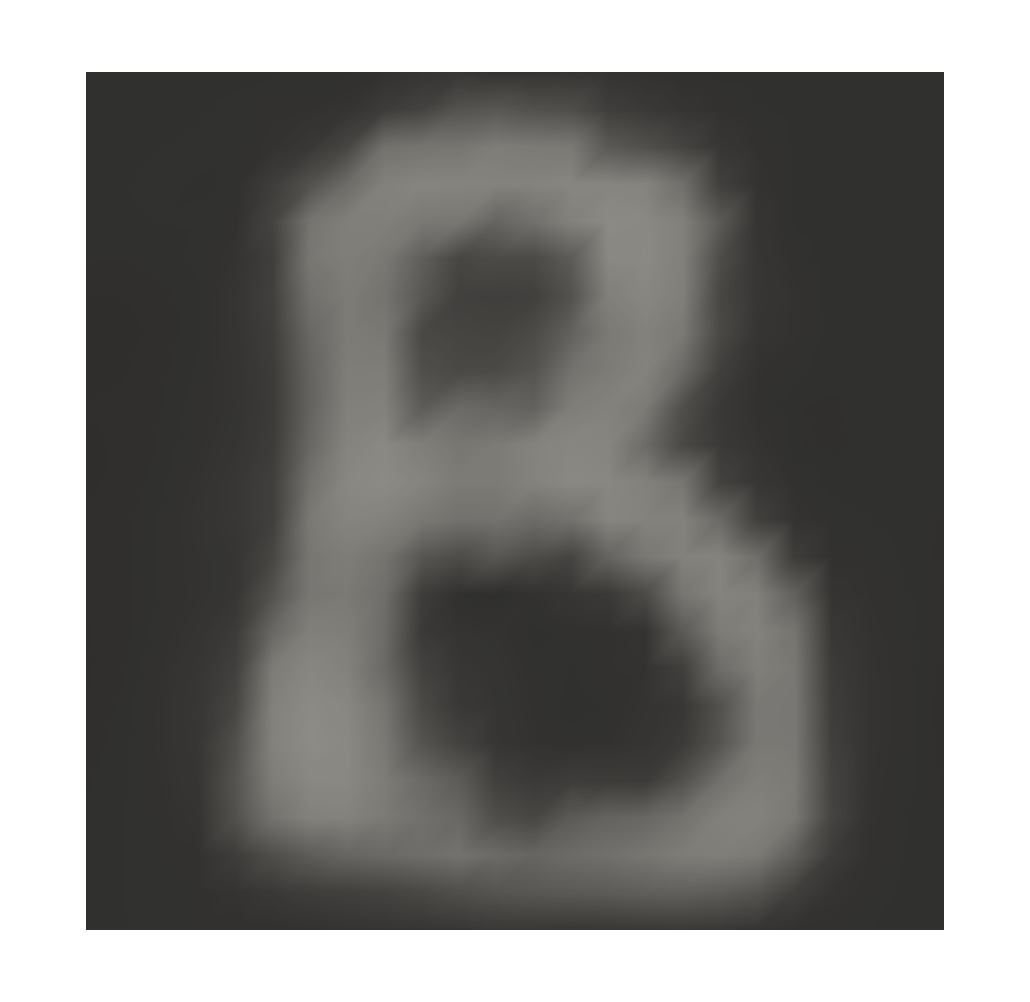}
 \includegraphics[width=0.19\textwidth]{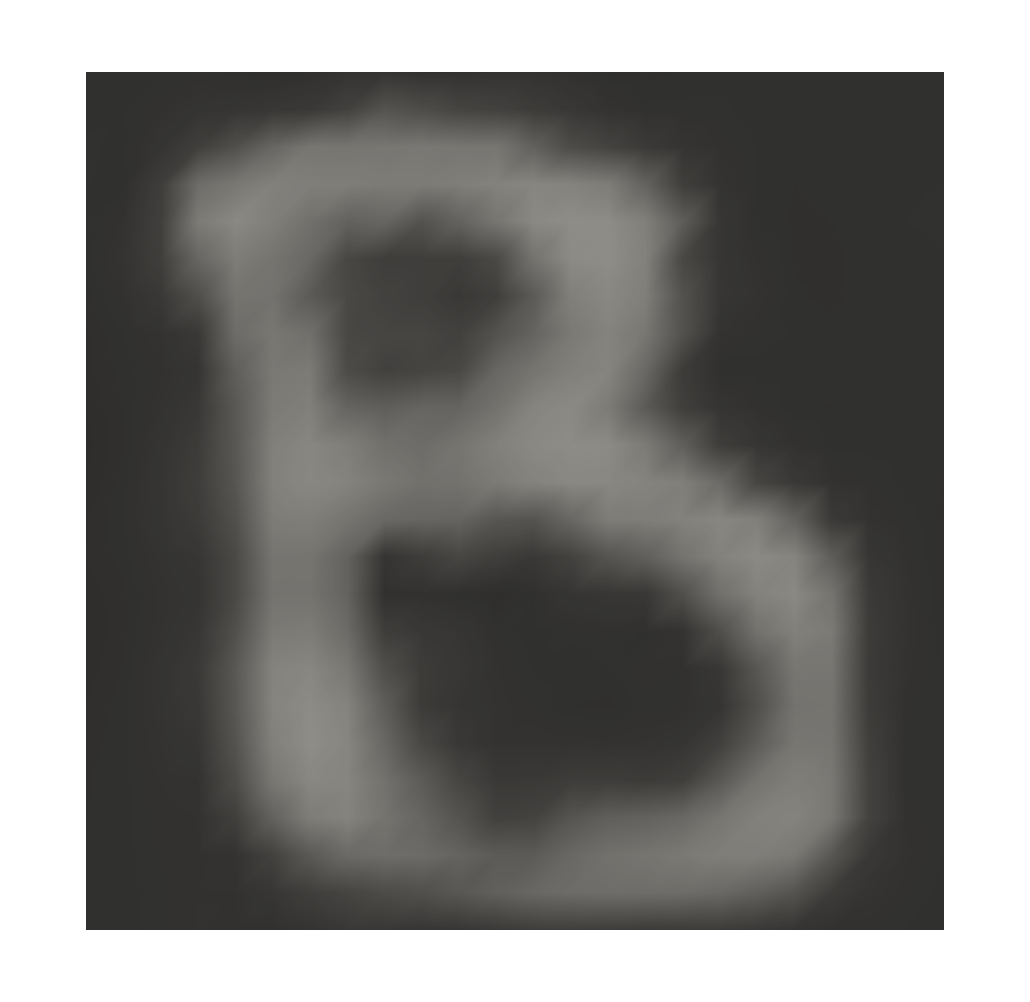}
 \includegraphics[width=0.19\textwidth]{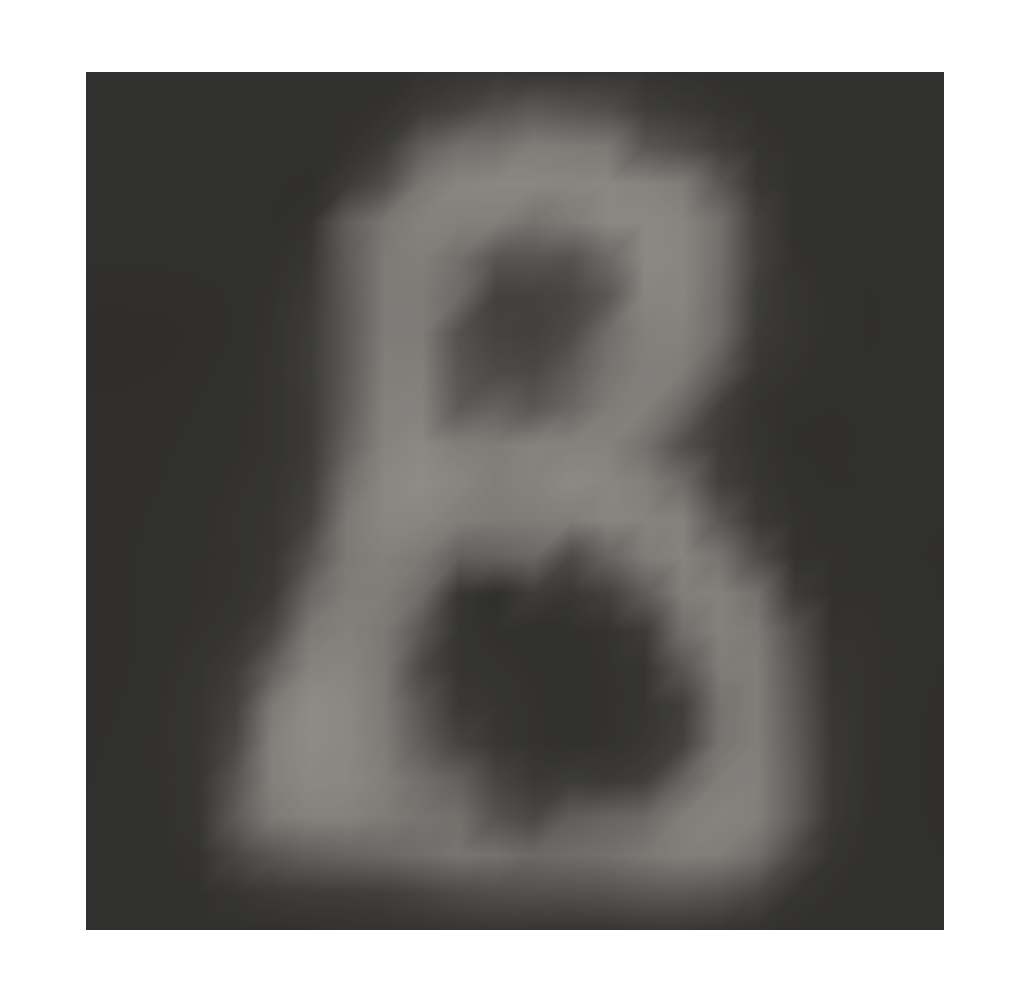}
 \includegraphics[width=0.19\textwidth]{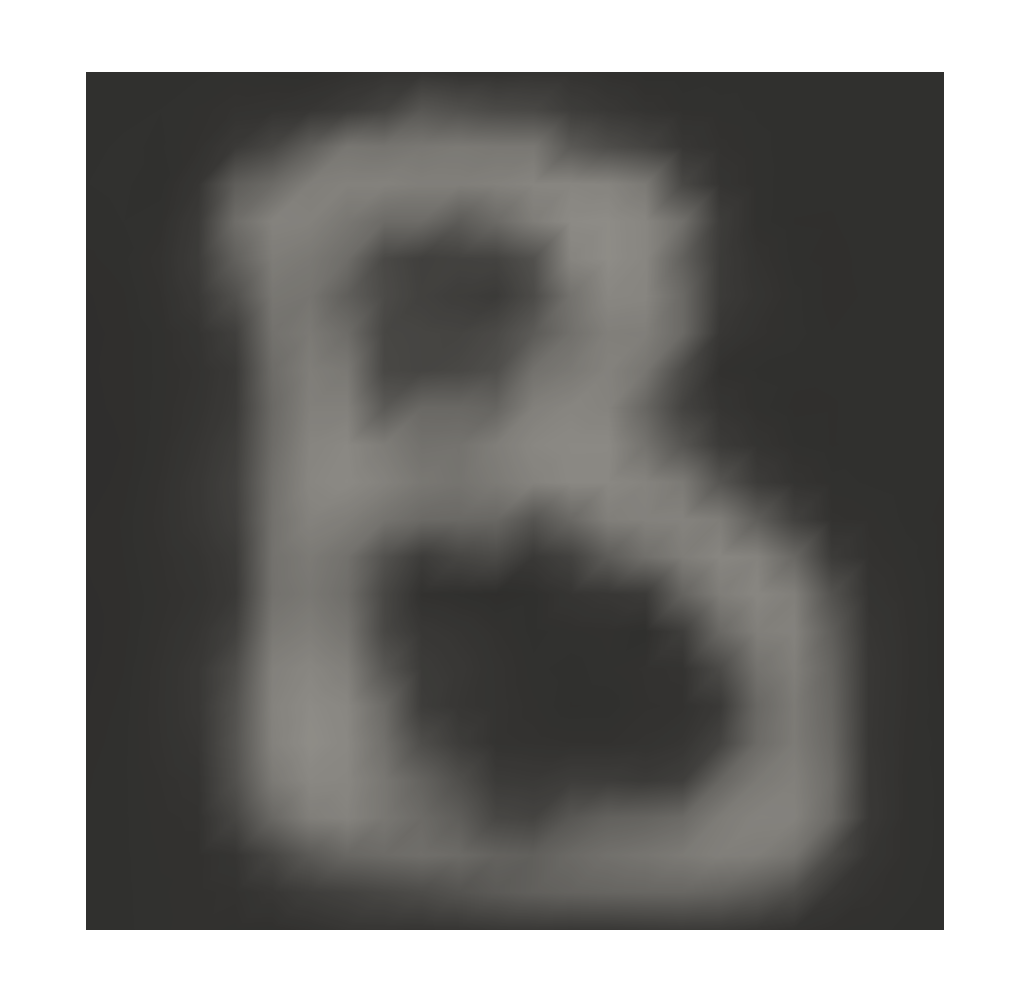}\\
 \includegraphics[width=0.19\textwidth]{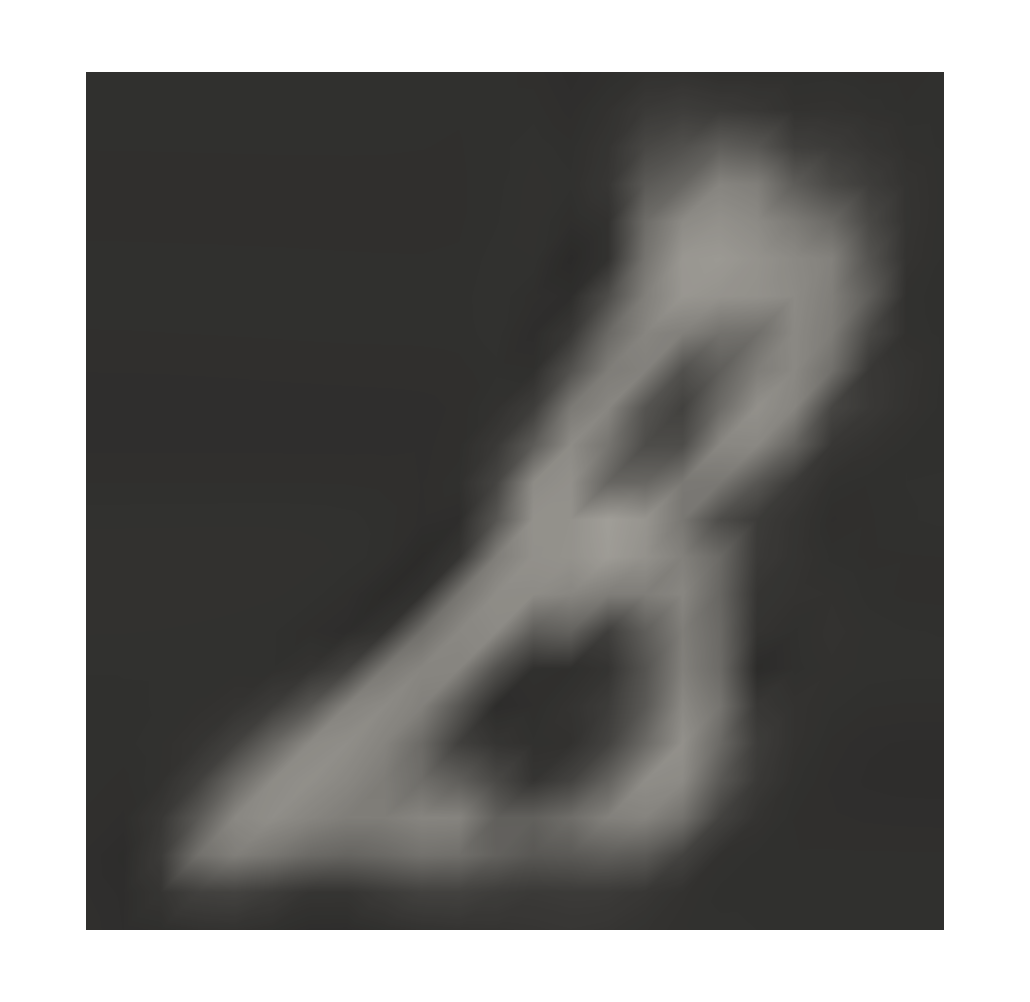}
 \includegraphics[width=0.19\textwidth]{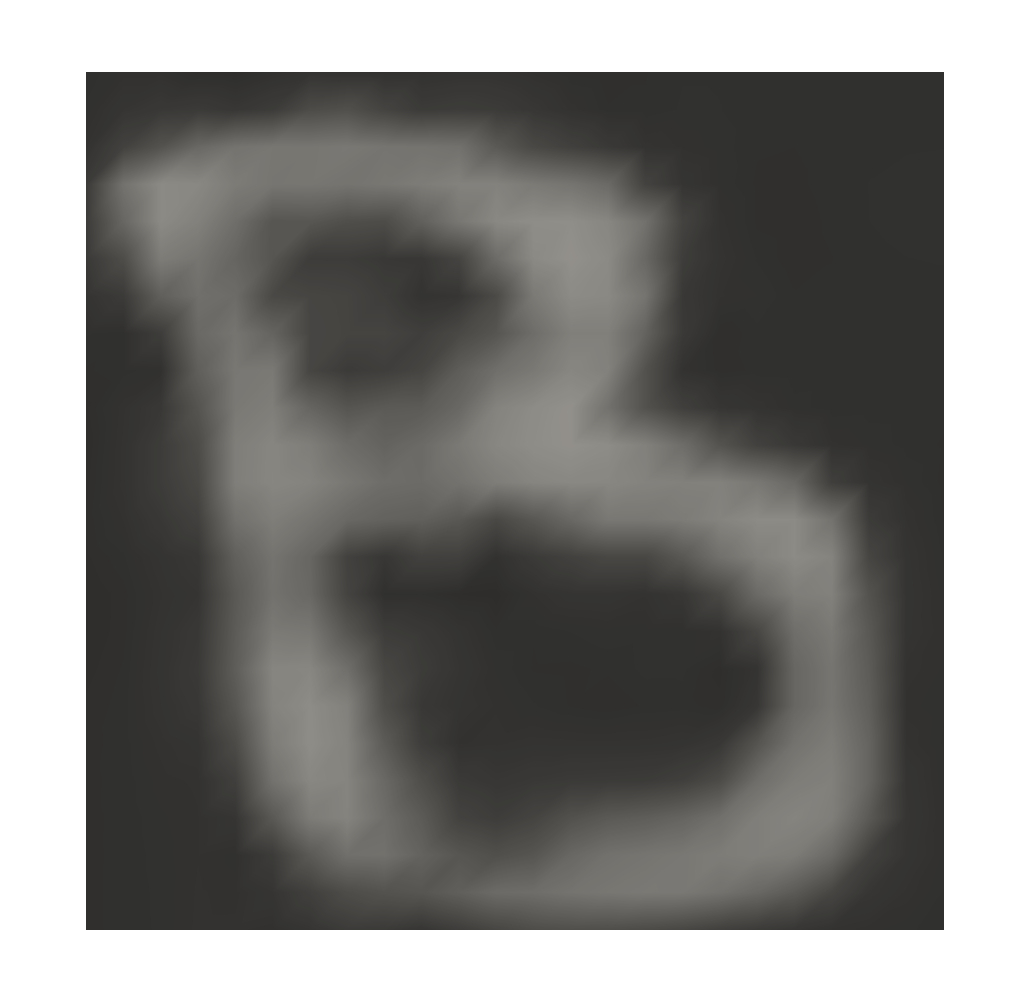}
 \includegraphics[width=0.19\textwidth]{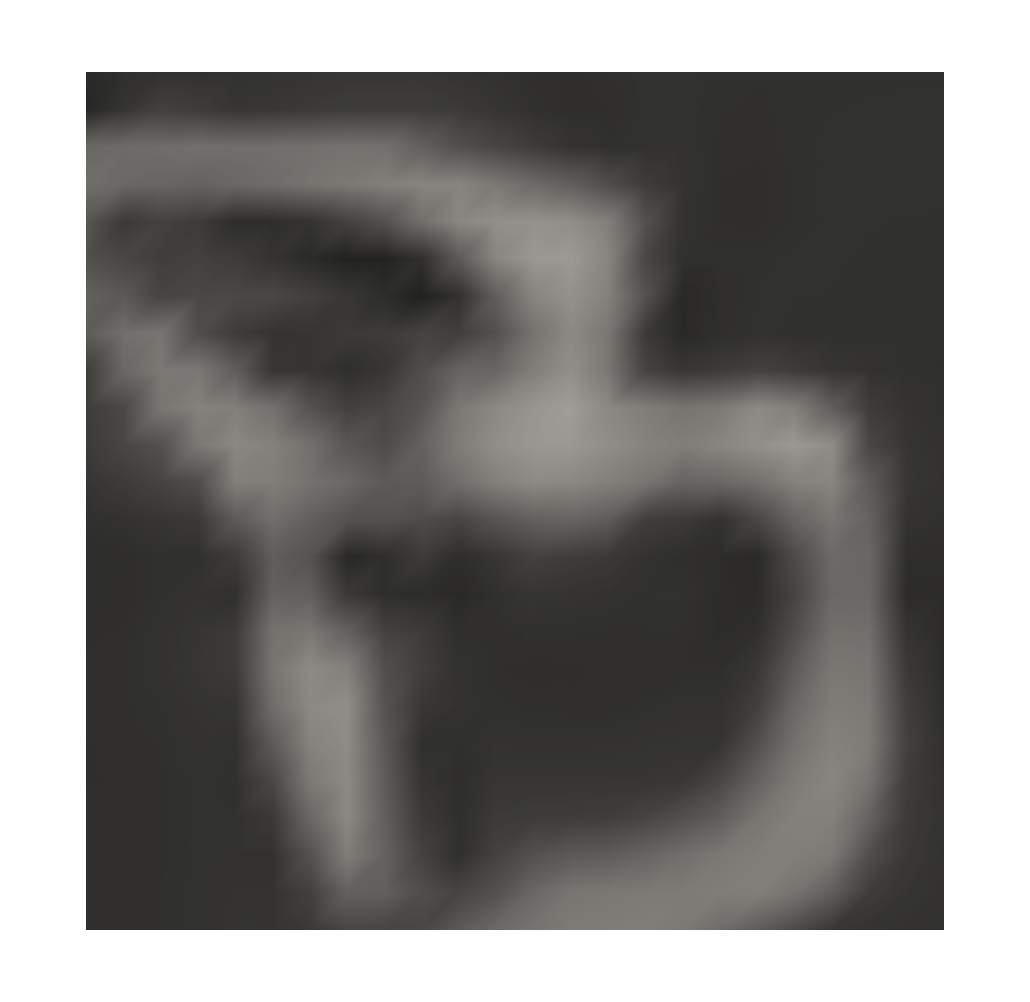}
 \includegraphics[width=0.19\textwidth]{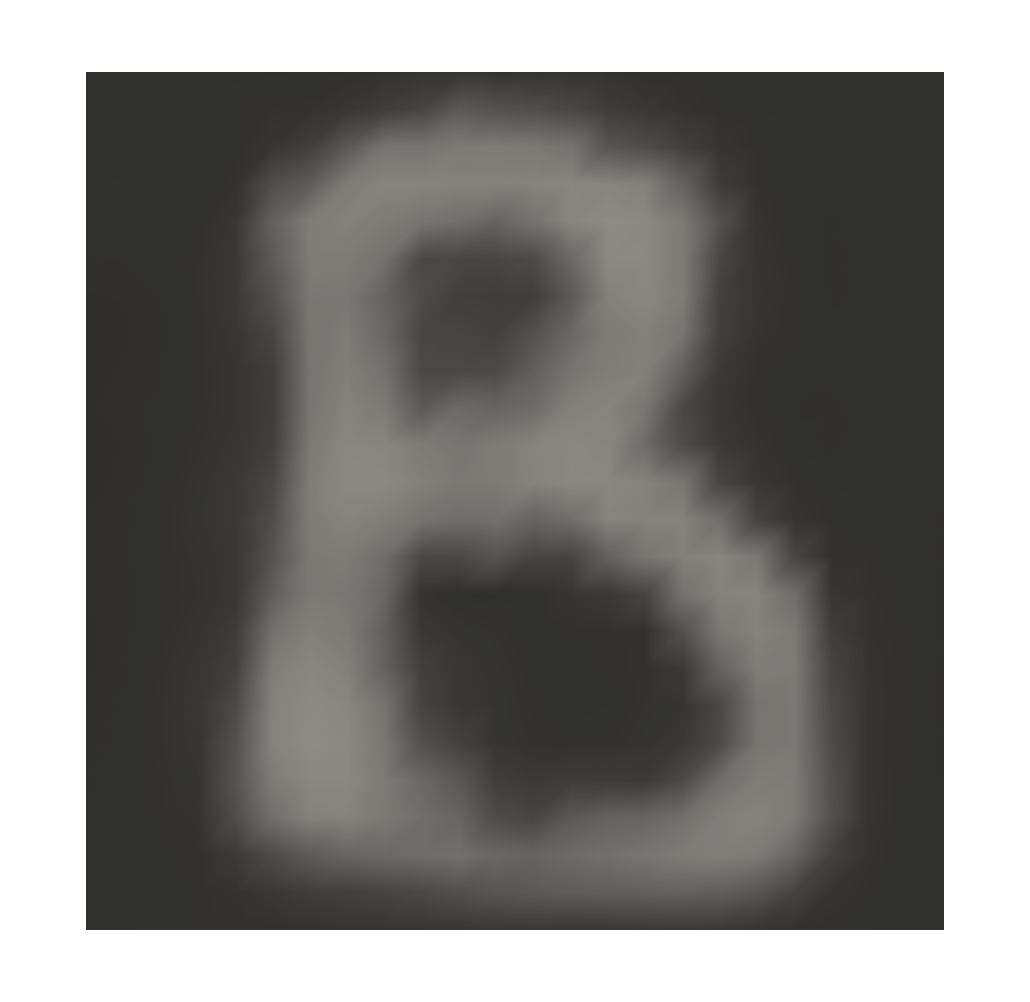}
 \includegraphics[width=0.19\textwidth]{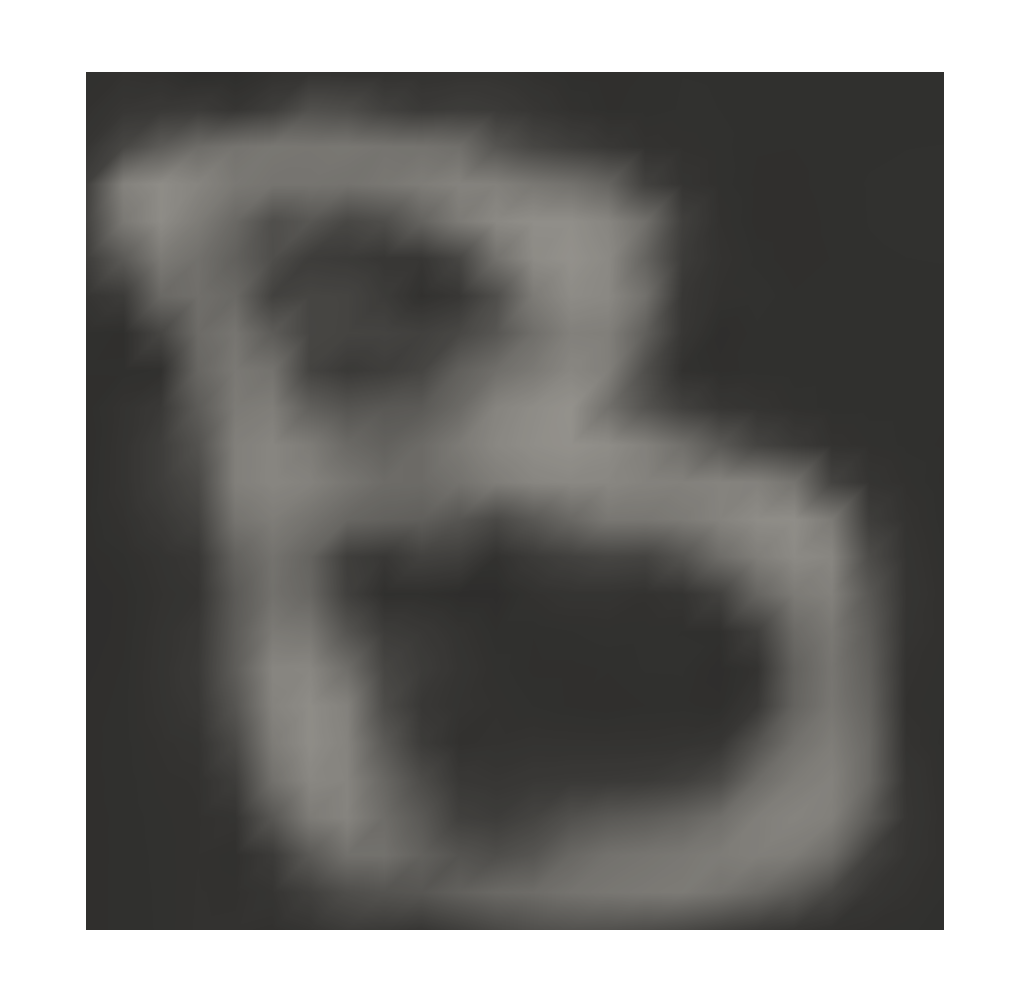}
\end{center}
\caption{Result of shooting with random momenta (described in equation \eqref{eq:rand.mom}) learned from letter matching. The figure provides ten independent samples.}
\label{fig.avg_b_momenta}
\end{figure}

\section{Rigorous Results}
\label{sec:proofs}

\subsection{Notation and Preliminary Results}
We first recall our main assumptions. Images ($m$ or $q$) belong to a Hilbert space  $H$, with norm equivalent to the $H^r(\mR^d)$ norm for some integer $r\geq 0$, with notation for the $H^r$ norm
\[
\|u\|_{r,2}^2 = \sum_{|\alpha| \leq r} \|\partial_\alpha u\|_2^2
\]
where $\alpha$ denotes a $d$-dimensional multi-index $(\alpha_1,\ldots,\alpha_d)$, $|\alpha|=\alpha_1+\cdots+\alpha_d$,
\[
\partial_\alpha u = \frac{\partial^{|\alpha|} u}{\partial^{\alpha_1}x_1\ldots\partial^{\alpha_d}x_d}
\]
and $\|\ \|_2$ is the $L^2$ norm.   We will use the usual notation $H^r(\mR^d)^* = H^{-r}(\mR^d)$. Most of the time, we will assume that $r>d/2 + k$ for some $k\geq 0$, which implies \cite{adams2003sobolev,brezis2011functional} that $H$ is continuously embedded in the space $C^k_0(\mR^d)$ of $k$-times continuously differentiable functions that vanish at infinity, together with their first $k$ derivatives, with norm
\[
\|u\|_{k,\infty} = \sum_{|\alpha| \leq k} \|\partial_\alpha u\|_{\infty}.
\]\\

We have denoted $\cB^p$ the space $C^p_0(\mR^d, \mR^d)$, with norm $\| \  \|_{p, \infty}$, and we will denote $\| \ \|_{p, \infty, *}$ the associated norm on the dual space $(\cB^p)^*$. We will assume that $V$  is a Hilbert space which is continuously embedded in $\cB^p$, with $p\geq \max(r,1)$ at least, and $p\geq r+1$ most of the time. If $v\in L^2([0,1], \cB^p)$ (which contains $L^2([0,1], V)$), the associated flow, $\varphi^v(s,t, \cdot)$, solution of $\partial_t \varphi^v = v(t, \varphi^v)$ with $\varphi^v(s,s,x)=x$ takes values in $\Diff^p(\mR^d)$, the group of diffeomorphisms $\psi$  such that $\psi-\id$ and $\psi^{-1}-\id$ both belong to $\cB^p$. More precisely\cite{agrachev2004control,ty_2005,younes_book_2010}, there exists a continuous function $c$ such that, for all $s,t\in [0,1]$,
\[
\|\varphi^v(s,t, \cdot) -\id\|_{p,\infty}  \leq c\left(\|v\|_{L^2([s,t], \cB^p)}\right)\|v\|_{L^2([s,t], \cB^p)}\,.
\]
In the following, we will use the generic notation $c(\cdot)$ to represent some continuous function of its arguments (the actual function can change from an equation to another, even if we still denote it $c$). The notation $\text{cst}$ will denote a generic constant.

The mapping $v \mapsto \varphi^v(s,t, \cdot)$ is differentiable from $L^2([0,1], \cB^p)$ to $\Diff^{p-1}$ with derivative
\[
\partial_v \varphi^v(s,t, \cdot) . h = \int_s^t D\varphi^v(u,t, \varphi^v(s,u, \cdot)) h(u, \varphi^v(s,u,\cdot)) du.
\]
Moreover, one can show that, if $v, \tilde v\in L^2([s,t], \cB^p)$, then
\[
\|\varphi^v(s,t, \cdot) -\varphi^{\tilde v}(s,t, \cdot)\|_{p-1,\infty}  \leq c\left(\|v\|_{L^2([s,t], \cB^p)}, \|\tilde v\|_{L^2([s,t], \cB^p)}\right)\|v-\tilde v\|_{L^2([s,t], \cB^p)}\,.
\]
Note that $\|v\|_{L^2([s,t], \cB^p)}$ is bounded, up to a multiplicative constant, by $\|v\|_{L^2([s,t], V)}$. 

Finally, we note that weak convergence of a sequence $v_n$ to a limit $v$ in $L^2([0,1], V)$ implies that $\varphi^{v_n}$ converges to $\varphi^v$ in the $(p, \infty)$ norm  over compact subsets of $\mR^d$ \cite{dupuis_miller_1998,younes_book_2010}. 

To simplify our expressions, we will simply denote $\varphi(t, x) = \varphi(0,t, x)$ when $s=0$.\\

We let $\Diff_V\subset \Diff^p$ denote the group of diffeomorphisms that can be obtained from flows associated to some $v\in L^2([0,1], V)$. For $\psi\in \Diff_V$, we  introduced the translation operators $\bT_\psi: V \to \cB^p$ and $\tilde \bT_\psi: H \to H$ defined by $\bT_\psi v = v\circ \psi$ and $\tilde \bT_\psi h = h\circ \psi$. The fact that $\tilde \bT_\psi$ maps $H$ onto itself (with $(\tilde \bT_\psi)^{-1} = \tilde \bT_{\psi^{-1}}$) is a consequence of $H$ being equivalent to $H^r(\mR^d)$ and of $p\geq r$ (see justification below).  The following lemma, which can be proved by induction, describes how $\tilde \bT_\psi$ commutes with partial derivatives.
\begin{lemma}
\label{lem:diff.comp}
Let $\alpha$ be a multi-index. Assume that $z:\mR^d \to\mR$ has at least $|\alpha|$ continuous derivatives, and let $\psi\in \Diff^p$ with $p\geq |\alpha|$.
One can write $\partial_\alpha(z\circ \psi^{-1})\circ \psi$ in the form
\[
\partial_\alpha(z\circ \psi^{-1})\circ \psi(y) = \sum_{\beta \leq \alpha} Q^\alpha_\beta(\psi)(y) \partial_\beta z(y)
\]
where $Q^\alpha_\beta (\psi)(y)$ depends on derivatives of $\psi$ at $y$, and can be written as a sum of terms
\[
\sigma(D\psi) (\partial_{\gamma_1} \psi_{j_1})^{\ell_1}\cdots (\partial_{\gamma_k} \psi_{j_k})^{\ell_k}
\]
with $|\gamma_q| > 1$ for $q=1, \ldots, k$ and $|\beta| + \sum_{q=1}^k \ell_q (|\gamma_q|-1) \leq |\alpha|$. In this expression, $\psi_j$ denotes the $j$th coordinate of $\psi$ and $\sigma$ is a continuous function of $D\psi$, which can be expressed as the ratio of a polynomial in the coefficients of $D\psi$ divided by $|\det D\psi|$ to some power.
\end{lemma}

This result (or a similar version of it) can be found in many places in the literature: see \cite{ebin1970groups,ebin1970manifold,inci2013regularity} and their references. This lemma implies, in particular, that 
\[
\sum_{|\alpha|\leq r} |\partial_\alpha \tilde \bT_{\psi^{-1}} z|^2 \leq c(\|\psi-\id\|_{p, \infty}) \sum_{|\alpha|\leq r} \tilde \bT_{\psi^{-1}} |\partial_\alpha  z|^2,
\]
from which one obtains the continuity of $\tilde \bT_{\psi^{-1}}$, with the operator norm $\|\tilde \bT_{\psi^{-1}}\|_{\mathcal L(H,H)}$   a continuous function of $\|\psi-\id\|_{p, \infty}$. \\

We will use the following result. Assume that $\psi_n$ is a sequence of diffeomorphisms of $\mR^d$ that converges pointwise to a diffeomorphism $\psi$, and such that $\|\psi_n^{-1}\|_{1, \infty}$ is bounded. Then, for any $z\in L^2(\mR^d)$, $z\circ \psi_n$ converges in $L^2$ to $z\circ\psi$. This can be proved by using the fact that for any $\varepsilon>0$, one can find a compact subset of $\mR^d$, $A_\varepsilon$, such that $z$ is continuous on $A_\varepsilon$, $A_\varepsilon^c = \mR^d\setminus A_\varepsilon$ has measure less than $\varepsilon$ and $\|z\mathbf 1_{A_\varepsilon^c}\|_2 \leq \varepsilon$. Assume without loss of generality that $\psi = \id$ and write
\[
\|z\circ \psi_n - z\|^2_2 = \int_{\mR^d} z^2 (|\det D(\psi_n^{-1})| -1) dx + \int_{A_{\varepsilon}} z(z\circ \psi_n-z) dx
+\int_{A^c_{\varepsilon}} z(z\circ \psi_n-z) dx.
\]
The last integral is less than $\|z\circ \psi_n-z\|_2 \|z\mathbf 1_{A_\varepsilon^c}\|_2\leq \mathrm{cst}.\|z\|_2 \varepsilon$ and the rest can be made arbitrarily small by letting $n$ go to infinity.

This result combined with Lemma \ref{lem:diff.comp} implies that, if $z\in H^r(\mR^d)$, then $\psi \mapsto \tilde \bT_\psi z$ is continuous in $\psi$ as a function from $\Diff^p$ ($p\geq r$) to $H^r(\mR^d)$ (or $H$). More generally,  if $\psi_n\in \Diff^p$ and its $p$ first derivatives converge to those of $\psi\in \Diff^p$ pointwise, with $\|\psi_n^{-1}\|_{p, \infty}$ bounded, then $\tilde \bT_{\psi_n} z$ converges to $\tilde \bT_\psi z$  in  $H^r(\mR^d)$. Finally
$\tilde \bT_\psi$ is, in addition, differentiable in $\psi$ in the following setting. If $z \in H^{r+1}(\mR^d)$, then $\psi \mapsto \tilde \bT_\psi z$ is differentiable, as a function from $\Diff^p$ to $H$, with differential $u \mapsto \nabla z \circ \psi \cdot u$. 
%More precisely, $\tilde \bT_\psi$ is differentiable in $\psi$ for the operator norm $\mathcal L (H^{r+1}(\mR^d), H)$. 
Indeed, starting
with a smooth $z$, one writes
\[
\tilde \bT_{\psi+u} z - \tilde \bT_{\psi} z -(\nabla z \circ \psi)\cdot u = \int_0^1 (\nabla z \circ (\psi+\varepsilon u) - \nabla z \circ \psi) \cdot u d\varepsilon.
\]
An application of Leibnitz formula yields
\[
\|(\nabla z \circ (\psi+\varepsilon u) - \nabla z \circ \psi)\cdot u\|_{r,2} \leq \text{cst} \|\nabla z \circ (\psi+\varepsilon u) - \nabla z \circ \psi\|_{r,2} \|u\|_{r, \infty}
\]
yielding
\[
\|\tilde \bT_{\psi+u} z - \tilde \bT_{\psi} z -(\nabla z \circ \psi)\cdot u\|_{r,2} \leq \text{cst} \|u\|_{r, \infty} \int_0^1 \|\nabla z \circ (\psi+\varepsilon u) - \nabla z \circ \psi\|_{r,2} d\varepsilon,
\]
which can be extended to arbitrary $z\in H^{r+1}(\mR^d)$ by density. The conclusion then follows from the continuity of $\psi \mapsto \nabla z \circ \psi$ as an $H^r(\mR^d, \mR^d)$-valued mapping, since $\nabla z\in H^r(\mR^d, \mR^d)$. From this, it also follows that $\psi \mapsto \tilde \bT_\psi^* \rho$ is differentiable in $\psi$, for the $H^{-r}(\mR^d)$ norm, as soon as $\rho \in H^{1-r}(\mR^d)$.  \\

We will also be interested, for $\psi\in G_V$, in the operator $\bL_\psi = \tilde \bT_\psi^* \bA_H \tilde \bT_\psi$, where $\bA_H$ is, as before, the duality isometry from $H$ to $H^*$, with inverse $\bK_H$. $\bL_\psi$ provides a bounded invertible mapping from $H$ to $H^*$, and one has
\[
\lform{\bL_{\psi^{-1}} z}{z} = \|z\circ \psi^{-1}\|^2_H.
\]
Note that $\|\bL_\psi\|_{\mathcal L(H, H^*)} \leq \|\tilde \bT_\psi\|^2_{\mathcal L(H)}$ and, using  $\bL_\psi^{-1} = \tilde \bT_{\psi^{-1}} \bK_H \tilde \bT_{\psi^{-1}}^*$, $\|\bL_\psi^{-1}\|_{\mathcal L(H^*, H)} \leq \|\tilde \bT_{\psi^{-1}}\|^2_{\mathcal L(H)}$. More generally, if $\psi\in \cB^{r+k}$, then $\bL_{\psi}$ maps $H^{r+k}(\mR^d)$ to $H^{k-r}(\mR^d)$ and $\|\bL_\psi\|_{\mathcal L(H^{r+k}, H^{k-r})} \leq \text{cst} \|\tilde \bT_\psi\|_{\mathcal L(H^{r+k})} \|\tilde \bT_\psi\|_{\mathcal L(H^{r-k})}$. Similarly, $\bL^{-1}_{\psi}$ maps $H^{k-r}(\mR^d)$ to $H^{r+k}(\mR^d)$ with $\|\bL_\psi\|_{\mathcal L(H^{k-r}, H^{r+k})} \leq \text{cst} \|\tilde \bT_{\psi^{-1}}\|_{\mathcal L(H^{r+k})} \|\tilde \bT_{\psi^{-1}}\|_{\mathcal L(H^{r-k})}$.

From the differentiability of $\tilde \bT_\psi$ and $\tilde \bT_\psi^*$, one obtains the fact that $\bL_\psi z$ and $\bL_\psi^{-1}\rho$ are differentiable in $\psi$ as soon as $z\in H^{r+1}(\mR^d)$ and $\rho\in H^{1-r}(\mR^d)$ (note that $\bA_H$ maps $H^{r+1}$ onto $H^{1-r}$). One can go a little further by assuming that $p\geq r+1$ and that the norm on $H$ results from a differential operator, i.e.,
\[
\|z\|_H^2 = \Big \|\sum_{|\alpha| \leq r} b_{\alpha} \partial_\alpha z\Big\|_2^2
\]
for some coefficients $b_\alpha$. One has, in this case,
\[
\|z\circ \psi^{-1}\|_H^2 = \int_{\mR^d} \Big(\sum_{|\alpha| \leq r} b_{\alpha}\circ \psi\, \partial_\alpha (z\circ \psi^{-1})\circ \psi\Big)^2 |\det D\psi| dy
\]
and using Lemma \ref{lem:diff.comp} to expand the partial derivatives, one sees that the integrand can be written as a polynomial in the partial derivatives of $z$, with coefficients expressed as smooth functions of $\psi$ and its first $r$ derivatives. From this, one concludes that $\bL_{\psi^{-1}}$ is differentiable in $\psi^{-1}$ for the $\mathcal L(H, H^*)$ operator norm, and so is the inverse map $\bL_{\psi^{-1}}^{-1}$.     \\

Finally, let $\varphi: [0,1] \to G_V$ be a continuous mapping (e.g., $\varphi=\varphi^v$ for some $v\in L^2([0,1], \mR^d)$). Define the operator
\[
\bR_\varphi = \int_0^1 \bL_{\phi(t)^{-1}}^{-1} dt = \int_0^1 \tilde \bT_{\varphi(t)} \bK_H \tilde \bT_{\varphi(t)}^* dt,
\]
defined on $H^*$, with values in $H$. 
This operator is continuous in $\varphi$ (for $\|\varphi\| = \sup_{t\in [0,1]} \|\varphi\|_{p,\infty}$), and is invertible. To prove the last statement, first notice that $\bR_\varphi$ has closed range. Indeed, if $\bR_\varphi \rho_n \to \xi$, then $\rho_n$ is bounded because
\begin{equation}
\label{eq:Rphi.lb}
\lform{\rho_n}{\bR_\varphi \rho_n} = \int_0^1 \|\tilde\bT_{\varphi(t)}^* \rho_n\|_{H^*}^2 dt \geq \left(\int_0^1 \|\tilde\bT_{\varphi(t)^{-1}}\|_{\mathcal L(H)}^{-2} dt\right) \|\rho_n\|_{H^*}^2
\end{equation}
so that 
\[
\|\rho_n\|_{H^*} \leq  \|\bR_\varphi \rho_n\|_H \left( \int_0^1 \|\tilde\bT_{\varphi(t)^{-1}}\|_{\mathcal L(H)}^{-2} dt\right)^{-1}.
\]
This implies that $\rho_n$ 
has a weakly converging subsequence in $H^*$, say $\rho_n \rightharpoonup \rho$, which implies $\bR_\varphi \rho_n \rightharpoonup \bR_\varphi \rho$ so that $\xi=\bR_\varphi\rho$. Thus, $\bR_\varphi$ is one-to-one and has closed range, which implies that it is $\bR_\varphi$ is invertible. \\

From \eqref{eq:Rphi.lb} and a similar upper bound for the inverse, we obtain the fact that $\|\bR_\varphi\|_{\mathcal L(H^*,H)}$ and $\|\bR_\varphi^{-1}\|_{\mathcal L(H, H^*)}$ are bounded by continuous functions of $\varphi$. From this, and the identity $\bR_{\varphi}^{-1} - \bR_{\varphi'}^{-1} = \bR_{\varphi}^{-1}(\bR_{\varphi'} - \bR_\varphi)\bR_{\varphi'}^{-1}$, it follow that $\bR_\varphi^{-1}$ is also continuous in $\varphi$. The differentiability of $\bR_{\varphi}$ in $\varphi$ comes from the differentiability  of $\bL_\psi^{-1}$, so that 
$\varphi \mapsto \bR_\varphi \rho$ is differentiable as soon as $\rho\in H^{1-r}(\mR^d)$. This statement holds also for $\rho\in H^{-r}(\mR^d)$ if $\|\ \|_H$ is associated to a differential operator.  From these results and the continuity of the inverse map, one also concludes that $\bR_{\varphi}^{-1} z$ is differentiable in $\varphi$ if $z\in H^{r+1}(\mR^d)$ (or $H^r(\mR^d)$ if $\|\ \|_H$ is associated to a differential operator).

\subsection{Existence of Solutions of the Boundary-Value Problem}
We start with the existence of solutions for Problems \eqref{eq:meta.cost.2} and \eqref{eq:meta.cost.3}. 

\begin{theorem}
\label{th:exist}
Assume $r > d/2$ and $p\geq \max(1,r)$. Then
Problems \eqref{eq:meta.cost.2} and \eqref{eq:meta.cost.3} have non-empty sets of solutions.

Let $x^{(0,n)} = \{x_k^{(0,n)}\}_{k=1}^{N_n}$ be nested sets of points in $\mR^d$ such that $\bigcup_n x^{(0,n)}$ is dense in $\mR^d$. Let $(v^{(n)}, \zeta^{(n)}, \varphi^{(n)}, m^{(n)})$ be solutions of Problem \eqref{eq:meta.cost.3} with $x^{(0)} = x^{(0,n)}$. Then, possibly after replacing them with subsequences, both $v^{(n)}$ and $\zeta^{(n)}$ weakly converge to limits $v$ and $\zeta$, while $\varphi^{(n)}$ and $m^{(n)}$ converge pointwise to the corresponding $\varphi$ and $m$ such that $(v, \zeta, \varphi, m)$ is a solution of \eqref{eq:meta.cost.2}.
\end{theorem} 
\begin{proof}
Let $(v^{(n)}, \zeta^{(n)}, \varphi^{(n)}, m^{(n)})$ be a minimizing sequence for Problem \eqref{eq:meta.cost.2}. Then (using a subsequence if needed), the bounded sequences $v^{(n)}$ and $\zeta^{(n)}$ weakly converge to limits $v$ and $\zeta$ in $L^2([0,1], V)$ and $L^2([0,1], H)$ respectively, with
\[
\|v\|_{L^2([0,1], V)} \leq \liminf \|v^{(n)}\|_{L^2([0,1], V)}
\text{ and } \|\zeta\|_{L^2([0,1], H)} \leq \liminf \|\zeta^{(n)}\|_{L^2([0,1], H)}.
\]
This weak convergence for $v^{(n)}$ implies that $\varphi^{(n)}$ converges to $\varphi$ uniformly on compact sets. For $x\in \mR^d$, write
\begin{multline*}
m^{(n)}(t,x) - m(t,x) = \int_0^t (\zeta^{(n)}(s, \varphi^{(n)}(s,x)) - \zeta^{(n)}(s, \varphi(s,x))) ds \\
+ \int_0^t (\zeta^{(n)}(s, \varphi(s,x)) - \zeta(s, \varphi(s,x))) ds .
\end{multline*} 
Since the linear form
\[
\zeta' \mapsto  \int_0^t \zeta'(s, \varphi(s,x)) ds
\]
is continuous in $L^2([0,1], H)$, the last term in the right-hand side converges to 0. Recall that $K_H$ denote the reproducing kernel on $H$, defined by $K_H(\cdot, x) = \bK_H\delta_x$. Rewrite the first term as
\begin{align*}
&\int_0^t (\zeta^{(n)}(s, \varphi^{(n)}(s,x)) - \zeta^{(n)}(s, \varphi(s,x))) ds\\
& = \int_0^t \scp{K_H(\cdot, \varphi^{(n)}(s,x)) - K_H(\cdot, \varphi(s,x))}{\zeta^{(n)}(s, \cdot)}_H ds\\
& \leq \left(\int_0^1 \|K_H(\cdot, \varphi^{(n)}(s,x)) - K_H(\cdot, \varphi(s,x))\|_H^2ds \right)^{1/2} \|\zeta^{(n)}\|_{L^2([0,1], H)} \\
& = \left(\int_0^1 (K_H(\varphi^{(n)}(s,x), \varphi^{(n)}(s,x)) - 2K_H(\varphi^{(n)}(s,x),\varphi(s,x)) + K_H(\varphi(s,x), \varphi(s,x)))^2ds \right)^{1/2}\\
&  \qquad \times \|\zeta^{(n)}\|_{L^2([0,1], H)} .
\end{align*}
This last term goes to 0 because $r>d/2$ implies that $K_H$ is continuous. As a consequence, we find that $m(1) = q^{(1)}(\varphi(1,x))$ is still satisfied at the limit, implying that $(v, \zeta, \varphi, m)$ is a solution of \eqref{eq:meta.cost.2}.
The proof for \eqref{eq:meta.cost.3} is exactly the same, since the only difference is that the constraint is enforced on a finite set instead of everywhere.\\

Now, let  $(v^{(n)}, \zeta^{(n)}, \varphi^{(n)}, m^{(n)})$ be a sequence of solutions of Problem \eqref{eq:meta.cost.3} with $x^{(0)} = x^{(0,n)}$. Since \eqref{eq:meta.cost.3} is a relaxation of \eqref{eq:meta.cost.2}, the optimal cost of the former is less than the optimal cost of the latter, implying that $v^{(n)}$ and $\zeta^{(n)}$ (or a subsequence) weakly converge to $v$ and $\zeta$ with pointwise convergence of $\varphi^{(n)}$ and $m^{(n)}$ to $\varphi$ and $m$ as above. Since the sets $x^{(0,n)}$ are nested, the constraint $m(1) = q^{(1)}(\varphi(1,x))$ is satisfied for all $x$ in their union, and therefore everywhere in $\mR^d$ since the union is dense. Finally, since the cost of the limit is no larger than the $\liminf$ of the costs of the sequence, which is itself no larger than the optimal cost of \eqref{eq:meta.cost.2}, we find that $(v,\zeta,\varphi,m)$ is an optimal solution of \eqref{eq:meta.cost.2}.
\end{proof}

The existence of solutions for the continuous problem \eqref{eq:meta.cost.2} is in fact true as soon as $r\geq 0$. Indeed, one can write
\[
 \int_0^1 \|\zeta(t)\|^2_H dt =  \int_0^1 \lform{\bL_{\varphi^v(t)^{-1}} \dot m}{\dot m} dt
\]
since $\dot m(t) = \zeta\circ \varphi^v(t)$, from which it results that the optimal $m$ at fixed $v$ is such that $\bL_{\varphi^v(t)^{-1}} \dot m $ remains constant over time. Letting $\sigma^{2}\rho_m\in H^*$ denote this constant value (the normalization by $\sigma^2$ ensures that $\rho_m$ coincides with the one introduced in \eqref{eq:pmp.2}), we get 
\[
m(t) - q^{(0)} = \left(\sigma^2\int_0^t \bL_{\varphi^v(t)^{-1}}^{-1} dt\right) \rho_m
\]
and using $m(1) = \tilde \bT_{\varphi^v(1)} q^{(1)}$, we get
\[
\sigma^2\rho_m = \bR_{\varphi^v}^{-1} (\tilde \bT_{\varphi^v(1)} q^{(1)} - q^{(0)})
\]
so that 
\[
 \int_0^1 \|\zeta(t)\|^2_H dt = \lform{\bR_{\varphi^v}^{-1}(\tilde \bT_{\varphi^v(1)} q^{(1)} - q^{(0)})}{\tilde \bT_{\varphi^v(1)} q^{(1)} - q^{(0)}}.
\]
The optimal $v$ must therefore minimize
\begin{equation}
\label{eq:v.reduc}
\frac12 \int_0^1 \|v(t)\|_V^2 dt + \frac1{2\sigma^2}\lform{\bR_{\varphi^v}^{-1}(\tilde \bT_{\varphi^v(1)} q^{(1)} - q^{(0)})}{\tilde \bT_{\varphi^v(1)} q^{(1)} - q^{(0)}}
\end{equation}
and an argument using minimizing sequences combined with the continuity of $\tilde \bT_\psi$ and $\bR_\varphi$ leads to the existence of a minimizer (this generalizes the result proved in \cite{ty_2005} in the $L^2$ case). Of course, the discretization in \eqref{eq:meta.cost.3} does not make sense for $r\leq d/2$, unless one replaces point evaluation by some other continuous linear forms on $H$, like evaluation against test functions. This would, however, have less practical interest, since test functions do not evolve in a computationally simple way under the action of diffeomorphisms.

\subsection{Optimality Conditions}

We  pass to the necessary conditions for optimal solutions of \eqref{eq:meta.cost.2}, and now assume that $r>d/2+1$ so that $H$ is embedded in $C^1_0(\mR^d)$. Note that, since \eqref{eq:meta.cost.3} can be reduced to \eqref{eq:meta.cost.4}, which is finite dimensional, its optimality conditions follow from the standard Pontryagin maximum principle.  For the infinite-dimensional case, we have:
\begin{theorem}
\label{th:opt.sol}
Assume that both $q^{(1)}$ and $q^{(0)}$ belong to $H^{(r+1)}(\mR^d)$. Then, if $(v, \zeta, \varphi, m)$ is an  optimal solution of 
\eqref{eq:meta.cost.2}, there exist $\rho_\varphi\in (\cB^p)^*$ and $\rho_m\in H^*$ such that \eqref{eq:pmp.2} is satisfied, with 
\[
\lform{\rho_\varphi(t)}{w} + \lform{\rho_m}{\nabla q^{(1)} \circ \varphi(1)\cdot w} = 0
\]
for all $w \in \cB^p$.
\end{theorem}

\begin{proof}

Let $(v, \zeta, \varphi, m)$ be an optimal solution and let $\rho_m = \bR_{\varphi}^{-1} (\tilde \bT_{\varphi(1)} q^{(1)} - q^{(0)})$. 
As remarked at the end of the previous section, the optimal $\zeta$ with fixed $v$ is given by
\begin{equation}
\label{eq:opt.zeta}
\zeta(t) =  \tilde \bT_{\varphi(t)^{-1}}  \bL_{\varphi^v(t)^{-1}}^{-1}\rho_m = \bK_H \tilde \bT_{\varphi(t)}^* \rho_m,
\end{equation}
which is consistent with \eqref{eq:pmp.2}.\\

We now consider the optimal $v$ when $\zeta$ is given by \eqref{eq:opt.zeta}, which minimizes
\begin{multline}
\label{eq:v.reduc.2}
\frac12 \|v\|_{L^2([0,1], V)}^2 + \frac1{2\sigma^2} \lform{\bR_{\varphi}^{-1}(\tilde \bT_{\varphi(1)} q^{(1)} - q^{(0)})}{\tilde \bT_{\varphi(1)} q^{(1)} - q^{(0)}}\\
\text{ subject to } \dot \varphi(t) = v(t) \circ \varphi(t).
\end{multline}

If $z \in H^{r+1}(\mR^d)$ the mapping $\varphi \mapsto \lform{\bR_\varphi^{-1} z}{z}$ is differentiable with differential
\[
\partial_\varphi \Big(\lform{\bR_\varphi^{-1} z}{z}\Big).w = -2 \int_0^1 \lform{\partial_\phi(\tilde \bT_{\varphi(t)}^* \eta).w}{\bK_H \tilde\bT_{\varphi(t)}^*\eta} dt 
\]
with $\eta = \bR_{\varphi}^{-1} z$.

Let $E(v, \varphi)$ denote the minimized term in \eqref{eq:v.reduc.2}. We assume that both $q^{(1)}$ and $q^{(0)}$ belong to $H^{r+1}(\mR^d)$, which implies that $q^{(1)}\circ \varphi(1) - q^{(0)}\in H^{r+1}(\mR^d)$ too.  From the previous discussion and the expression of $\rho_m$, $E$ is differentiable in $\varphi$, with
\begin{eqnarray*}
\partial_\varphi E .  w &=& - \int_0^1 \lform{\partial_\varphi(\tilde \bT_{\varphi(t)}^* \rho_m).w(t)}{\zeta(t)} dt  + \lform{\rho_m}{\nabla q^{(1)} \circ \varphi(1)\cdot w(1)}\\
&=& - \int_0^1 \partial_\varphi\lform{\rho_m}{\tilde \bT_{\varphi(t)}\zeta(t)}.w(t) dt  + \lform{\rho_m}{\nabla q^{(1)} \circ \varphi(1)\cdot w(1)}
\end{eqnarray*}

Define $\mu(t) = \partial_\psi \lform{\rho_m}{\tilde T_{\psi}\zeta(t)}_{|_{\psi = \varphi(t)}}$. The derivative exists, since $\rho_m\in H^{-r}(\mR^d)$ and $\zeta\in H^{r+1}(\mR^d)$, and provides a  a bounded linear form on $C^r(\mR^d, \mR^d)$. Define also the form $\nu: w\mapsto \lform{\rho_m}{(\nabla q^{(1)} \circ \varphi(1))\cdot w}$, which is also bounded on $C^r(\mR^d, \mR^d)$.  Define $\rho_\varphi(t)$ as the solution of the ODE
\[
\dot \rho_\varphi = -\partial_\psi \lform{\rho_\varphi}{T_{\varphi(t)}v(t)}_{|_{\psi = \varphi(t)}} -\mu(t)
\]
with $\rho_\varphi(1) = -\nu$ (this ODE is the third equation in \eqref{eq:pmp.2}). To see that this solution is well defined, first note that, for any given $\rho \in (\cB^{p-1})^*$ and $w\in V$, the mapping $\psi \mapsto \lform{\rho}{w\circ \psi}$, defined on $\cB^p$ is differentiable in $\psi$, with differential
\[
\partial_\psi\lform{\rho}{w\circ \psi}\cdot\delta\psi = \lform{\rho}{Dw\circ\psi\cdot\delta\psi}.
\] 
As a consequence, we have 
\[
\partial_\psi\lform{\rho}{w\circ \psi}\in C^{p-1}(\mR^d, \mR^d)^*,
\]
with norm bounded by $\text{cst}.\|\rho\|_{p-1,\infty,*} \|w\|_{p,\infty} \, \|\psi\|_{p-1, \infty}$ .
The map $Q_{w,\psi} : \rho \mapsto \partial_\psi\lform{\rho}{w\circ \psi}$ therefore is a bounded linear map on $C^{p-1}(\mR^d, \mR^d)^*$, satisfying
\[
\int_0^1 \|Q_{v(t), \varphi(t)}\|^2 dt < \infty
\]
as soon as $\int_0^1 \|v(t)\|_{p,\infty}^2 dt <\infty$, which is true for a minimizer of \eqref{eq:v.reduc.2}. Since both $\mu(t)$ and $\nu$ belong to $C_0^r(\mR^d, \mR^d)^*\subset C_0^{p-1}(\mR^d, \mR^d)^*$ (since $p\geq r+1$), the solution $\rho_\varphi$ of $\dot\rho_\varphi = Q_{v,\varphi} \rho_{\varphi} + \mu$ initialized at $\rho_\varphi (1) = -\nu$ is uniquely defined over $[0,1]$. \\

If $\delta v \in L^2([0,1], V)$,  the directional derivative $\delta \varphi := \partial_v \varphi^v . \delta v$ satisfies (since $\varphi = \varphi^v$)
\[
\partial_t \delta \varphi(t) = \partial_{\varphi}( T_{\varphi(t)}v(t)) \cdot \delta \varphi(t) +\delta v(t) \circ \varphi(t)
\]
with $\delta\varphi(0) = 0$. From the definition of $\rho_\varphi$, we have
\[
\partial_t \lform{\rho_\varphi}{\delta \varphi} = - \lform{\mu}{\delta\varphi} + \lform{\rho_\varphi}{\delta v\circ \varphi}
\]
so that
\[
-\int_0^1 \lform{\mu(t)}{\delta\varphi(t)}dt + \lform{\nu(1)}{\delta\varphi(1)} = - \int_0^1 \lform{\rho_\varphi(t)}{\delta v(t) \circ \varphi(t)}dt.
\]

If $v$ is an optimal solution of \eqref{eq:meta.cost.2}, we must have
\[
\int_0^1 \lform{\bA_V v(t)}{\delta v(t)} dt -  \int_0^1 \lform{\rho_\varphi(t)}{\delta v(t) \circ \varphi(t)}dt = 0
\]
for all $\delta v$, which implies that
\[
v(t) = \bK_V \bT_\varphi^* \rho_\varphi(t).
\]
This is the fifth equation in \eqref{eq:pmp.2}, and completes the proof of Theorem \ref{th:opt.sol}. 
\end{proof}

\subsection{Existence of Solutions of the Initial-Value Problem}

We now discuss the existence and uniqueness of solutions of  \eqref{eq:pmp.2} with initial conditions $\varphi(0) = \mathrm{id}$, $m(0) = m_0$, $\rho_{\varphi}(0) = \rho_{\varphi,0}$ and $\rho_{m}(0) = \rho_{m,0}$. We will assume that $\rho_{\varphi,0} \in (\cB^{p-2})^*$ and $\rho_{m,0} \in H^{1-r}(\mR^d)$ with $p\geq r+1$.

Since $\rho_m$ is constant and $m$ is obtained via quadrature given $\zeta$ and $\varphi$, we will focus on the subsystem
\begin{equation}
\label{eq:pmp.2.red}
\begin{dcases}
\dot \varphi(t) = v(t) \circ \varphi(t)\\
\dot \rho_\varphi(t) = - \partial_{\varphi(t)}\lform{\rho_\varphi(t)}{v(t)\circ \varphi(t)} - \partial_{\varphi(t)} \lform{\rho_m}{\zeta(t)\circ \varphi(t)}\\
v(t) = \bK_V \bT_{\varphi(t)}^* \rho_\varphi(t)\\
\zeta(t) = \sigma^2 \bK_H \tilde \bT_{\varphi(t)}^* \rho_m
\end{dcases}
\end{equation}

If $\rho\in (\cB^{p-2})^*$ and $w\in V\subset \cB^p$, the mapping $\psi \mapsto \lform{\rho}{w\circ \psi} = \lform{\rho}{\bT_\psi w}$ is differentiable in $\psi\in \Diff^p$ with $\partial_\psi\lform{\rho}{w\circ \psi}.h = \lform{\rho}{Dw\circ \psi.h}$.  One deduces from this that $\rho\mapsto \partial_\psi\lform{\rho}{w\circ \psi}$ is a bounded endomorphism of $(\cB^{p-2})^*$ with operator norm bounded by $c(\|\psi-\id\|_{p-2, \infty}) \|w\|_{p-1,\infty}$. If $\rho_m\in H^{1-r}$ and $\zeta\in H^r(\mR^d)$, we have $\partial_\psi \lform{\rho_m}{\zeta\circ \psi}.h = \lform{\rho_m}{\nabla \zeta\circ \psi \cdot h}$.

From the expressions of $v$ and $\zeta$, one easily checks that 
\[
\partial_\psi\lform{\rho_\varphi}{v\circ \psi} = \frac12 \partial_\varphi\lform{\rho_\varphi}{\bT_\psi\bK_V\bT_\psi^*\rho_\varphi} \text{ and }
\partial_\psi \lform{\rho_m}{\zeta\circ \psi} = \frac12 \partial_\psi\lform{\rho_m}{\tilde \bT_\psi \bK_H \tilde \bT^*_\psi \rho_m}.
\]
One also has
$\|v\|_V^2 = \lform{\rho_\varphi}{v\circ\varphi}$ and $\|\zeta\|_H^2 = \sigma^2 \lform{\rho_m}{\zeta\circ \varphi}$, from which one deduces that, along any solution of \eqref{eq:pmp.2.red}, one has
\[
\partial_t\Big(\|v\|_V^2 + \|\zeta\|_H^2/\sigma^2\Big) = 0
\]
since this time derivative is equal to 
\[
\lform{\partial_t \rho_\varphi}{v\circ\varphi} + \partial_\varphi\lform{\rho_\varphi}{v\circ \varphi}.\partial_t\varphi + \partial_\varphi \lform{\rho_m}{\zeta\circ \varphi}.\partial_t \varphi 
\]
which vanishes since $\partial_t\varphi = v\circ \varphi$. This implies, in particular, that 
$v\in L^2([0,t], V)$ and $\zeta\in L^2([0,t], H)$ along any solution of \eqref{eq:pmp.2.red} on the interval $[0,t]$.

Conversely, as soon as $v\in L^2([0,t], V)$ and $\zeta\in L^2([0,t], H)$, the equation
\begin{equation}
\label{eq:ode.rho}
\dot \rho_\varphi = - \partial_\varphi\lform{\rho_\varphi}{v\circ \varphi} - \partial_\varphi \lform{\rho_m}{\zeta\circ \varphi}
\end{equation}
is  a well-defined linear equation on $(\cB^{p-2})^*$, with a unique solution, since we assume $\rho_{\varphi,0} \in (\cB^{p-2})^*$. Its solution can be made explicit by noting that
\begin{eqnarray*}
\partial_t \lform{\rho_{\varphi}}{D\varphi\, w} &=&  - \lform{\rho_\varphi}{Dv\circ \varphi\, D\varphi\, w} - \lform{\rho_m}{\nabla \zeta \circ \varphi \cdot D \varphi\, w} + \lform{\rho_\varphi}{Dv\circ \varphi\, D\varphi\, w}\\ &=& - \lform{\rho_m}{\nabla (\zeta \circ \varphi) \cdot w},
\end{eqnarray*}
from which we conclude that
\[
\lform{\rho_{\varphi}(t)}{D\varphi(t) w} = \lform{\rho_{\varphi,0}}{w} - \int_0^t \lform{\rho_m}{\nabla (\zeta(s) \circ \varphi(s)) \cdot w} ds.
\]
Given this, we can summarize system \eqref{eq:pmp.2} with a single consistency equation for $v$, namely, for all $w\in V$:
\begin{equation}
\label{eq:v.consist}
\lform{\bA_V v(t)}{w} = \lform{\rho_{\varphi,0}}{ \Ad_{\varphi^v(t)^{-1}} w} -  \sigma^2 \int_0^t \lform{\rho_m}{\nabla (\bL_{\varphi^v(s)^{-1}}^{-1}\rho_m) \cdot \Ad_{\varphi^v(t)^{-1}}w} dt,
\end{equation}
in which we have introduced (for $\psi\in \Diff^{p-1}$) the ``adjoint'' operator $\Ad_\psi: w \mapsto (D\psi\, w)\circ \psi^{-1}$, as an operator from $V$ to $\cB^{p-2}$ and used the fact that $\zeta\circ \varphi =  \sigma^2 \tilde \bT_\varphi \bK_H \tilde \bT_\varphi^* \rho_m = \sigma^2 \bL_{\varphi^{-1}}^{-1} \rho_m$. Equation \eqref{eq:v.consist} with $\sigma^2=0$ is of course the well-known momentum conservation equation over diffeomorphisms with a right-invariant metric \cite{arnold1966geometrie,arnold1976methodes,marsden1992lectures,holm1998euler}. \\

Let $\beta^v$ denote the time-dependent linear form applied to $w$ in the right-hand side of \eqref{eq:v.consist}, which therefore can be summarized as $v(t) = \bK_V \beta^v(t)$. Fix a constant $M$. We first check that, for small enough $t$,
$\beta^v(t) \in V^*$ as soon as  $v \in L^2([0,t], V)$ and $\|v\|_{L^2([0,t], V)} \leq M$ implies $\|\beta^v\|_{L^2([0,t], V^*)} = \|\bK_V\beta^v\|_{L^2([0,t],V)} \leq M$ also.

%Indeed, write
%\[
%\lform{\beta^v(t) - \rho_{\varphi,0}}{w} = \lform{\rho_{\varphi,0}}{ \Ad_{\varphi^v(t)^{-1}} w-w} -  \sigma^2 \int_0^t \lform{\rho_m}{\nabla (\bL_{\varphi^v(s)^{-1}}^{-1}\rho_m) \cdot \Ad_{\varphi^v(t)^{-1}}	w} dt.
%\]
We have, for $\psi\in \Diff^p$, $\Ad_{\psi^{-1}} w = ((D\psi)^{-1} - \mathrm{Id}_{\mR^d}) w\circ \psi + w\circ \psi$, from which one gets
\[
\lform{\rho_{\varphi,0}}{ \Ad_{\varphi^v(t)^{-1}} w} \leq c(\|\varphi(t) -\id\|_{p-1, \infty}) \|\rho_{\varphi,0}\|_{p-2, \infty, *} \|w\|_{p-2, \infty}
\]
Since
\[
\|\varphi(t) -\id\|_{p-1, \infty} \leq c(\|v\|_{L^2([0,t], V)}) \|v\|_{L^2([0,t], V)}
\]
we find
\begin{equation}
\label{eq:rho.bound}
\lform{\rho_{\varphi,0}}{ \Ad_{\varphi^v(t)^{-1}} w} \leq c(\|v\|_{L^2([0,t], V)})  \|\rho_{\varphi,0}\|_{p-2, \infty, *}  \|w\|_{p-2, \infty}.
\end{equation}

Since $\bL_{\psi^{-1}}^{-1}$ maps $H^{1-r}(\mR^d)$ onto $H^{r+1}(\mR^d)$ for $\psi\in \Diff^p$, we have 
\[
\|\bL_{\psi^{-1}}^{-1}\rho_m\|_{r+1,2}\leq \text{cst} \|\tilde \bT_\psi\|_{L^2(H^{r+1})}^2 \|\rho_m\|_{1-r, 2} \leq c(\|\psi\|_{r+1,\infty}) \|\rho_m\|_{1-r, 2}.
\]
 Combined with the previous estimate, this yields, for $\psi,\tilde\psi\in \Diff^p$,
\[
\lform{\rho_m}{\nabla (\bL_{\psi^{-1}}^{-1}\rho_m) \cdot \Ad_{\tilde\psi^{-1}} w} \leq c(\|\psi\|_{r+1,\infty},\|\tilde\psi\|_{r-1,\infty}) \|\rho_m\|^2_{1-r, 2}\| \|w\|_{r-2, \infty}.
\]
Since $r\leq p-1$, we can conclude that 
\begin{equation}
\label{eq:betav}
\|\beta^v(t)\|_{p-2, \infty, *} \leq c(\|v\|_{L^2([0,t], V)})(\|\rho_{\varphi,0}\|_{p-2, \infty, *} + t \|\rho_m\|^2_{1-r, 2}),
\end{equation}
from which it follows that $\beta^v(t) \in (\cB^{p-2})^* \subset V^*$ with 
\[
\|\beta^v\|^2_{L^2([0,t], V^*)} \leq t c(\|v\|_{L^2([0,t], V)})(\|\rho_{\varphi,0}\|_{p-2, \infty, *}  + t \|\rho_m\|^2_{1-r, 2})^2.
\]
If we assume that $\|v\|_{L^2([0,t], V)}\leq M$, we  get $\|\beta^v\|_{L^2([0,t], V^*}\leq M$ for $t\leq t_0$, where $t_0$ is chosen such that $t_{0} c(M)(\|\rho_{\varphi,0}\|_{p-2, \infty, *}  + t_{0} \|\rho_m\|^2_{1-r, 2})^2\leq M$. It is important to notice that, beside universal constants and $M$, $t_0$ only depends on $\|\rho_{\varphi,0}\|^2_{p-2, \infty, *}$ and $\|\rho_m\|_{1-r,2}$. In the following, we take $M$ large enough so that any solution of \eqref{eq:pmp.2} must satisfy $\|v\|_{L^2([0,t_0], V)}\leq M$ for any $t_0\leq 1$. This is possible since we have remarked that $\|v(t)\|^2_V + \sigma^{-2}\|\zeta(t)\|^2_H$ remains constant along any solution of \eqref{eq:pmp.2} so that, if $t_0\leq 1$, one must have
\[
\|v\|^2_{L^2([0,t_0], V)} \leq \|v(0)\|^2_V + \sigma^{-2}\|\zeta(0)\|^2_H = \|\rho_{\varphi,0}\|^2_{V^*} + \sigma^2 \|\rho_m\|^2_{H^*}.
\]\\

%We have $\varphi^v\in C([0,t], \Diff^{p})$ so that $\Ad_{\varphi(t)^{-1}}: w \mapsto D\varphi(t)^{-1} w\circ \varphi(t)$ is a bounded linear map from $\cB^{p-2}$ onto itself. This implies that $\Ad_{\varphi(t)^{-1}}^*\rho_{\varphi,0} \in (\cB^{p-2})^*$.   Since we assume $\rho_m\in \cH^{2-r}(\mR^d)$, we have $\nabla (\bL_{\varphi^v(s)^{-1}}^{-1}\rho_m)\in \cH^{r+1}(\mR^d)$, so that
%\[
%w\mapsto \nabla (\bL_{\varphi^v(s)^{-1}}^{-1}\rho_m) \cdot \Ad_{\varphi^v(t)^{-1}} w
%\]
%continuously maps $\cB^{p-2}$ into $\cH^{r-1}(\mR^d)$ since we assume that $p \geq r+1$. This is more than needed to ensure that
%\[
%w\mapsto \lform{\rho_m}{\nabla (\bL_{\varphi^v(s)^{-1}}^{-1}\rho_m) \cdot \Ad_{\varphi^v(t)^{-1}}	w},
%\]
%and therefore $\beta^v$,	belong to $(\cB^{p-2})^*$. Since $(\cB^{p-1})^* \subset (\cB^{p})^* \subset V^*$, we see that $\bK_V\beta^v \in V$ at all times as soon as $v \in L^2([0,t], V)$.  The norms of the linear operators that intervened in these operations are all bounded by continuous functions of $\|\varphi\|_{p, \infty}$, which can be itself controlled by the norm of $v$ in $L^2([0,t], V)$. Using this, one concludes that $v \mapsto K_V\beta^v$ maps  $L^2([0,t], V)$ into itself.

We now estimate the Lipschitz constant of $v\mapsto \beta^v$ on the ball of radius $M$ of $L^2([0,t], V)$ for $t\leq t_0$. In the computations that follow, we will use repetitively the fact that $\|\varphi^v(t) - \id\|_{l, \infty} \leq c(M)$ for any $l\leq p$, as soon as $\|v\|_{L^2([0,t], V)}\leq M$ (recall that $c$ is a notation for a generic continuous function).  Recall also that $p\geq r+1$. 
Writing, assuming $\max(\|v\|_{L^2([0, t_0], V)},\|\tilde v\|_{L^2([0, t_0], V)}) \leq M$,
\[
\Ad_{\varphi^v(t)^{-1}} w - \Ad_{\varphi^{\tilde v}(t)^{-1}} w = (D\varphi^v(t)^{-1} - D\varphi^{\tilde v}(t)^{-1}) w\circ  \varphi^v(t) + D\varphi^{\tilde v}(t)^{-1} (w\circ \varphi^v(t) - w\circ \varphi^{\tilde v}(t))
\]
and using Lemma  \ref{lem:diff.comp} and Leibnitz formula, we get
\begin{equation}
\label{eq:ad.lip}
\|\Ad_{\varphi^v(t)^{-1}} w - \Ad_{\varphi^{\tilde v}(t)^{-1}} w\|_{p-2, \infty} \leq c(M) \|\varphi^v(t)-\varphi^{\tilde v}(t)\|_{p-1, \infty} \|w\|_{p-1, \infty}
\end{equation}
as soon as $\psi, \tilde\psi\in \Diff^{p-1}$ and $w\in \cB^{p-1}$. This immediately implies 
\[
\|(\Ad_{\varphi^v(t)^{-1}}^* - \Ad_{\varphi(t)^{-1}}^*)\rho_{\varphi,0}\|_{p-1, \infty,*} \leq c(M) \|\rho_{\varphi,0}\|_{p-2, \infty,*}\|\varphi^v(t)-\varphi^{\tilde v}(t)\|_{p-1, \infty}.
\]

Write
\begin{eqnarray}
\nonumber
&&\lform{\rho_m}{\nabla (\bL_{\varphi^v(s)^{-1}}^{-1}\rho_m) \cdot \Ad_{\varphi^v(t)^{-1}}w - \nabla (\bL_{\tilde \varphi^v(s)^{-1}}^{-1}\rho_m) \cdot \Ad_{\tilde \varphi^v(t)^{-1}}w} \\
\nonumber
&&\quad = \lform{\rho_m}{(\nabla (\bL_{\varphi^v(s)^{-1}}^{-1}\rho_m) - \nabla (\bL_{\tilde \varphi^v(s)^{-1}}^{-1}\rho_m))\cdot \Ad_{\varphi^v(t)^{-1}}w}\\
\nonumber
&& \qquad + \lform{\rho_m}{ \nabla (\bL_{\tilde \varphi^v(s)^{-1}}^{-1}\rho_m) \cdot (\Ad_{\varphi^v(t)^{-1}}w - \Ad_{\tilde \varphi^v(t)^{-1}}w)}\\
\label{eq:ivp.4} 
&&\quad \leq \text{cst}\,\|\rho_m\|_{1-r,2} \|\bL_{\varphi^v(s)^{-1}}^{-1}\rho_m - \bL_{\tilde \varphi^v(s)^{-1}}^{-1}\rho_m\|_{r,2} \|\Ad_{\varphi^v(t)^{-1}}w\|_{r-1, \infty}\\
\nonumber
&& \qquad + \text{cst}\,\|\rho_m\|_{1-r,2} \|\bL_{\tilde \varphi^v(s)^{-1}}^{-1}\rho_m\|_{r, 2} \|\Ad_{\varphi^v(t)^{-1}}w - \Ad_{\tilde \varphi^v(t)^{-1}}w\|_{r-1, \infty} 
\end{eqnarray}
We have (letting $\psi = \varphi^v(s)$ and $\tilde\psi = \varphi^{\tilde v}(s)$)
\begin{eqnarray}
\nonumber
\|\bL_{\psi^{-1}}^{-1} \rho_m - \bL_{\tilde \psi^{-1}}^{-1} \rho_m\|_{r,2} &=& \|\tilde\bT_\psi\bK_H \tilde\bT_\psi^*\rho_m - \tilde\bT_{\tilde\psi} \bK_H \tilde\bT_{\tilde\psi}^*\rho_m\|_{r,2}\\
\label{eq:ivp.2} 
&\leq& \|\tilde\bT_\psi\|_{r,2} \|\bK_H \tilde\bT_\psi^*\rho_m - \bK_H \tilde\bT_{\tilde\psi}^*\rho_m\|_{r,2} \\
\nonumber
&&
+ \|\tilde\bT_\psi\bK_H \tilde\bT_{\tilde\psi}^*\rho_m - \tilde\bT_{\tilde\psi} \bK_H \tilde\bT_{\tilde\psi}^*\rho_m\|_{r,2}
\end{eqnarray}

Let us consider the last two terms separately. We have $\|\tilde \bT_\psi\|_{r,2} = c(\|\psi-\mathrm{id}\|_{r,\infty}) \leq c(M)$. Also, 
\[
\|\bK_H \tilde\bT_\psi^*\rho_m - \bK_H \tilde\bT_{\tilde\psi}^*\rho_m\|_{r,2} \leq \text{cst} \|\tilde\bT_\psi^*\rho_m - \tilde\bT_{\tilde\psi}^*\rho_m\|_{-r,2} = \text{cst}\,\sup\Big(\lform{\rho_m}{\tilde\bT_\psi z- \tilde\bT_{\tilde\psi} z}: \|z\|_{r,2}\leq 1\Big)
\]
and we have 
\[
\lform{\rho_m}{\tilde\bT_\psi z- \tilde\bT_{\tilde\psi} z} \leq \|\rho_m\|_{1-r,2} \|\tilde\bT_\psi z- \tilde\bT_{\tilde\psi} z\|_{r-1,2} \leq c(M) \|\rho_m\|_{1-r,2} \|\psi - \tilde\psi \|_{r-1,\infty} \|z\|_{r,2}
\]
so that
\begin{equation}
\label{eq:ivp.1}
\|\tilde\bT_\psi\|_{r,2} \|\bK_H \tilde\bT_\psi^*\rho_m - \bK_H \tilde\bT_{\tilde\psi}^*\rho_m\|_{r,2} \leq c(M) \|\rho_m\|_{1-r,2} \|\psi - \tilde\psi \|_{r-1,\infty}.
\end{equation}
For the second term in \eqref{eq:ivp.2}, write
\begin{eqnarray*}
\|\tilde\bT_\psi\bK_H \tilde\bT_{\tilde\psi}^*\rho_m - \tilde\bT_{\tilde\psi} \bK_H \tilde\bT_{\tilde\psi}^*\rho_m\|_{r,2} &\leq& c(M) \|\psi - \tilde\psi \|_{r,\infty} \|\bK_H \tilde\bT_{\tilde\psi}^*\rho_m \|_{r+1,2} \\
&\leq&  c(M) \|\psi - \tilde\psi \|_{r,\infty} \|\tilde\bT_{\tilde\psi}^*\rho_m\|_{1-r,2}\\
&\leq&  c(M) \|\psi - \tilde\psi \|_{r,\infty} \|\rho_m\|_{1-r,2}
\end{eqnarray*}
From this and \eqref{eq:ivp.2}, \eqref{eq:ivp.1}, we get
\begin{equation}
\label{eq:ivp.3}
\|\bL_{\psi^{-1}}^{-1} \rho_m - \bL_{\tilde \psi^{-1}}^{-1} \rho_m\|_{r,2} \leq c(M) \|\psi - \tilde\psi \|_{r,\infty} \|\rho_m\|_{1-r,2}
\end{equation}
Since
\[
\|\Ad_{\varphi^v(t)^{-1}}w\|_{r-1, \infty} \leq c(\|\varphi^v(t) -\id\|_{r,\infty})\|w\|_{r-1, \infty}
\]
we find that the first term in \eqref{eq:ivp.4} is less than  $c(M) \|\psi - \tilde\psi \|_{r,\infty} \|\rho_m\|^2_{1-r,2}\|w\|_{r-1, \infty}$. 

For the second term in \eqref{eq:ivp.4}, we have $\|\bL_{\tilde \varphi^v(s)^{-1}}^{-1}\rho_m\|_{r, 2}\leq c(M) \|\rho_m\|_{1-r,2}$ while, similarly to \eqref{eq:ad.lip},
\[
\|\Ad_{\varphi^v(t)^{-1}}w - \Ad_{\tilde \varphi^v(t)^{-1}}w\|_{r-1, \infty} \leq c(M) \|\varphi^v(t) - \varphi^{\tilde v}(t)\|_{r, \infty} \|w\|_{r, \infty}.
\]
This finally gives the upper-bound
\begin{multline*}
\lform{\rho_m}{\nabla (\bL_{\varphi^v(s)^{-1}}^{-1}\rho_m) \cdot \Ad_{\varphi^v(t)^{-1}}w - \nabla (\bL_{\tilde \varphi^v(s)^{-1}}^{-1}\rho_m) \cdot \Ad_{\tilde \varphi^v(t)^{-1}}w} \\
\leq c(M) \|\rho_m\|_{1-r,2}^2 \|w\|_{r, \infty} \sup_{s\leq t}\|\varphi^v(s) - \varphi^{\tilde v}(s)\|_{r, \infty}
\end{multline*}•
so that (using $r\leq p-1$)
\[
\|\beta^v(t) - \beta^{\tilde v}(t)\|_{p-1, \infty, *} \leq c(M) \big(\|\rho_{\varphi,0}\|_{p-2, \infty,*} + t\|\rho_m\|_{1-r,2}^2\big) \sup_{s\leq t} \|\varphi^v(s)-\varphi^{\tilde v}(s)\|_{p-1, \infty}.
\]
Using the fact that 
\[
\sup_{s\leq t_0} \|\varphi^v(s)-\varphi^{\tilde v}(s)\|_{p-1, \infty} \leq c(M) \|v-\tilde v\|_{L^2([0,t_0], V)}
\]
we find that $v \mapsto \bK_V\beta^v$ is Lipschitz on the ball of radius $M$ in $L^2([0,t_0], V)$, with Lipschitz constant less than
$c(M) t_0\big(\|\rho_{\varphi,0}\|_{p-2, \infty,*} + t_0\|\rho_m\|_{1-r,2}^2\big)$.

Reducing the value of $t_0$ if needed, one can make this upper-bound less than 1 to ensure  that $v \mapsto \bK_V \beta^v$ has a unique fixed point in the ball of radius $M$ in $L^2([0,t_0], V)$. This shows that system \eqref{eq:pmp.2.red} has a unique solution (with the considered initial condition) over the interval $[0, t_0]$.\\

A valid choice for $t_0$ can therefore be made in terms of $M$, $\|\rho_{\varphi,0}\|_{p-2, \infty,*}$, and  $\|\rho_m\|_{1-r,2}$ uniquely; since $M$ can itself be chosen as a function of the last two norms, their values are sufficient to specify $t_0$. If we now define $T_0$ to be the largest time $T_0\leq 1$ such that a solution exists over all intervals $[0,t]\subset [0, T_0)$, we must have $T_0 = 1$ unless $\|\rho_{\varphi}(t)\|_{p-2, \infty,*}$ tends to $\infty$ when $t$ tends to $T_0$ (recall that $\rho_m$ is time-independent). Since $\rho_{\varphi}(t) = \bT_{\varphi(t)^{-1}}^* \bA_V v(t) = \bT_{\varphi(t)^{-1}}\beta^v(t)$, equation \eqref{eq:betav} shows that $\|\rho_{\varphi}(t)\|_{p-2, \infty,*}$ must remain bounded, showing that $T_0 = 1$ necessarily.

Since one can obviously replace the unit interval by any interval $[0,T]$, we have obtained the following result.

\begin{theorem}
\label{th:ivp}
Assume that $p\geq 1+d/2$ and $p\geq r+1$. Then system \eqref{eq:pmp.2} has a unique solution over any bounded interval as soon as $\rho_{\varphi,0}\in (\cB^{p-2})^*$ and $\rho_m\in H^{1-r}(\mR^d)$.
\end{theorem}
Note that, with metamorphosis, the boundary condition requires that $\lform{\rho_{\varphi,0}}{w} = \lform{\rho_m}{\nabla q_0\cdot w}$. Assuming that $q_0 \in H^1(\mR^d)$ (which is restrictive only for $r=0$), we see that $\rho_m\in H^{1-r}(\mR^d)$ implies that $\rho_{\varphi,0} \in (\cB^{r-1})^* \subset (\cB^{p-2})^*$ since $p\geq r+1$, so that the regularity condition for $\rho_{\varphi,0}$ is automatically satisfied.\\

\noindent {\textbf{Remark}} In the previous result, we ``lose'' two derivatives in the initial condition for $\rho_\varphi$ and one in $\rho_m$. This can be improved under more restrictive assumptions on the spaces $V$ and $H$.
\begin{itemize}
\item Assume that the norm on $H$ is specified by a differential operator. We have seen that $\psi\mapsto \lform{\rho_m}{\bL^{-1}_{\psi^{-1}}\rho_m}$ was a smooth function of $\psi\in \Diff^p$ as soon as $\rho_m\in H$, with 
\[
\lform{\rho_m}{\nabla(\bL^{-1}_{\psi^{-1}}\rho_m)\cdot w} = (1/2)\partial_\psi \lform{\rho_m}{\bL^{-1}_{\psi^{-1}}\rho_m}. w.
\]
Using this property, one can carry on the estimates on the second term in $\beta^v$ using only the assumption $\rho_m\in H^{-r}$, and therefore extend the conclusion of the theorem to this case.
\item If one makes the same hypothesis   for $V$, namely that $V\sim H^p(\mR^d, \mR^d)$ with $p\geq 1+d/2$ and $p\geq r$ (note that this assumption only implies that $V$ is embedded in $\cB^1$), the associated group $\Diff_V$  is then included in the Hilbert manifold $\mathcal D^p$ of diffeomorphisms $\psi$ such that $\psi - \id$ and $\psi^{-1}-\id$ both belong to $H^{p}(\mR^d, \mR^d)$, on which the right invariant metric is a strong Riemannian metric (i.e., the Riemannian topology coincides with the one induced by $H^p(\mR^d, \mR^d)$). This is a consequence of Lemma 1 and of results on the stability of Sobolev spaces by products which implies that all terms $(\partial_{\gamma_1} \psi_{j_1})^{\ell_1}\cdots (\partial_{\gamma_k} \psi_{j_k})^{\ell_k}\partial_\beta z$ are square integrable as soon as $|\beta| + \sum_{q=1}^k \ell_q(|\gamma_q|-1)\leq p$ \cite{palais1968foundations}. (It has actually recently been showed that $\mathcal D^{p}$ coincides with $\Diff_{V}$; see \cite{bruveris2014completeness}.)  The right-invariant metric
\[
\|(\xi, z)\|_{(\psi,q)}^2 = \|\xi\circ \psi^{-1}\|_V^2 + \sigma^{-1} \|z\circ \psi^{-1}\|_{r,2}^2
\]  
on the  product space $\Diff_V\times H^r(\mR^d)$ is then also a strong metric as soon as $r\leq p$, and since \eqref{eq:pmp.2} is the geodesic equation on this manifold,  its solutions are uniquely defined over arbitrary time intervals without loss of derivatives (see \cite{lang1962introduction,abraham1963lectures,ebin1970groups,ebin1970manifold,misiolek2010fredholm,inci2013regularity}, and the references therein, for more details).
\end{itemize}

\section{Discussion}
\label{sec:conclusion}

In this paper we developed new numerical tools, combined with an extension of known theoretical results, on image metamorphosis. We proposed, in particular, a particle-based optimization method for their estimation, based on the determination of initial conditions of the geodesic equation performed via a shooting method. The resulting algorithm allows for a numerically-stable sparse representation of the target image in a template-centered coordinate system, which was  hard to achieve using previous methods. This improvement was made possible by the introduction of a Sobolev norm in image space, allowing for particle solutions that were not available when using an $L^2$ norm.

One of the limitations of the discretization scheme discussed in section \ref{sec:singular} is its asymmetry, since the evolving image is represented using a moving grid, $x$, which is specified at time $t=0$, in the template coordinate frame (the continuous problem itself is symmetric, so that the asymmetry disappears in the discretization limit). Our scheme  can, however, be modified to incorporate more symmetry by introducing a second set of particles, this time defined in the target coordinate frame.   More precisely, one can add to \eqref{eq:meta.cost.3} another set of constraints, associated to a new grid $y$ and image value $n$ (in addition to $x$ and $m$) in the form $\dot y_k = v(t, y_k)$, $\dot n_k = \zeta(t, y_k)$, $n_k(0) = q^{(0)}(y_k(0))$, $n_k(1) = q^{(1)}(y_k^{(1)})$. The optimality equations are similar to those derived in \eqref{eq:disc.syst} (the states are simply extended from $x$ to $(x,y)$ and from $m$ to $(m,n)$, with extended control variables $z$ and $\alpha$). The shooting algorithm must then be parametrized by the initial controls, as described in this paper, but also by the initial position of the $y$ variables, with a new objective function  
\[
E = \sum_{k=1}^N (m_k(1) - q^{(1)}(x_k(1)))^2 + \sum_{k=1}^N (n_k(1) - q^{(1)}(y_k^{(1)}))^2 + \sum_{k=1}^N |y_k(1) - y_k^{(1)}|^2.
\]
This symmetrized discretization scheme can be addressed along the same lines as the one studies in the present paper.

\section{Appendix}
\subsection{Forward and Adjoint Systems}
We here provide more details on the implementation of the adjoint method described in Section \ref{sec:solution}. We assume, in the following, that $K_V$ is a scalar multiple of the identity matrix, and 
taking variations of \eqref{eq:disc.syst} in the discrete variables $x_k, z_k, m_k$ yields a forward system of equations 
that evolves these variations (note, to keep the equations below compact, we do not write the explicit evaluation of $K_{V}$ and $K_H$ at $x_k, x_\ell$):
\begin{align*}
  \partial_t (\delta x_k) = & \frac{1}{\sigma^2} \sum_{\ell=1}^{N} \left( (\nabla_1 K_{V})\cdot \delta x_k z_\ell(t)  
  + ( \nabla_2 K_{V} )\cdot \delta x_\ell z_\ell(t)  + K_{V}\delta z_\ell(t) \right), \\
  \partial_t (\delta z_k) = & - \frac{1}{\sigma^2} \sum_{\ell=1}^{N} \big( z_\ell(t)\cdot z_k(t) (D^2_{11} K_{V})^T \delta x_k  +  
  z_\ell(t)\cdot z_k(t) (D^2_{12} K_{V})^T \delta x_\ell   \\ 
  \nonumber & \qquad\qquad  + \nabla_1 K_{V} z_k(t)\cdot \delta z_\ell(t) + \nabla_1 K_{V} z_\ell(t)\cdot \delta z_k(t) \big) \\
  \nonumber
  &- \sum_{\ell=1}^{N} \left( \alpha_k \alpha_\ell (D^2_{11} K_H)^T \delta x_k - 
  \alpha_k \alpha_\ell (D^2_{12} K_H)^T \delta x_\ell - \nabla_1 K_H \alpha_k \delta \alpha_\ell - \nabla_1 K_H \alpha_\ell \delta \alpha_k \right), \\
  \partial_t (\delta m_k)  = & \sum_{\ell=1}^N \left( \alpha_\ell (\nabla_1 K_H)\cdot \delta x_k + \alpha_\ell (\nabla_2 K_H)\cdot \delta x_\ell 
  +  K_H \delta \alpha_\ell \right), 
\end{align*}
where $\delta x_k(t)$ denotes a variation in the value of the position of the node $x_k$ at time $t$ (and analogously for $\delta z_k$, $\delta m_k$), and
$K_V, K_H$ are treated as functions on $\mathbb{R}^d \times \mathbb{R}^d$, and so the subscripts for the gradient and Jacobian denote differentiation with respect to 
the first and second variables $x_k, x_\ell \in \mathbb{R}^d$. 

Let $\xi_x, \xi_z, \xi_m,  \eta_\alpha$ denote dual forms to the variations $\delta x, \delta z, \delta m$, and $\eta_{\alpha}$ the associated variation in $\alpha$, as introduced in \eqref{eq:adj.syst}, which expands as (again without writing the evaluation of the kernel terms, and combining 
the summations for compactness of notation):
\begin{align*}
\left( \partial_t \xi_x \right)_k & = \sum_{\ell=1}^N \Big\{ - \frac{1}{\sigma^2} \big( \nabla_1 K_{V} z_\ell(t)\cdot \xi_{x,k}(t) + 
\nabla_1 K_{V} z_k(t)\cdot \xi_{x,\ell}(t) + z_\ell(t)\cdot z_k(t) D^2_{11} K_{V} \xi_{z,k}(t)  \\
\nonumber &  \qquad\qquad + ~ z_k(t)\cdot z_\ell(t) D^2_{21} K_{V} \xi_{z,\ell}(t) \big)   + \alpha_k\alpha_\ell D^2_{11} K_H \xi_{z,k}(t) + \alpha_\ell \alpha_k D^2_{21} K_H \xi_{z,\ell}(t)   \\
\nonumber & \qquad\qquad - ~  \alpha_\ell \nabla_1 K_H \xi_{m,k}(t) - \alpha_k \nabla_1 K_H \xi_{m,\ell}(t) 
 \Big\}, \\
 \left( \partial_t  \xi_z \right)_k  & = \frac{1}{\sigma^2}  \sum_{\ell=1}^N \Big\{ - K_{V} \xi_{x,\ell}(t) + z_\ell(t) (\nabla_1 K_{V})\cdot \xi_{z,k}(t) +
 z_\ell(t) (\nabla_2 K_{V})\cdot \xi_{z,\ell}(t)  \Big\} \\
 \left( \partial_t \eta_\alpha \right)_k & =   \sum_{\ell=1}^N \Big\{  \alpha_\ell (\nabla_1 K_H)\cdot \xi_{z,k}(t) + \alpha_\ell (\nabla_2 K_H)\cdot \xi_{z,\ell}(t) - K_H \xi_{m,\ell}(t) \Big\} .
\label{eqn.adjoint_system}
\end{align*}
Note that since no other variables depend on $m_k$ in the forward system, the dual variable $\xi_m$ is constant in time, and so we do not display its evolution in the list above.
\bibliographystyle{plain}

%\bibliography{RKHS_Shooting_Meta_RY.bib}

\begin{thebibliography}{10}

\bibitem{abraham1963lectures}
Ralph~H Abraham and Stephen Smale.
\newblock {\em Lectures of Smale on differential topology}.
\newblock Columbia University, Department of Mathematics, 1963.

\bibitem{adams2003sobolev}
Robert~A Adams and John~JF Fournier.
\newblock {\em Sobolev spaces}.
\newblock Academic press, 2003.

\bibitem{agrachev2004control}
Andrei~A Agrachev and Yuri Sachkov.
\newblock {\em Control theory from the geometric viewpoint}.
\newblock Springer, 2004.

\bibitem{arnold1966geometrie}
Vladimir Arnold.
\newblock Sur la g{\'e}om{\'e}trie diff{\'e}rentielle des groupes de lie de
  dimension infinie et ses applications {\`a} l'hydrodynamique des fluides
  parfaits.
\newblock {\em Annales de l'institut Fourier}, 16(1):319--361, 1966.

\bibitem{arnold1976methodes}
Vladimir Arnold.
\newblock {\em Les m{\'e}thodes math{\'e}matiques de la m{\'e}canique
  classique}.
\newblock Editions Mir, 1976.

\bibitem{aronszajn1950theory}
Nachman Aronszajn.
\newblock Theory of reproducing kernels.
\newblock {\em Transactions of the American mathematical society}, pages
  337--404, 1950.

\bibitem{brezis2011functional}
Haim Br{\'e}zis.
\newblock {\em Functional analysis, Sobolev spaces and partial differential
  equations}.
\newblock Springer, 2011.

\bibitem{bruveris2014completeness}
Martins Bruveris and Fran{\c{c}}ois-Xavier Vialard.
\newblock On completeness of groups of diffeomorphisms.
\newblock {\em arXiv preprint arXiv:1403.2089}, 2014.

\bibitem{camionyounes_2001}
Vincent Camion and Laurent Younes.
\newblock Geodesic interpolating splines.
\newblock {\em Energy Minimization Methods in Computer Vision and Pattern
  Recognition, Lecture Notes in Computer Science}, pages 513--527, 2001.

\bibitem{dupuis_miller_1998}
Paul Dupuis, Ulf Grenander, and Michael Miller.
\newblock Variational problems on flows of diffeomorphisms for image matching.
\newblock {\em Quarterly of Applied Math}, 1998.

\bibitem{ebin1970manifold}
D~Ebin.
\newblock The manifold of {R}iemannian metrics.
\newblock In {\em Proc. Symp. AMS}, volume~15, pages 11--40, 1970.

\bibitem{ebin1970groups}
David~G Ebin and Jerrold Marsden.
\newblock Groups of diffeomorphisms and the motion of an incompressible fluid.
\newblock {\em Annals of Mathematics}, pages 102--163, 1970.

\bibitem{paraview}
A.~Henderson.
\newblock {\em ParaView Guide, A Parallel Visualization Application}.
\newblock Kitware, Inc., 2007.

\bibitem{holm1998euler}
Darryl~D Holm, Jerrold~E Marsden, and Tudor~S Ratiu.
\newblock The {E}uler--{P}oincar{\'e} equations and semidirect products with
  applications to continuum theories.
\newblock {\em Advances in Mathematics}, 137(1):1--81, 1998.

\bibitem{hty_2008}
Darryl~D. Holm, Alain Trouv{\'e}, and Laurent Younes.
\newblock The {E}uler-{P}oincare theory of metamorphosis.
\newblock {\em Quart. Appl. Math.}, 67:661--685, 2009.

\bibitem{inci2013regularity}
Hasan Inci, Thomas Kappeler, and Peter Topalov.
\newblock {\em On the regularity of the composition of diffeomorphisms}, volume
  226.
\newblock American Mathematical Soc., 2013.

\bibitem{scipy}
Eric Jones, Travis Oliphant, Pearu Peterson, et~al.
\newblock {SciPy}: Open source scientific tools for {Python}, 2001.

\bibitem{leafsnap}
Neeraj Kumar, Peter~N Belhumeur, Arijit Biswas, David~W Jacobs, W~John Kress,
  Ida~C Lopez, and Jo{\~a}o~VB Soares.
\newblock Leafsnap: A computer vision system for automatic plant species
  identification.
\newblock In {\em Computer Vision--ECCV 2012}, pages 502--516. Springer, 2012.

\bibitem{lang1962introduction}
Serge Lang.
\newblock {\em Introduction to differentiable manifolds}.
\newblock Interscience Publishers, 1962.

\bibitem{marsden1992lectures}
JE~Marsden.
\newblock {\em Lectures on geometric mechanics}.
\newblock Cambridge University Press, 1992.

\bibitem{miller_younes_1}
Michael~I. Miller and Laurent Younes.
\newblock Group action, diffeomorphism and matching: a general framework.
\newblock {\em Int. J. Comp. Vis.}, 41:61--84, 2001.

\bibitem{misiolek2010fredholm}
Gerard Misio{\l}ek and Stephen~C Preston.
\newblock Fredholm properties of riemannian exponential maps on diffeomorphism
  groups.
\newblock {\em Inventiones mathematicae}, 179(1):191--227, 2010.

\bibitem{palais1968foundations}
Richard~S Palais.
\newblock {\em Foundations of global non-linear analysis}, volume 196.
\newblock Benjamin New York, 1968.

\bibitem{ry_2013}
Casey~L Richardson and Laurent Younes.
\newblock Computing metamorphoses between discrete measures.
\newblock {\em Journal of Geometric Mechanics}, 5:131--150, 2013.

\bibitem{ty_2005}
Alain Trouv\'{e} and Laurent Younes.
\newblock Local geometry of deformable templates.
\newblock {\em SIAM Journal on Mathematical Analysis}, 37(1):17--59, 2005.

\bibitem{tangent_space_rep_2004}
Marc Vaillant, Michael~I. Miller, Alain Trouv'e, and Laurent Younes.
\newblock Statistics on diffeomorphisms via tangent space representations.
\newblock {\em Neuroimage}, 23(S1):S161--S169, 2004.

\bibitem{vincent1999nonlinear}
Thomas~L Vincent and Walter~J Grantham.
\newblock {\em Nonlinear and optimal control systems}.
\newblock John Wiley \& Sons, Inc., 1999.

\bibitem{wahba1990spline}
Grace Wahba.
\newblock {\em Spline models for observational data}.
\newblock Siam, 1990.

\bibitem{shape_dat_2007}
Lei Wang, Faisal Beg, Tilak Ratnanather, Can Ceritoglu, Laurent Younes, John~C.
  Morris, John~G. Csernansky, and Michael~I. Miller.
\newblock Large deformation diffeomorphism and momentum based hippocampal shape
  discrimination in dementia of the alzheimer type.
\newblock {\em IEEE Transactions on Medical Imaging}, 26:462--470, 2007.

\bibitem{hipp_volume_2003}
Lei Wang, Jeffrey~S. Swank, Irena~E. Glick, Mokhtar~H. Gado, Michael~I. Miller,
  John~C. Morris, and John~G. Csernansky.
\newblock Large deformation diffeomorphism and momentum based hippocampal shape
  discrimination in dementia of the alzheimer type.
\newblock {\em NeuroImage}, 20:667--682, 2003.

\bibitem{younes_book_2010}
Laurent Younes.
\newblock {\em Shapes and Diffeomorphisms}, volume 171 of {\em Applied
  Mathematical Sciences}.
\newblock Springer, Berlin, 2010.

\end{thebibliography}
\end{document}